\documentclass[12pt]{article}
\usepackage{amssymb,amsmath,epsfig,amscd,eucal,psfrag,amsthm,relsize}
\usepackage{graphicx,hyperref}
\usepackage{tikz}
\usepackage{tikz-cd}
\usetikzlibrary{arrows}
\usepackage{anysize}
\usepackage{indentfirst}
\usepackage{float} 
\usepackage{array}
\usepackage[utf8]{inputenc}
\usepackage{enumerate}
\usepackage{mathdots}

\newtheorem{theorem}{Theorem}[section]

\newtheorem{lemma}[theorem]{Lemma}
\newtheorem{proposition}[theorem]{Proposition}

\newtheorem{corollary}[theorem]{Corollary}

\newtheorem{question}[theorem]{Question}

\theoremstyle{definition}
\newtheorem{definition}[theorem]{Definition}

\theoremstyle{remark}
\newtheorem{remark}[theorem]{Remark}
\newtheorem{example}[theorem]{Example}

\newcommand{\B}{\mathcal{B}}
\newcommand{\C}{\mathcal{C}}
\newcommand{\G}{\mathcal{G}}
\newcommand{\K}{\mathcal{K}}
\newcommand{\F}{\mathbb{F}}
\newcommand{\N}{\mathbb{N}}

\newcommand{\R}{\mathbb{R}}
\newcommand{\Z}{\mathbb{Z}}

\newcommand{\tU}{\widetilde{U}}
\newcommand{\tV}{\widetilde{V}}

\newcommand{\A}{\mathcal{A}}
\newcommand{\curlyC}{\mathcal{C}}

\newcommand{\E}{\mathcal{E}}
\newcommand{\curlyG}{\mathcal{G}}
\newcommand{\curlyF}{\mathcal{F}}

\newcommand{\curlyH}{\mathcal{H}}

\newcommand{\link}{{\mbox{Link}}}

\newcommand{\bigcupdot}{\hspace{9pt}\cdot \hspace{-9pt}\bigcup}

\input xy
\xyoption{all}

\title{ JSJ decompositions of pro-$p$ groups}
\author{ Pavel Zalesskii\footnote{Partially supported by CNPq and FAPDF}}

\usetikzlibrary{shapes,decorations,arrows,calc,arrows.meta,fit,positioning}
\tikzset{
	-Latex,auto,node distance =1 cm and 1 cm,semithick,
	state/.style ={ellipse, draw, minimum width = 0.7 cm},
	point/.style = {circle, draw, inner sep=0.04cm,fill,node contents={}},
	bidirected/.style={Latex-Latex,dashed},
	el/.style = {inner sep=2pt, align=left, sloped}
}

\begin{document}

\maketitle

\begin{abstract}

\end{abstract}

\tableofcontents
\newpage
\section{Introduction}

The dramatic advance  of  classical combinatorial group theory happened
in the 1970's, when  the  Bass-Serre theory of groups acting on trees
changed completely the face of the theory. This theory allowed to treat uniformly basic constructions of combinatorial theory of  groups such as free products with amalgamation and HNN-extensions and obtain structural subgroup results of these constructions.

At the end of 20-th and in the first quarter of 21-st centuries Bass-Serre theory faced further development: the theory of JSJ decompositions of finitely generated groups. In fact JSJ decompositions first appeared in 3-dimensional topology with the theory of the characteristic submanifold by Jaco-Shalen and Johannson \cite{JS79, Joh79} (the terminology JSJ
was popularized by Sela). The topological ideas were carried over to group theory first by Kropholler \cite{K} for some Poincar\'e duality groups. Later constructions of JSJ decompositions were given in various
settings by Sela for torsion-free hyperbolic groups \cite{Sel97b}, and in various settings by Rips-Sela \cite{RS97}, Bowditch \cite{Bow98}, Dunwoody-Sageev \cite{DS99}, Fujiwara-Papasoglu \cite{FP06}, Dunwoody-Swenson \cite{DS00}, Guirardel and Levitt \cite{Gui08},  \cite{GL}, \cite{Gui11} ...  ). This has had a vast influence and range of applications in geometric and combinatorial group theory.
Roughly speaking, a JSJ-splitting of a group $G$ over a certain class of subgroups is a graph of groups decomposition of $G$ which describes all possible decompositions of G as an amalgamated free product or an HNN extension
over subgroups lying in the given class.

\smallskip
The profinite version of the Bass-Serre  theory was initiated by Dion Gildenhuys and Luis Ribes in \cite{GR} and then developed by Luis Ribes,
Oleg Melnikov and the  author because of the absence of the classical
methods of combinatorial group theory for profinite groups (see \cite{ZM-88, ZM-89, Zal90, Z-92, RZ-00, RZ-14, R}. However it
does not work in full strength even in the pro-$p$ case. The reason is
that if a pro-$p$ group $G$ acts on a pro-$p$ tree $T$ then  a maximal subtree of the quotient graph $T/G$ does not always exist and even if it exists it does not always lift to $T$. As a consequence the pro-$p$ version of the  Bass-Serre theory
does not give subgroup structure theorems the way it does in the
classical Bass-Serre theory. In fact,  for infinitely generated pro-$p$ subgroups there are counter examples. Nevertheless, the pro-$p$ version of the subgroups structure theorem works for a finitely generated pro-$p$ group $G$ acting on a pro-$p$ tree $T$ but not the same way as in the classical Bass-Serre theory way. What is proved in \cite{WZ18} that there exists another pro-$p$ $G$-tree  $S$ in the same deformation space  such that $G$ is the fundamental pro-$p$ group of the corresponding graph of groups $(\G, S/G)$ in the standard Bass-Serre theory maner. However, in contrast to the abstract situation  $(\G, S/G)$ can be reduced and infinite, and the topology of $S/G$ usually prevents removing an edge. So in general, the fundamental group of a profinite graph of pro-$p$ groups does not split as an amalgamated free pro-$p$ product or pro-$p$ HNN-extension over an edge group. Nevertheless, if $G$ is finitely generated then such a splitting of $G$ exists  over the edge stabilizer $G_e$ for any edge $e$ of $T$ by \cite[Theorem 4.2]{CZ} (obtained in differen manner). This shows that the pro-$p$ version of Bass-Serre theory is ready for JSJ-decomposition theory and to develop it is the obective of the present paper.

\medskip
In the exposition we follow a general strategy of the nice book on JSJ-decompositions of (abstract) groups by Guirardel and Levitt \cite{GL} stressing similarities and differences between the abstract and pro-$p$ cases. Namely, 
we give a  general definition of JSJ decompositions (or rather of their Bass-Serre pro-$p$ trees), as maximal universally elliptic pro-$p$ trees. As in the abstract case, in general, there is no preferred JSJ
decomposition, and the right object to consider is the whole set of JSJ decompositions: the JSJ deformation space. The importance of the notion of deformation space has complementary value in the pro-$p$ case explained in the preceding paragraph. However, one needs to mention that  the   pro-$p$  deformation space  is usually much biggger than in its abstract analog because the group of automorphisms of pro-$p$ groups usually are much larger: for example the automorphism group $Aut(F)$ of a free pro-$p$ group $F$ of finite rank is topologically infinitely generated in contrast to the abstract case.

We prove that JSJ decompositions exist for any finitely generated pro-$p$ group, without
any assumption on edge groups.

\begin{theorem}\label{existenceint} Let   $\E$ be a continuous family of subgroups of a pro-$p$ group $G$. If $G$ is finitely generated or  $\E$-accessible, then JSJ pro-$p$ $\E$-tree $T$ exists.

\end{theorem}

In the presence of acylindricity, a pro-$p$ group $G$ is  accessible (cf. \cite{CAZ}) as well as for splittings
of  CSA groups over abelian groups (in particular, for splittings of pro-$p$ limit groups). We define  pro-$p$ trees of cylinders and obtain
canonical JSJ pro-$p$ trees which are invariant under automorphisms.

\begin{theorem} Let $G$ be a finitely generated CSA pro-$p$ group. Suppose that $G$ is indecomposable into a free pro-$p$ product and  $\E$ consists of  abelian subgroups. Then the pro-$p$ tree of cylinders $T_c$ is    $Aut(G)$-invariant JSJ $\E$-tree relative to not virtually cyclic abelian subgroups.
\end{theorem}

\begin{corollary} Let $G$ be a pro-$p$ limit group and $\E$ be the set of  abelian subgroups of $G$. Suppose that $G$ is indecomposable into a free pro-$p$ product. Then the pro-$p$ tree of cylinders $T_c$ is    $Aut(G)$-invariant JSJ $\E$-tree relative to not virtually cyclic abelian subgroups. 

\end{corollary}

We introduce a variant in which the property of being universally elliptic is replaced
by the more restrictive and rigid property of being universally compatible. This yields
a canonical compatibility JSJ pro-$p$ tree, not just a deformation space. We show that it exists
for a  pro-$p$ group, without
any assumption on edge groups (in contrast to the abstract case when this is true only for finitely presented groups).

We give many examples, and we work from time to time with relative decompositions
(restricting to pro-$p$ trees where certain subgroups are elliptic).

\medskip
For the reader’s convenience we now describe the detailed content of each section.

\smallskip
In the preliminary section (Section 2), we collect basic facts about pro-$p$-groups acting on
pro-$p$ trees, and we define $(\E, \curlyH)$-trees (trees with edge stabilizers in $\E$ relative to $\curlyH$, collapse maps, graphs of pro-$p$ groups and frequently used results on pro-$p$ groups acting on pro-$p$ trees).  We also define deformation spaces here.

Section 3 describes all the instances of accessibility known in the pro-$p$ case.

 Section 4 introduces refinements, compatibility and 
 domination. It contains  a construction of standard refinement. Namely, let $G$ be a pro-$p$ group acting on pro-$p$ trees $T_1$ and $T_2$.  If all edge stabilizers of $T_1$ are elliptic  in $T_2$ and $|T_1/G|< \infty$,
there is a refinement of $T_1$ which dominates $T_2$. If one drops the condition  $|T_1/G|< \infty$ it does not work as in th abstract case: but under assumption that $G$ is finitely generated, one can choose a pro-$p$ tree $S$ from the deformation space of $T_1$ that admits a refinement $\widehat S$ of $S$ dominating $T_2$.

 Section 5 is dedicated to universal ellipticity and properties of it. 

In Section 6 we define
JSJ pro-$p$ trees and the JSJ deformation space.
We prove the existence of JSJ decompositions. 

Section 7 is devoted to examples. We first consider Grushko decompositions
(over the trivial group) and Stallings-Dunwoody decompositions (over finite $p$-groups), explaining how to interpret them as JSJ decompositions. We also consider small pro-$p$ groups, and
locally finite pro-$p$ trees (such as those associated to cyclic splittings of generalized Baumslag-
Solitar pro-$p$ groups). All these examples of JSJ decompositions  have only rigid vertices. At the
end of the section we work out an example where the JSJ decomposition has flexible
vertices. Moreover, we prove that flexible vertices of cyclic pro-$p$ trees split a s a free pro-$p$ product.

\begin{theorem} Let $G$ be a finitely generated pro-$p$ group and $\E$ be the family of infinite cyclic pro-$p$ groups. If $Q\neq \Z_p$ is a vertex stabilizer of a cyclic JSJ $G$-tree and  $Q$ does not split as a free pro-$p$ product then it is elliptic in any cyclic $G$-tree.\end{theorem}

Section 8 contains various useful technical results. Given a vertex of a graph of pro-$p$ 
groups or of a pro-$p$ tree,  we point out that any splitting
of the vertex group which is relative to the incident edge groups extends to a splitting of
$G$. Given a universally elliptic splitting of G, one may obtain a JSJ decomposition of $G$
from relative JSJ decompositions of vertex groups. In particular, one may usually restrict
to one-ended groups when studying JSJ decompositions (see Corollary \ref{refining Grushko}).
We also define and discuss relative finite generation and presentation (relative
finite presentation of vertex groups is studied in Subsection 8.3).

Section 9 is devoted to the pro-$p$ tree of cylinders. Given an admissible equivalence relation
on the set $\E$ of pro-$p$ groups in $\E$, one may associate a pro-$p$ tree of cylinders $T_c$ to any
pro-$p$ tree $T$ with compact set of edges and  with edge stabilizers in $\E$. This pro-$p$ tree only depends on the deformation space of $T$. We apply it to deduce the result on $Aut(G)$ in Section 9.2. Here is one of them, where for a subgroup $K\leq G$ we use notation $Aut_G(K)$ to be the group of all automorphisms of $G$ that leave $K$ invariant.

\begin{theorem}\label{amalgamint} Let $G=G_1\amalg_H G_2$ be a free pro-$p$ product with amalgamation. Suppose that $G_1$ and $G_2$ are rigid. Then $Aut(G)$ is either 

$$Aut(G)= Aut_G(G_1)\amalg_{ Aut_G(G_1)\cap Aut_G(N_G(H))}Aut_G(N_G(H))\amalg_{Aut_G(G_2)\cap Aut_G(N_G(H))} Aut_G(G_2)$$ 
or
$$Aut(G)= Aut_G(G_1)\amalg_{ Aut_G(G_1)\cap Aut_G(N_G(H))}Aut_G(N_G(H)),$$
where 1 can be replaced by 2.

	\end{theorem}
	
	If $H$ is malnormal then formula for splitting simplifies.
	
	\begin{corollary}\label{malnormalamalgint} Let $G=G_1\amalg_H G_2$ be a non-fictitious free pro-$p$ product with malnormal amalgamation. Suppose that $G_1$ and $G_2$ are rigid. Then $Aut(G)$ is either 
$$Aut(G)= Aut_G(G_1)\amalg_{Aut_G(H)}Aut_G(G_2)$$ 
or
$$Aut(G)= Aut_G(G_1),$$
where 1 can be replaced by 2.
	 \end{corollary}

In Subsection 9.3 we study tree of cylinders for torsion-free CSA groups:  pro-$p$ limit groups is the canonical example of them. 

In Section 10 we extend Corollary \ref{malnormalamalgint} to  pro-$\C$ groups assuming that the amalgamated subgroup is malnormal ( $\C$ is a class of finite groups closed for subgroups, quotients and extensions).

In Section 11 we define universal compatibility and we show Theorem  \ref{co} (existence of
the compatibility JSJ deformation space $D_{co}$). We then construct the compatibility JSJ
tree $T_{co}$ and we give examples.

In the last section we prove results on compatibility of free constructions (free pro-$p$ products with amalgamation and HNN-extensions) with free and virtually free pro-$p$ products. Here are two samples of them.

\begin{theorem}\label{one edgeint} Let $G=G_1\amalg_{C} G_2$ be a finitely generated  free pro-$p$ product with cyclic amalgamation. Then $G$  splits as a free pro-$p$ product if and only if $C$  belongs to a free factor of $G_1$  or $G_2$. 
\end{theorem}

\begin{theorem}\label{one edgeHNN} Let $G=HNN(G_1,C, t)$  be a finitely generated  HNN-extension with cyclic associated subgroup. If $G$  splits as a free pro-$p$ product then $C$ or $C^t$  belongs to a free factor of $G_1$. 
\end{theorem}

\bigskip
 {\bf Conventions.} Throughout the paper, unless otherwise stated, groups are pro-$p$, subgroups are closed, and homomorphisms are continuous. In particular, $\langle H\rangle$ will mean the topological generation in the paper and presentations are taking in the category of pro-$p$ groups; $a^g$ will stand for $g^{-1}ag$. $\Z_p, \F_p$ denote the p-adic integers and the field of $p$ elements respectively and $\Phi(G)$ denotes the Frattini subgroup of $G$.

 \bigskip
 
 {\it Acknowledgment.} The paper was started when the author was visiting Department of Mathematics of University of Cambridge. The author thanks Henry Wilton for fruitful discussions, especially for pointing out abstract versions of the result of the last section of the paper.

\section{Preliminaries}

 In this section, we recall the necessary elements of the theory of pro-$p$ trees. One can find more on this in \cite{RZ-00, R}.

A \emph{graph} $\Gamma$ is a disjoint  union $E(\Gamma) \cup V(\Gamma)$
with two maps $d_0, d_1 : \Gamma \to V(\Gamma)$ that are the
identity on the set of vertices $V(\Gamma)$.  For an element $e$ of
the set of edges  $E(\Gamma)$, $d_0(e) $ is called the \emph{initial} and
$d_1(e) $ the \emph{terminal} vertex of $e$.

\begin{definition}
A \emph{profinite graph} $\Gamma$ is a graph such that:
\begin{enumerate}
\item $\Gamma$ is a profinite space (i.e.\ an inverse limit of finite
discrete spaces);
\item $V(\Gamma)$ is closed; and
\item the maps $d_0$ and $d_1$
are continuous.
\end{enumerate}
Note that $E(\Gamma)$ is not necessary closed.
\end{definition}

A \emph{morphism} $\alpha:\Gamma\longrightarrow \Delta$ of profinite graphs is a continuous map with $\alpha d_i=d_i\alpha$ for $i=0,1$.

By \cite[Prop.~1.7]{ZM-88} every profinite
graph $\Gamma$ is an inverse limit of finite quotient graphs of
$\Gamma$. A {\it connected} profinite graph is by definition a profinite graph whose all finite quotient graphs are connected (in the usual sense).

\subsection{Trees}

  Let $\F_p$ be the field of $p$ elements. 
For a profinite space $X$  that is the inverse limit of finite
discrete spaces $X_j$, $[[\mathbb{F}_p X]]$ is the inverse
limit  of $ [\mathbb{F}_p X_j]$, where
$[\mathbb{F}_p X_j]$ is the free
$\mathbb{F}_p$-module with basis $X_j$. For a pointed
profinite space $(X, *)$ that is the inverse limit of pointed
finite discrete spaces $(X_j, *)$, $[[\mathbb{F}_p (X,
*)]]$ is the inverse limit  of $ [\mathbb{F}_p (X_j, *)]$,
where $[\mathbb{F}_p (X_j, *)]$ is the free
$\mathbb{F}_p$-module with basis $X_j \setminus \{ *
\}$ \cite[Chapter~5.2]{RZ-10}.

For a profinite graph $\Gamma$, define the pointed space
$(E^*(\Gamma), *)$ as  $\Gamma / V(\Gamma)$, with the image of
$V(\Gamma)$ as a distinguished point $*$, and denote the image of $e\in E(\Gamma)$ by $\bar{e}$.  
A profinite graph is called a \emph{pro-$p$ tree} if one has the following exact sequence:

$$
0 \to [[\mathbb{F}_p(E^*(\Gamma), *)]]
\stackrel{\delta}{\rightarrow} [[\mathbb{F}_p V(\Gamma)]]
\stackrel{\epsilon}{\rightarrow} \mathbb{F}_p \to 0
$$
where $\delta(\bar{e}) = d_1(e) - d_0(e)$ for every $e \in E(\Gamma)$ and $\epsilon(v) = 1$ for every $v \in V(\Gamma)$. 

 If $v$  and $w$ are elements  of a profinite tree   $T$, we denote by $[v,w]$ the smallest profinite subtree of $T$ containing $v$ and $w$ (i.e. the intersection of all pro-$p$ subtrees containing $v,w$) and call it geodesic. 

\subsection{G-trees}

By definition, a profinite group $G$ \emph{acts} on a profinite graph
$\Gamma$ if  we have a continuous action of $G$ on the profinite
space $\Gamma$ that commutes with the maps $d_0$ and $d_1$. We shall denote by $T^G$ the set of fixed points for $G$.

\begin{proposition}\label{mod tilde}(\cite[Proposition 3.5]{RZ-00}). Let $G$ be a pro-$p$ group acting on a pro-$p$ tree $T$ and $\widetilde G=\langle G_v\mid v\in V(T)\rangle$ be the subgroup generated by all vertex stabilizers. Then $\widetilde G$ is normal in $G$ and $T/\widetilde G$ is a pro-$p$ tree.\end{proposition}

\begin{definition}
	If $G$ acts on a pro-$p$ tree $T$, any subgroup or element fixing a vertex will be called 
	{\it elliptic} for the pair $(G,T)$. Otherwise the element is called hyperbolic. We denote by $\widetilde G$ the subgroup of $G$ generated by all elliptic elements. \end{definition}

We prove here a pro-$p$ version of the statement proved by Tits \cite[3.4]{tits}; note that the Tits option of having a $G$-invariant end does not occur in the pro-$p$ case. So the two propositions below  are valid in abstract case only for finitely generated groups.

\begin{proposition}\label{ellipticity} Let $G$ be a pro-$p$ group acting on a pro-$p$ tree $T$. If every element of $G$ is elliptic then so is $G$.
	\end{proposition}

\begin{proof} Let $U$ be an open subgroup of $G$. By  Proposition \ref{mod tilde} $G/U$ acts on pro-$p$ tree $T/ U$ and so by \cite[Theorem 3.9]{RZ-00} fixes a vertex, i.e. the subset of fixed points $T/U^{G/U}\neq \emptyset$.  Then $ T^G=\varprojlim_{U}T/U^{G/U}\neq \emptyset$ as required.\end{proof}

From the proof of Proposition \ref{ellipticity} we deduce the following

\begin{corollary}\label{fixed vertex} Let $G$ be a pro-$p$ group acting on a pro-$p$ tree $T$. Suppose that $G$ is virtually elliptic, i.e. possesses an open normal elliptic subgroup $U$.
Then $G$ is elliptic.

\end{corollary}

\begin{theorem}\label{fixed geodesic} (\cite[Corollary 3.8]{RZ-00}).
Suppose that a pro-$p$ group $G$ acts on a pro-$p$ tree $T$, and let $v$ and
$w$ be two different vertices of $T$. Then the set of edges $E([v,w])$ of the geodesic $[v,w]$
is nonempty, and $G_v \cap G_w \leq G_e$ for every $e \in E([v,w])$.\end{theorem}

\begin{proposition} \label{hyperbolic in virtual abelianization} Let $G$ be a pro-$p$ group acting on a pro-$p$ tree $T$ and $g$ a hyperbolic element of $G$. Then there exists an open subgroup $U$ containing it such that for any $g\leq K\leq U$ the image of $g$ in $K^{ab}$ is not torsion.\end{proposition}

\begin{proof} Since $\langle g\rangle=\bigcap \{U\mid g\in U\leq_o G\}$ we can view it as $\langle g\rangle=\varprojlim_{U\leq_o G} U$. It follows that $\langle g\rangle=\varprojlim_{U\leq_o G} U/\widetilde U$ and moreover $\langle g\rangle$ is the projective limit of the abelianizations $\langle g\rangle=\varprojlim_{U\leq_o G} (U/\widetilde U)^{ab}$, where $\widetilde U$ is the normal subgroup of $U$ generated by elliptic elements. Hence the image $g_U$ of $g$ in $(U/\widetilde U)^{ab}$ is non-trivial for some $U$. By \cite[Corollary 3.6]{RZ-00} $U/\widetilde U$ is free pro-$p$ and so its abelianization is torsion free and therefore $g_U$ is not torsion. Now for any $K\leq U$ containing $g$ one has the natural map $(K/\widetilde K)^{ab}\longrightarrow (U/\widetilde U)^{ab}$ that proves the result. 

\end{proof}

\begin{proposition}\label{unique G-invariant} (\cite[Lemma 3.11]{RZ-00}). Let $G$ be a pro-$p$ group acting on a pro-$p$ tree $T$. Then there exists a minimal $G$-invariant pro-$p$ subtree $D$ of $T$. If $|D|> 1$ then it is unique.\end{proposition}

\begin{theorem}\label{splitting} (\cite[Theorem 4.2]{CZ}). Let $G$ be a finitely generated pro-$p$ group acting on a pro-$p$ tree $T$ with no global fixed point. Then it splits as an amalgamated free pro-$p$ product or pro-$p$ HNN-extension over some  edge stabilizer $G_e$ in $T$.\end{theorem}

\begin{definition}\label{def: Profinite acylindrical}
The action of a pro-$p$ group $G$ on a pro-$p$ tree $T$ is said to be \emph{$k$-acylindrical}, for $k\in \N$ a constant, if the set of fixed points of $g \in G$ has diameter at most $k$ whenever $ g\neq 1$.
\end{definition}

\begin{definition}\label{tree reduced} A $G$-tree is called reduced if for any $e\in E(T)$ with vertices $v,w$ one has $G_v\neq G_e\neq G_w$ unless $v,w$ are in the same $G$-orbit.\end{definition}

\subsection{$(\A,\K)$-trees}

Besides $G$, we usually also fix a (nonempty) family $\E$ of subgroups of $G$ which is stable under conjugation and under taking subgroups.  More precisely, for a pro-$p$ groups $G$, define $Subgp(G)$ to be the set of all closed subgroups of $G$ and the topology for which the subsets $\{Subgp(U)\mid U\leq_o G\}$ of subsets of $Subgp(G)$, for all open subgroups $U$ of $G$, is a base for the topology on $Subgp(G)$; it is called the {\it \'etale} topology on $Subgp(G)$.

\begin{definition} Let $G$ be a pro-$p$ group and $\E$ be a   family of subgroups of $G$ closed for  conjugation.  An $\E$-tree is a tree $T$ whose edge stabilizers belong to $\E$.
	
\end{definition}

\begin{definition} Let $X$ be a profinite space and $G$ a profinite group. A family of closed subgroups $\{G_x\mid x\in X\}$  of $G$ is said to be continuous if for any open subset $U$ of $G$ the subset $\{ x\in X\mid G_x\subseteq U\}$ is open in $X$. \end{definition}

Note that for a continuous family $\{G_x\mid x\in X\}$ and any closed subset $Y$ in $X$ the subfamily $\{G_y\mid y\in Y\}$ is continuous.

\begin{lemma}\label{Ribes 5.2.2}(\cite[Lemma 5.2.2]{R}). Let  $G$  be a profinite group  acting continuously on a profinite set $X$ is continuous. Then the family of a point stabilizers $\{G_x\mid x\in X\}$ is continuous.\end{lemma} 

In particular, if $G$ acts on a profinite graph $T$ the the family of stabilizers $\{G_t\mid t\in T\}$ and the family of the vertex stabilizers $\{G_v\mid v\in V(T)\}$
is continuous. 
By \cite[Lemma 5.2.1]{R} a family 
$\{G_x\mid x\in X\}$ is continuous if and only if  $\varphi:X\longrightarrow Subgr(G)$ is continuous, where $Subgr(G)$ is endowed with the 
\'etale topology.

 We often say that the corresponding splitting of $G$ is $\E$-splitting or  is the splitting over $\E$, or over groups in $\E$. We say cyclic pro-$p$ tree (abelian pro-$p$ tree,  . . . ) when $\E$ is the family of cyclic (abelian, . . . ) subgroups. We also fix an arbitrary set $\K$ of subgroups of $G$, and we restrict to $\E$-trees $T$ such that each $K\in  \K$ is elliptic in $T$ (in terms of graphs of groups, $K$ is contained in a conjugate of a vertex group; if $K$ is not finitely generated, then in the abstract case this is stronger than requiring that every $k\in K$ be elliptic, but it the pro-$p$ case it is the same by Proposition \ref{ellipticity}). We call such a pro-$p$ tree an $(\E,\K$)-tree, or a tree over $\E$ relative to $\K$. The set of $(\E,\K)$-trees does not change if we replace a group of $\K$ by a conjugate, or if we enlarge $\K$ by making it invariant under conjugation. If G acts non-trivially on an $(\E,\K)$-tree (i.e. without global fixed point), we say that $G$ splits over $\E$ (or over a group of $\E$) relatively to $\K$. The group $G$ is freely indecomposable relative to $\K$ if it does not split as a free pro-$p$ product  relatively to $\K$ (this means that all subgroups of $\K$ are in the free factor up to conjugation). If $G$ is second countable then equivalently (unless $G=\Z_p$ and $\K$ is trivial),  $G$ acts on a pro-$p$ tree with trivial edge stabilizers relative to $\K$  (see \cite[Theorem 9.6.1]{R}). 
 
\begin{definition}  Let $G$ be a pro-$p$ group acting on pro-$p$ $\E$-trees $T_1$ and $T_2$. We say that $T_1$ and $T_2$ belong to the same deformation space if  the set of elliptic elements for $(G,T_1)$ and $(G,T_2)$ are the same. Equivalently, if $(G,T_1)$ and $(G,T_2)$  have the same  maximal vertex stabilizers.
	
\end{definition}

Another way of saying this is the following: two pro-$p$ trees belonging to the same deformation
space over $\E$ have the same vertex stabilizers, provided one restricts to groups not in $\E$.

\subsection{Graph of pro-$p$ groups}

 When we say that ${\cal G}$ is a \emph{finite graph of pro-$p$ groups}, we mean that it contains the data of the
underlying finite graph, the edge pro-$p$ groups, the vertex pro-$p$ groups and the attaching continuous maps. More precisely,
let $\Delta$ be a connected finite graph.  The data of a graph of pro-$p$ groups $({\cal G},\Delta)$ over
$\Delta$ consists of a pro-$p$ group ${\cal G}(m)$ for each $m\in \Delta$, and continuous monomorphisms
$\partial_i: {\cal G}(e)\longrightarrow {\cal G}(d_i(e))$ for each edge $e\in E(\Delta)$.

The definition of the pro-$p$ fundamental  group of a connected
profinite graph of pro-$p$ groups is quite involved (see
\cite{ZM-89}). However, the pro-$p$ fundamental  group
$\Pi_1(\G,\Gamma)$ of a finite graph of finitely generated
pro-$p$ groups $(\G, \Gamma)$ can be defined as the pro-$p$ completion
of the abstract (usual) fundamental group $\Pi_1^{abs}(\G,\Gamma)$
(we use here that every subgroup of finite index in a finitely
generated pro-$p$ group is open (cf. \cite[Theorem 1.1]{NS-07}).  In general, we can define the fundamental pro-$p$ group 
$\Pi_1(\G,\Gamma)$ of finite graph of pro-$p$ groups by
 the following presentation:
\begin{eqnarray}\label{presentation}
\Pi_1(\mathcal{G}, \Gamma)&=&\langle \G(v), t_e\mid rel(\G(v)),
\partial_1(g)=\partial_0(g)^{t_e}, g\in \G(e),\nonumber\\
 & &t_e=1 \  {\rm for}\ e\in
T\rangle;
\end{eqnarray}
here $T$ is a maximal subtree of $\Gamma$ and
$$\partial_0:\G(e)\longrightarrow
\G(d_0(e)), \partial_1:\G(e)\longrightarrow \G(d_1(e))$$ are
monomorphisms. The elements $t_e$ will be called stable letters.

If all $\G(v)=1$ for $v\in V(\Gamma)$, then $\Pi_1(\mathcal{G}, \Gamma)$ is simply the fundamental pro-$p$ group of the graph $\Gamma$ denoted by $\pi_1(\Gamma)$. Clearly, $\pi_1(\Gamma)$ is free pro-$p$. 

\begin{lemma}\label{generation} Let $G=\Pi_1(\G, \Gamma)$ be the fundamental group of a finite graph of pro-$p$ groups and suppose $G$ is finitely generated. Then 
$G=\langle g_1,\dotso, g_n\rangle$, where either $g_i$ is in a vertex group of $(\mathcal{G},\Gamma)$, or $g_i$ is a stable letter.
\end{lemma}

\begin{proof} Consider $\widetilde{G}=\langle G(m)^g\mid m\in \Gamma, g\in G \rangle$. By \cite[Corollary 3.9.3]{R}  $$G/\widetilde{G}\cong \pi_{1}(G\backslash T)=\pi_{1}(\Gamma)$$ is a free pro-$p$ group. Then $G=\widetilde{G}\rtimes \pi_{1}(\Gamma)$. Let $f:G\longrightarrow G/\Phi(G)$ be the natural epimorphism. Then $G/\Phi(G)=f(\overline{G})\oplus f(\pi_{1}(\Gamma))$ is an  $\F_p$-vector space (cf. \cite[Lemma 2.8.7 (b)]{RZ-00}). Since $f(\overline{G})$ is finite (because  $G$ is finitely generated), there are $w_1,\ldots, w_l\in V(\Gamma)$, such that $f(\widetilde{G})=\langle f(G_{w_1}),\dotso , f(G_{w_l})\rangle$. Then $G/\Phi(G)=\{f(G_{w_1}),\dotso , f(G_{w_l}),f(\pi_{1}(\Gamma))\}$. Hence $G=\overline{\langle g_1,\dotso, g_n\rangle}$, where eihter $g_i$ is in a vertex group of $(\mathcal{G},\Gamma)$, or $g_i$ is a stable letter.

\end{proof}

In contrast to the abstract case, the vertex groups of $(\G,
\Gamma)$ do not always embed in $\Pi_1(\G,\Gamma)$, i.e.,
$\Pi_1(\G,\Gamma)$ is not always proper. However, the edge and
vertex groups can be replaced by their images in
$\Pi_1(\G,\Gamma)$ and after this replacement $\Pi_1(\G,\Gamma)$
becomes proper. Thus from now on we shall assume that
$\Pi_1(\G,\Gamma)$ is always proper, unless otherwise stated.

If all vertex and edge groups are trivial, we get the definition of the fundamental pro-$p$ group of a finite connected graph which is the pro-$p$ completion of the abstract fundamental group of $\Gamma$; hence it is free pro-$p$ on the basis $\Gamma\setminus V(\Gamma)$.

\begin{remark}\label{reduction}  A finite graph of  pro-$p$ groups $(\G,\Gamma)$ is said to be {\it
reduced}, if for every edge $e$ which is not a loop
neither $\partial_1\colon \G(e)\to \G(d_1(e))$ nor $\partial_0:
\G(e)\to \G(d_0(e))$ is an isomorphism; we call an edge $e$ which is not a loop and such that one of the edge maps $\partial_i$ is an isomorphism \emph{fictitious}. Any finite graph of  pro-$p$ groups
can be transformed into a reduced finite graph of pro-$p$ groups  by collapsing fictitious  edges using the
following procedure. If $e$ is a fictitious  edge, we can remove $\{e\}$ from the edge set of $\Gamma$, and
identify $d_0(e)$ and $d_1(e)$ to a new vertex $y$. Let
$\Gamma^\prime$ be the finite graph given by
$V(\Gamma^\prime)=\{y\}\sqcup V(\Gamma)\setminus\{d_0(e),d_1(e)\}$
and $E(\Gamma^\prime)=E(\Gamma)\setminus\{e\}$, and let
$(\G^\prime, \Gamma^\prime)$ denote the finite graph of  groups
based on $\Gamma^\prime$ given by $\G^\prime(y)=\G(d_1(e))$ if
$\partial_0(e)$ is an isomorphism, and $\G^\prime(y)=\G(d_0(e))$ if
$\partial_0$ is not an isomorphism. This procedure can be
continued until there are no fictitious edges. The resulting finite graph of groups
$(\overline \G,\overline \Gamma)$ is reduced. 

The reduction procedure just described does not change the fundamental group (as a group given by presentation), i.e. choosing a maximal subtree to contain the collapsing edge  the morphism $(\G,\Gamma)\longrightarrow (\G^\prime, \Gamma^\prime)$ induces the identity map on the fundamental group given by presentation by eliminating redundant relations associated with the fictitious edge that are just collapsed by reduction.

 \end{remark}

The  pro-$p$ fundamental  group $\Pi_1(\G,\Gamma)$ acts on the standard 
pro-$p$ tree $S$ (defined analogously to the abstract one)
associated to it, with vertex and edge stabilizers being conjugates
of vertex and edge groups, and such that
$S/\Pi_1(\G,\Gamma)=\Gamma$  \cite[Proposition 3.8]{ZM-88}. Namely, the  {\it standard (universal) pro-$p$ tree}\label{standard}
	associated with the finite graph of pro-$p$ groups $({\cal G}, \Gamma)$    (or universal covering graph) is 
	$S=S(G)=\bigcupdot_{m\in \Gamma}
	G/\G(m)$ (cf. \cite[Proposition 3.8]{ZM-88}).  The vertices of
	$S$ are those cosets of the form
	$g\G(v)$, with $v\in V(\Gamma)$
	and $g\in G$; its edges are the cosets of the form $g\G(e)$, with $e\in
	E(\Gamma)$; and the incidence maps of $T$ are given by the formulas:
	
	$$d_0 (g\G(e))= g\G(d_0(e)); \quad  d_1(g\G(e))=gt_e\G(d_1(e)) \ \ 
	(e\in E(\Gamma), t_e=1\hbox{ if }e\in D).  $$
	
	There is a natural  continuous action of
	$G$ on $S$, and clearly $  S/G= \Gamma$. 
	Remark also that since $\Gamma$ is finite, $E(S)$ is compact.

\bigskip

Note that $(\G,\Gamma)$ is reduced if and only if $\pi_1(\G,\Gamma)$-tree $S(G)$ is reduced.

\bigskip
\begin{example}\label{graph group completion} If $G=\pi_1(\G,\Gamma)$ is the fundamental group of a finite graph of (abstract) groups then one has the induced graph of pro-$p$ completions of edge  and vertex groups $(\widehat\G,\Gamma)$ (not necessarily proper) and  a natural homomorphism $G=\pi_1(\G,\Gamma)\longrightarrow \Pi_1(\widehat\G,\Gamma)$. It is an embedding if $\pi_1(\G,\Gamma)$ is residually $p$. In this case $\Pi_1(\G,\Gamma)=\widehat{G}$ is simply the pro-$p$ completion.
Moreover,
 the standard tree $S(G)$
  naturally embeds in $S(\widehat G)$ if and only if  the edge and vertex groups $\G(e)$, $\G(v)$     are closed in the pro-$p$ topology of  $\pi_1(\G,\Gamma)$, or equivalently $\G(e)$ are closed in $\G(d_0(e))$,   $\G(d_1(v))$    with
respect to the topology induced by the pro-$p$ topology on $G$
(\cite[Proposition 2.5]{CB-13}).  \end{example}

\begin{definition}\label{acylidrical graph of groups} A graph of groups $(\G,\Gamma)$ will be called $k$-acylindrical if the action of $\Pi_1(\G,\Gamma)$ on $S=S(\G,\Gamma)$ is $k$-acylindrical.\end{definition}

\medskip
In what follows, the notion of the fundamental pro-$p$ group of an infinite profinite graph of pro-$p$ groups is used only in Theorem \ref{pro-pbass-serre}, and in fact only in the form of the inverse limit of fundamental groups of finite graphs of finite $p$-groups.  We state this as a proposition that the reader can use as a definition. This all done in details in \cite{AZ}.

\begin{definition} An inverse limit of  finite graphs of finite $p$-groups is called a \emph{profinite graph of pro-$p$ groups}.\end{definition}

\begin{proposition}\label{induces}( \cite[Proposition 2.1, Definition 2.13 and Proposition 2.19]{AZ} )  An inverse limit $(\curlyG,\Gamma)=\varprojlim_i (\curlyG_i,\Gamma_i)$ of finite graphs of  pro-$p$ groups  induces an inverse  limit of the fundamental  pro-$p$ groups $\varprojlim_i  \Pi_1(\curlyG_i,\Gamma_i)$ and $ \Pi_1(\curlyG,\Gamma)=\varprojlim_i  \Pi_1(\curlyG_i,\Gamma_i)$.\end{proposition}

If $(\G,\Gamma)$ is a finite graph of pro-$p$ groups and $\Delta$ a connected subgraph of $\Gamma$ 
we shall denote by $(\G,\Delta)$ the graph of pro-$p$ groups restricted to $\Delta$.

\begin{lemma}\label{restricted graph} (\cite[Lemma 2.4]{SZ}) Let $(\G, \Gamma)$ be a proper finite graph   of pro-$p$ groups and $\Delta$ a connected subgraph of $\Gamma$. Then the natural homomorphism $\Pi_1(\G, \Delta)\longrightarrow \Pi_1(\G, \Gamma)$ is a monomorphism.
\end{lemma}

 \begin{lemma}\label{trivial edge} Let $G=(\G, \Delta)$ be  the fundamental group of a finite graph of pro-$p$ groups. Suppose   $G=H_1\amalg H_2$ such that all infinite vertex groups are contained in $H_1$. Then
$(\G, \Delta)$ possesses a trivial edge group.   
   \end{lemma}
   
   \begin{proof} Let $\Delta_1$ be a maximal subgraph of $\Delta$ such that all vertex groups  of $\Delta_1$ are in $H_1$.  Since any finite subgroups of $G$ is conjugate into $H_1$ or $H_2$ (see \cite[Theorem 4.2]{RZ-00}) w.l.o.g. we may assume that $\Delta_1$ is non-empty. Then for every edge $e\not\in \Delta_1$ with vertex in $\Delta_1$ the edge group $G(e)=1$ since $H_1\cap H_i^g=1$ whenever $g\not\in H_1$ or $i\neq 1$ (cf. \cite[Theorem 9.1.12]{RZ-10}. As $\Delta_1$ is proper in $\Delta$ and $\Delta$ is connected, the result follows.  
     \end{proof}

\begin{corollary}\label{graph of groups free product decomposition} Let $G=(\G, \Delta)$  be  the fundamental group of a finite graph of finite 
$p$-groups. Then $G$ splits as a free pro-$p$ product $G=\coprod_{i=1}^n G_n\amalg F$ such that each $G_i$ is  the fundamental group  of a subgraph of groups $(\G, \Gamma_i)$ with non-trivial edge groups and $F$ is free pro-$p$.

\end{corollary}

We shall use several times later the following lemma which is the main ingredient of the proof of subsequent theorem.

\begin{lemma}\label{inverse limit of virtually free groups} (\cite[Lemma 4.1]{CZ}) Let $G$ be a finitely generated pro-$p$ group acting on a pro-$p$ tree $T$. Then $G=\varprojlim_{U\triangleleft_o G} G/\tilde U$ and $G_U=G/\tilde U=\Pi_1(\G_U,\Gamma_U)$ is the fundamental group of a finite reduced graph of  finite $p$-groups. Moreover, the inverse system $\{G/\tilde U, \pi_{VU}\}$ can be chosen in such a way that it is linearly ordered and for each $\{G/\tilde V\}$ of the system with $V\leq U$  there exists a natural morphism $(\eta_{VU},\nu_{VU}):(\G_V, \Gamma_V)\longrightarrow (\G_U, \Gamma_U)$ where $\nu_{VU}$ is just a collapse of edges of $\Gamma_V$ and $\eta_{VU}(\G_V(m))=\pi_{VU}(\G_V(m))$; the induced  homomorphism of the
pro-$p$ fundamental groups coincides with the canonical projection
$\pi_{VU}\colon G/\tilde V\longrightarrow G/\tilde U$.\end{lemma}

\begin{theorem} \label{pro-pbass-serre} (\cite[Theorem 1.1]{WZ18}) Let $G$ be a finitely generated pro-$p$ group acting on a pro-$p$ tree $T$. Then $G$ is the fundamental pro-$p$ group of a profinite graph of
pro-$p$ groups $(\G, \Gamma)$.
Moreover, the vertex and edge groups of $(\G, \Gamma)$ are  stabilizers of certain vertices and edges of $T$ respectively, and stabilizers of vertices and edges of $T$ in $G$ are conjugate to subgroups of vertex and edge groups of $(\G, \Gamma)$ respectively. 
\end{theorem}

\begin{remark} \label{cofinite action} If in Theorem \ref{pro-pbass-serre} $T/G$ is finite, then this theorem is subject of \cite[Proposition 4.4]{ZM-89} where finite generatedness of $G$ is not asked; in fact in this case theorem states exactly as the main theorem of Bass-Serre theory. Therefore, the reduction procedure can be applied directly to $T$, the obtained tree will also be called reduced.\end{remark}

We also going to use often the following

\begin{corollary}\label{open subgroup}(\cite[Corollary 4.5]{ZM-89}). Let $H$ be an open subgroup of the fundamental group $\Pi_1(\G,\Gamma)$ of a finite graph of pro-$p$ groups. Then $H$ is the fundamental group $\Pi_1(\curlyH,\Delta)$ of a finite graph of pro-$p$ groups that are intersections of $H$ with the vertex and edge groups of $(\G, \Gamma)$.

\end{corollary}

\bigskip
Let $G$ be a finitely generated pro-$p$ group acting on a pro-$p$ tree with finite vertex stabilizers.
It is not known whether $\Gamma$ can be made finite in Theorem \ref{pro-pbass-serre}. In the abstract case it comes automatically. One could achieve it if one could  bound the size of a  reduced finite graph of finite $p$-groups in terms of the minimal number of generators of its pro-$p$ fundamental group \cite[Theorem 6.3]{CZ} (i.e. if one have certain accessibility, see Section \ref{accessibility} for details). In the case of cyclic edge groups this is possible.

\begin{corollary}(\cite[Theorem 6.6]{CZ}) \label{cyclic edge stabilizers} Let $G$ be a finitely generated pro-$p$ group acting on a pro-$p$ tree $T$ with cyclic edge stabilizers. Then $G$ is the fundamental pro-$p$ group of a finite graph of
pro-$p$ groups $(\G, \Gamma)$.
Moreover, the vertex and edge groups of $(\G, \Gamma)$ are  stabilizers of certain vertices and edges of $T$ respectively, and stabilizers of vertices and edges of $T$ in $G$ are conjugate to subgroups of vertex and edge groups of $(\G, \Gamma)$ respectively.\end{corollary}

We finish the section stating the pro-$p$ version of the Kurosh subgroup theorem.

\begin{theorem}\label{KST}\cite{Me}  Let
$G = \coprod_{t\in T} G_t$ be a second countable pro-$p$ group
decomposed as a free pro-$p$ product and let $H$ be a subgroup of
$G$. There exists a closed system $\{g_{t,\tau}\in H\backslash
G/G_t\mid t\in T\}$ of double coset representatives such that
$$H=(\coprod_{t\in T} \coprod_{g_{t,\tau}\in H\backslash G/G_t} H\cap G_t^{g_{t,\tau}})\amalg F,$$
where $F$ is a free pro-$p$ group.\end{theorem}

We shall need later a description of the normalizers of a subgroup over which splittings occur. 

\begin{proposition}\label{normalizer} \begin{itemize} 

\item[(1)]
Let $G=G_1\amalg_HG_2$ be a free pro-$p$ product with amalgamation. Then $N_G(H)=N_{G_1}(H)\amalg_H N_{G_2}(H)$.  

\item[(2)] Let $G=HNN(G_1,H,t)$ be  a pro-$p$  HNN-extension. 

\begin{itemize}

\item If $H$ and $H^{t}$ are conjugate in $G_1$ then: $N_{G}(H)=HNN(N_{A}(H),H,t')$ and $G=HNN(G,H,t')$.

\item If $H$ and $H^{t}$ are not conjugate in $G_1$ then $N_{G}(H)=N_1\amalg_{H} N_{2}$, where $N_{1}=N_{G_1^{t}}(H)$ and $N_{2}=N_{G_1}(H)$.

\end{itemize}

\end{itemize}

\end{proposition}

\begin{proof} (1) Note that the fact that all conjugates of $H$ containing $H$ coincide with $H$ in the pro-$p$ case (as well as in the profinite) imply that $N_G(H)$ acts on the subtree of fixed points $S^H$ such that $S^H/N_G(H)=S/G$ and so   $N_G(H)={N_{G_1}(H)}\amalg_H {N_{G_2}(H)}$. 

\medskip
(2) Case 1. $H$ and $H^t$ are not conjugate in $G_1$. 

Like in (1) $N_G(H)$ acts on the subtree of fixed points $S^H$ such that $E(S^H)/N_G(H)=E(S)/G$ and  $V(S^H)/N_G(H)$ has two vertices.

In this case the normalizer splits as in (1): $N_G(H)={N_{G_1}(H)}\amalg_H {N_{G_1}(H^t)}$. 

\medskip
Case 2. 
$H$ and $H^t$ are  conjugate in $G_1$ say $H^{g_1}=H^t$ for some $g_1\in G_1$. Then replacing $t$ by $tg_1^{-1}$ we may assume that $t$ normalizes $H$.  
In this case $N_G(H)=HNN({N_{G_1}(H)},H,t') $ is an HNN-extension.

\end{proof}

\section{Accessibility }\label{accessibility}

Constructions of JSJ decompositions are based on accessibility theorems stating that, given suitable $G$ and $\E$, there is an a priori bound for the number of orbits of edges of reduced $(G,\E)$-trees, under the assumption that there is no redundant vertex (if $v$ has valence 2, it is the unique fixed point of some $g\in G$). 

\begin{definition} Let $\E$ be a family of pro-$p$ groups. A pro-$p$ group $G$ will be called 
	$\E$-{\em accessible} if there is a number $a=a(G)$ such that any finite,
	proper, reduced graph of pro-$p$ groups  with  edge groups in $\E$
	having fundamental group isomorphic to $ G$ has at most $a$
	edges. If $\E$ is the class of all finite $p$-groups, an $\E$-accessible pro-$p$ group will simply be called {\em accessible}. \end{definition}

The accessibility holds in the pro-$p$ case in the following cases:

(1)  $G$ is finitely generated and all groups in $\E$ are finite with bounded order \cite{wilk};

(2)  $G$ is finitely generated and all groups in $\E$ are cyclic \cite{CZ};

(3)  $ G$ is finitely generated and the graph of groups  are $k$-acylindrical for some $k$ \cite{CAZ}. 

\medskip

For completeness we state precisely these three results.

\begin{theorem}\cite[Theorem 3.1]{wilkes} Let $G$ be a finitely generated pro-$p$ group and let $k$ be an integer. Let $(\G,\Gamma)$ be a proper reduced finite graph of pro-$p$ groups with $G=\Pi_1(\G,\Gamma)$ such that each edge group $G_e$ has size at most $k$. Then $\Gamma$ has at most
	$\frac{pk}{p-1} (d(G)-1)+1$
	edges, where $d(G)$ is the minimal number of generators of $G$.\end{theorem}

\begin{theorem}\cite[Theorem 6.6]{CZ}\label{cyclic accessibility}     Let $G=\pi_1(\G,\Gamma)$ be the fundamental group of a
	finite reduced graph of pro-$p$-groups, with procyclic edge groups, and assume
	that $d(G)=d\geq 2$. Then  the vertex groups are finitely generated, the number
	of vertices of $\Gamma$ is $\leq 2d-1$, and the number of edges of $\Gamma$ is
	$\leq 3d-2$.\end{theorem}

\begin{theorem}\label{k-acylindrical accessibility}\cite[Theorem 1.1]{CAZ}  Let $G=\Pi_1(\G, \Gamma)$ be the fundamental group of a finite reduced $k$-acylindrical graph of pro-$p$ groups. Then $|E(\Gamma)|\leq d(G)(4k+1)-1 $, $|V(\Gamma)|\leq 4kd(G)$.  
\end{theorem}

We shall state here an accessibility for CSA pro-$p$ groups. We say that a pro-$p$ group is ${\it SCA}$ if the maximal abelian subgroups of $G$ are malnormal; this is equivalent to transitive commutativity of $G$. The proof of the next theorem will be performed in Subsection 9.3, but we state it here for the completeness.

\begin{theorem}\label{CSA} Let $\E$ be a class of abelian pro-$p$ groups and $G$ a finitely generated pro-$p$ CSA-group. Then $G$ 
	is $\E$-accessible.
\end{theorem}

Denote by
$\mathcal{G}_0$ the class of all free pro-$p$ groups of finite
rank. Following  \cite{KZ}  define inductively the class $\mathcal{G}_n$ of pro-$p$
groups $G_n$ in the following way: $G_n$ is a free 
amalgamated pro-$p$ product $G_{n-1}\amalg_{C}A$, where $G_{n-1}$ is any
group from the class $\mathcal{G}_{n-1}$, $C$ is any
self-centralizing procyclic pro-$p$ subgroup of $G_{n-1}$ and $A$
is any finite rank free abelian pro-$p$ group such that $C$ is a
direct summand of $A$. The class of pro-$p$ groups $\mathcal{L}$ ({\em pro-$p$ limit groups}) 
consists of all finitely generated pro-$p$ subgroups $H$ of some
$G_n\in \mathcal{G}_n$, where $n\geq 0$. 
Then $H$ is a  subgroup of a free
amalgamated pro-$p$ product $G_n=G_{n-1}\amalg_{C}A$, where
$G_{n-1}\in \mathcal{G}_{n-1}$, $C\cong \mathbb{Z}_p$ and
$A=C\times B\cong \mathbb{Z}_p^m$. By 
\cite[Theorem 3.2]{Ribes3}, this free amalgamated pro-$p$ product is proper. Thus $H$
acts naturally on the pro-$p$ tree $T$ associated to $G_n$  and its edge stabilizers are procyclic.

By \cite{KZ} $L$ is CSA and so  we deduce  

\begin{corollary} Let $G$ be a pro-$p$ limit group. Then $G$ is accessible over abelian subgroups.
\end{corollary}

Of course pro-$p$ limit pro-$p$ groups accessible also over cyclic groups by Theorem \ref{cyclic accessibility}. The accessibility number  $a(G)$ is $3d(G)-2$ by \cite[Corollary 6.7]{CZ}.

\begin{proposition}\label{finite accessibility} Suppose $\E$ is a finite set of finitely generated subgroups of a finitely generated pro-$p$  group up to conjugation. Then $G$ is $\E$-accessible. More precise, $|V(\Gamma)|\leq d(G)|E/G|$. 
\end{proposition}

\begin{proof} 
	
	Let $(\G,\Gamma)$ be a reduced graph of groups with edge groups in $\E$ such that $G=\Pi_1(\G,\Gamma)$.  First observe that, for every maximal subtree $T$ of $\Gamma$, there are at most $d(G)$ edges in $\Gamma\setminus T$. Therefore it suffices to bound the number of edges of $\Gamma$ whenever $\Gamma$ is a tree. We use induction on $|\E/G|$.

	Let $S$ be the standard pro-$p$ tree for $(\G,\Gamma)$. By \cite[Theorem 2.10]{ZM-88} or \cite[Theorem 3.7]{RZ-00} the subset of fixed points $S^A$ of $A$ is a pro-$p$ subtree of $S$. One deduces that the maximal  subgraph $\Delta_A$ of  $\Gamma$ such that $A$ is a subgroup of each edge group of $(\G,\Delta_A)$ is connected.  
	 Choose $A$ to be minimal in $\E$ (by inclusion). Collapse  $\Delta_B$ for each $B$ minimal distinct from $A$. Putting $\Pi_1(\G, \Delta_B)$ on top of the vertex $b$ to which a subgraph $\Delta_B$ was collapsed for every such $B$ we obtain the graph of groups $(\G_A,\Delta_A)$ such that $G=\Pi_1(\G_A,\Delta_A)$ and every edge group is $A$.   
	
	We show the result  for $G=\Pi_1(\G,\Delta_A)$.  Since $(\G,\Delta_A)$ is reduced, $G_A/(A)^{G_A}\cong \Pi_1(\overline\G,\Delta_A)$, where $ \overline G(v)=G(v)/A^{G(v)}$, and so $(\overline\G,\Delta_A)$ is reduced (since $A$ is conatined in every edge group of $(\G,\Delta_A)$ and the number of edge subgroups in $G/(A)^{G_A}cong \Pi_1(\overline\G,\Delta_A)$ is $|\E/G|-1$.   Therefore, $|V(\Delta_A)|\leq d(G)(|\E/G|-1)$ by inductive hypothesis. Hence $|V(\Delta_A)|\leq  d(G)(|\E/G|-1)\leq d(G)|\E/G|$ as required.
	
	Now $\Gamma$ is covered by $\Gamma_A$, where $A$  runs via minimal subgroups of $E$. Thus $V(\Gamma)\leq |E/G| d(G)$.

\end{proof}

\section{Standard refinements}

  A collapse map $f : T \longrightarrow T_
0$
is a map obtained by collapsing connected components of certain subgraph $D$ to vertices. If $T$ is a $G$-tree, then assuming that $D$ is $G$-invariant we have a $G$-collapse. Then 
by equivariance, the set of collapsed edges is $G$-invariant.
Equivalently, $f$ preserves alignment: the image of any geodesic $[a, b]$ is a vertex or the geodesic
$[f(a), f(b)]$. Another characterization is that the preimage of every subtree is a subtree (cf. \cite[Lemma 3.1]{CZ}).

\begin{remark}\label{collapsed graph of groups} If $T$ has a connected transversal of $T/G$ one can express it in terms of graphs of groups. Namely one passes from $(\G,\Gamma) = T /G$ to $\Gamma_0 = T_0/G$. Then the procedure of collapsing in the graph of pro-$p$ groups $(\G,\Gamma)$ can be generalized using Lemma \ref{restricted graph}. If $\Delta$ is a connected subgraph then we can collapse $\Delta$ to a vertex $v$ and put $G(v)=\Pi_1(\G,\Delta)$ leaving the rest of edge and vertex groups unchanged. The fundamental group $\Pi_1(\G_{\Delta},\Gamma/\Delta) =\Pi_1(\G,\Gamma)$. The graph of groups  $(\G_{\Delta},\Gamma/\Delta)$ will be called {\it collapsed}.
\end{remark}

\begin{definition} We say that a a pro-$p$ $G$-tree
$T$ refines a pro-$p$ $G$-tree $T_
0$ if there exists a $G$-collapse $f:T\longrightarrow  T_0$.\end{definition}

The inverse operation of the collapse is a refinement or the blowup. Suppose $(\G,\Gamma)$ is a finite graph of pro-$p$ groups and  $v$ is a vertex of  $\Gamma$. 
Let  $(\G_v,\Gamma_v)$ be a splitting of $G_v$ in
which incident edge groups to $v$ are elliptic. One may then refine $\Gamma$ at $v$ using $\Gamma_v$, so as to
obtain a pro-$p$ splitting $(\G,\Gamma_0)$ whose edges are those of $\Gamma$
together with those of $\Gamma_v$. Note that $\Gamma_0$ is not uniquely defined because there is flexibility in the way edges of
$\Gamma_
0 $ are attached to vertices of $\Gamma_v$; (abstract version  is discussed in [GL16, Section 4.2]).

\begin{definition} Two  $G$-trees $T_1, T_2$ are compatible if they have a common refinement: there exists a tree
$\hat T$ with collapse maps $g_i
: \hat T\longrightarrow T_i$ $i=1,2$.\end{definition}

 The next lemma shows that there is minimal such a $\hat T$.

 \begin{lemma} Let $T_1$ and $T_2$ be compatible trees with a common refinement. There is  a common refinement $T$ of $T_1$ and $T_2$ such that no edge of
 $T$ is collapsed in both $T_1$ and $T_2$.\end{lemma}
 
 \begin{proof} Let $T$ be any common refinement. Let $\Delta_1$ and $\Delta_2$ be the subgraphs of $T$ whose connected components are collapsed to get $T_1$ and $T_2$ respectively. Put $\Delta=\Delta_1\cap\Delta_2$.  We collapse all connected components of $\Delta$.
   It is a common refinement of $T_1$ and $T_2$, and no edge is collapsed in both $T_1$
 and $T_2$.\end{proof}

Let $T_1$, $T_2$ be pro-$p$ $G$-trees.

\begin{definition} We say that a $G$-tree $T_1$ dominates a $G$-tree $T_2$ if every vertex stabilizer of $T_1$
fixes a vertex in $T_2$,  i.e. every subgroup which is elliptic in $T_1$ is also elliptic in $T_2$.\end{definition}

In particular,
every refinement of $T_1$ dominates $T_1$. Beware that domination is defined by considering
ellipticity of subgroups, not just of elements (this may make a difference in the abstract case if vertex stabilizers
are not finitely generated, but not in pro-$p$ by Proposition \ref{ellipticity}).

\begin{example} Let $G$ be a virtually free pro-$p$ group. Then by  \cite{Zperm} $G$ acts on a pro-$p$ tree $T$ with finite vertex stabilizers. So the   $G$-tree $T$  dominates any  $G$-tree by Corollary \ref{fixed vertex}.\end{example}

\begin{definition} (Ellipticity of trees) We say that a $G$-tree $T_1$ is elliptic with respect to $T_2$ if every edge stabilizer
of $T_1$ fixes a vertex in $T_2$.\end{definition}

\begin{example} If the edge stabilizers of $T_1$ are   finite (or more generally $FA$ pro-$p$, i.e.  pro-$p$ groups  that do not act on a pro-$p$ tree without global fixed point, then $T_1$ is elliptic with respect to any  $G$-tree.\end{example}

\begin{proposition}\label{refinement} Let $T_1$ and $T_2$ be pro-$p$ $G$-trees such that   $T_1/G$ is finite. Suppose $T_1$ is elliptic with respect to $T_2$.  Then there exists a $G$-tree $\hat T_1$ with  a morphism $p:\hat T_1\longrightarrow T_1$  such that
	
	\begin{enumerate}
		\item[(1)] $p$ is a collapse map;
		\item[(2)] $\hat T_1$ is a refinement of $T_1$ that dominates $T_2$;
		\item[(3)] the stabilizer of any edge of $\hat T_1$ fixes an edge in $T_1$ or in $T_2$;
		\item[(4)]  a subgroup of $G$ is elliptic in $\hat T_1$ if and only if it is elliptic in both $T_1$ and $T_2$.
		
		\end{enumerate}

\end{proposition}

\begin{proof}  We construct $\hat T_
	1$ as follows.
	For each vertex $v \in V (T_1)$, with stabilizer $G_v$ which is not elliptic in $T_2$, choose the minimal $G_v$-invariant subtree $\tilde Y_
	v$ of
	$T_2$ (or put $\tilde Y_v=T_2$). For each
	edge $e$ with vertices  $v,w \in E(T_1)$, choose vertices $p_v \in \tilde Y_v$ and $p_w \in \tilde Y_w$ fixed by $G_e$; this is possible
	because $G_e$ is elliptic in $T_2$ by assumption and so has a fixed points in  $\tilde Y_v$ and $ \tilde Y_w$ by \cite[Corollary 3.10]{RZ-00}.
	We can make these choices $G$-equivariantly,  since $T_1/G$ is finite.
	We can now define a tree $\widetilde T_1$ by blowing up each vertex $v$ of $T_1$ into $\widetilde Y_v$, and attaching
	edges of $T_1$ using the points $p_v$. Formally, we consider the disjoint union $(\bigcup_{v\in V (T_1)}\tilde Y_v) \cup (\bigcup_{e\in E(T_1)}e)$ (with $Y_v=v$ if $G_v$ is elliptic in $T_2$), and for each edge $e$ with vertices  $v,w$ of $T_1$ we identify $v$ with $p_v \in \tilde Y_v$ and $w$ with $p_w \in \tilde Y_w$. We define $\tilde p : \widetilde T_1 \longrightarrow T_1$ by sending $Y_v$ to $v$, and sending $e \in E(T_1)$ to itself. This map clearly
	satisfy the first  requirement. Then by \cite[Proposition on page 486]{Z-92} or\cite[Proposition 3.9.1 combined with Corollary 3.10.2 and Propositiuon 3.10.4]{R} 
	 $\widetilde T_1$ is a pro-$p$ tree.
	
	In general, $\widetilde T_
	1$ may fail to be minimal, so we define $\widehat T_
	1 \subset \widetilde T_1$ as the unique minimal $G$-invariant subtree  (the action of $G$ on $\widehat T_1$ is non-trivial unless $T_1$ and $T_2$ are both
	points). We then define $p$  as the restriction of $\tilde p$  to $\hat T_
	1$. 
	Let us check that the other properties follow. Assertion (2) is clear.
	
	If $e$ is an edge of $\hat T_
	1$ that is not collapsed by $p$, then $G_e$ fixes an edge of $T_1$. Otherwise, 
	 $e\in Y_v$ for some vertex $v\in V(T_1)$, so $G_e$ fixes an edge in $T_2$, and
	Assertion (3) holds.

	(2) implies one direction of (4). To prove the non-trivial direction of Assertion (4), assume that $H$ is elliptic in $T_1$ and
	$T_2$. Then $H \leq G_v$ for some $v \in T_1$, so $H$ preserves the subtree $Y_v \subset \widehat T_
	1$. So it is enough to prove that $H$ fixes a point in $Y_v$. This holds because
	$H$ is elliptic in $T_2$ (see \cite[Theorem 2.13(b)]{ZM-88}).

	\end{proof}
	
	\begin{remark}\label{cofinite second tree} If $T_2/G$ is finite, one may think of this construction in terms of graphs of pro-$p$ groups, as follows. Starting from the graph of pro-$p$ group $(\G,\Gamma_1=T_1/G)$, one replaces each vertex $v\in \Gamma_1$ by the graph of groups $(\G_v,\Gamma_v=Y_v/G)$ obtained from  the action of $G_v$ on its minimal subtree in $T_2$, and one attaches each edge $e$ of $\Gamma_1$ incident to $v$ to a vertex of $\Gamma_v$ whose group contains a conjugate of $G_e$. \end{remark}

	\begin{definition} (Standard refinement). Any tree $\widehat T_1$ as in Proposition \ref{refinement} will be called a standard refinement of $T_1$ dominating $T_2$.\end{definition}

In general, there is no uniqueness of standard refinements. However, by Assertion (4) of Proposition \ref{refinement}, all standard refinements belong to the same deformation space, which is the lowest deformation space dominating the deformation spaces containing $T_1$ and $T_2$ respectively. If $T_1$ dominates $T_2$ (resp. $T_2$ dominates $T_1$), then $\widehat T_1$ is in the same deformation space as $T_2$ (resp. $T_1$). Moreover, there is some symmetry: if $T_2$ also happens to be elliptic with respect to $T_1$, then any standard refinement $\widehat T_2$ of $T_2$ dominating $T_1$ is in the same deformation space as $T_1$.

\bigskip
In Proposition \ref{refinement} we assumed that $T_1/G$ is finite. If $G$ is finitely generated we can generalize the proposition for general case on the cost of changing $T_1$ within deformation space without changing the edge stabilizers.

\begin{proposition}\label{refinement for finitely generated G} Let $T_1$ and $T_2$ be pro-$p$ $G$-trees with $G$ finitely generated. Suppose $T_1$ is elliptic with respect to $T_2$.  Then there exists a $G$-tree $S$ from the  deformation space of $T_1$ and a pro-$p$ tree $\widehat S$ with  a morphism $p:\widehat S\longrightarrow S$  such that
	
	\begin{enumerate}
	\item[(0)] stabilizers of edges of $S$ are stabilizers of certain edges of $T_1$ and edge stabilizers of $T_1$ are subgroups of some edge stabilizers of $S$. 
		\item[(1)] $p$ is a collapse map and the subtree that collapsed by $p$ to $v$ is the minimal $G_v$-invariant subtree of $T_2$.
		\item[(2)] $\hat S$ is a refinement of $S$ that dominates $T_2$.
		\item[(3)] the stabilizer of any edge of $\hat S$ fixes an edge in $T_1$ or in $T_2$.
		\item[(4)]  a subgroup of $G$ is elliptic in $\hat S$ if and only if it is elliptic in both $T_1$ and $T_2$.
		\item[(5)] if $T_2$ is minimal, every edge stabilizer of $T_2$ contains an edge stabilizer of $\hat S$.
		\end{enumerate}

\end{proposition}

\begin{proof} 
Let $U$ be a linearly ordered by inclusion system of open normal subgroups $U$ of $G$ with $\bigcap_U U=1$. Let $\widetilde U$ be the (normal) subgroup of $G$ generated by all $U$-stabilizers of vertices in $T_2$. Put $G_U=G/\widetilde U$ and $T_U=T_2/\widetilde U$. Then $$G=\varprojlim_{U\triangleleft_o G} G_U\  {\rm and}\   T_2=\varprojlim_{U\triangleleft_o G}T_U.$$ By Proposition \ref{mod tilde} $T_U$ is a pro-$p$ tree on which $G_U$ naturally acts with $T_U/G_U=T_2/G:=\Gamma$. 

By Theorem \ref{pro-pbass-serre}  $G$ is the fundamental pro-$p$ group of a reduced profinite graph of	pro-$p$ groups $(\G, \Delta)$ whose edge and vertex groups are stabilizers of edges and vertices of $T_1$ and stabilizers of vertices and edges in $T_1$ are contained in  stabilizers of some vertex and edge subgroups of $(\G,\Gamma)$ up to conjugation. Moreover (see the proof of Theorem \ref{pro-pbass-serre}), $(\G,\Delta)=\varprojlim_U (\G_U,\Gamma_U)$ is the inverse limit of the fnite reduced graphs of finite $p$-groups where the maps of the inverse system  $\{\Gamma_U\}$ are collapses.   Let $S=S(\G,\Delta)$ be the standard pro-$p$ tree. Then   $T_1$, $S$ are in the same deformation space, so  (0) holds.  

\medskip	
	Let $S_U$ be the standard pro-$p$ tree of $(\G_U,\Gamma_U)$. Then   $S=\varprojlim_{U\triangleleft_o G} S_U$  and the projection $S\longrightarrow S_U$ is a collapse of connected components of certain profinite subgraph of $S$.

	Note that  $S_U$ is elliptic with respect to $T_U$ since $T_1$ is elliptic with respect to $T_2$ and therefore  $S$ is elliptic with respect to $T_2$. Hence we can consider a standard refinement $\hat S_U$ of $S_U$ dominating $T_U$ as in Proposition \ref{refinement}. In particular, $p_U:\hat S_U\longrightarrow S_U$ is the collapse map and the pro-$p$ subtrees that are collapsed by $p_U$ to $v_U$ are the minimal invariant subtrees of $G_{v_U}$ for each non-elliptic $G_{v_U}$ in $T_U$. Since $U$ is linearly ordered, we can make  $\hat S_U$ to form an inverse system by attaching every edge $e_{v_V}$ incident to $v_V\in S_V$ (for $V\leq U$) to a vertex of $T_V$ whose image is exactly the vertex of $T_U$ to which the image $e_U$ of $e_{v_V}$ in $S_U$ was attached. Thus  $\hat S=\varprojlim_{U\triangleleft_o G} \hat S_U$ is a pro-$p$ tree (by \cite[Lemma 2.6 (b)]{RZ-00} that dominates $T_2$. Moreover, the natural map $p:\widehat S\longrightarrow S$ (the projective limit of the collapse maps $p_U:\widehat S_U\longrightarrow S_U$ (cf. Proposition \ref{refinement}) is the collapse map such that the subtrees that collapsed by $p$ to $v$ is the minimal $G_v$-invariant subtree of $T_2$). Thus (1) holds. 

\medskip	
	Let us check that the other properties follow. Assertion (2) is clear.

\medskip	
	If $e$ is an edge of $\hat S$ that is not collapsed by $p$, then $G_e$  fixes the edge in $S$ and therefore in $T_1$. Otherwise,  $e$ is collapsed by $p$ and so belongs to $T_2$. Hence  $G_{e}$ fixes an edge of $T_2$. Thus
	assertion (3) holds.

\medskip 	(2) implies one direction of (4). To prove the non-trivial direction of Assertion (4), assume that $H$ is elliptic in $T_1$ and
	$T_2$. Then $H \leq G_{v}$ for some $v \in T_1$, so  $H \leq G_{v}$ for some $v\in V(S)$ and therefore $H\leq G_{v_U}$ for every $U$, where $v_U$ is the image of $v$ in $S_U$. Hence $H$ preserves the subtree $Y_{Uv} \subset \widehat S_
	U$. So it is enough to prove that $H$ fixes a point in $Y_{Uv}$ for every $U$. This holds because
	$H$ is elliptic in $T_U=T_2/\widetilde U$.
	
	\medskip
	(5) Pick $e\in E(T_2)$. By \cite[Theorem 4.2]{CZ}  $G$ splits as an amalgamated free pro-$p$ product $G_1\amalg_{G_e} G_2$ or as HNN-extension $G=G_1\amalg_{G_e}$ over $G_e$. As  $\hat S$ dominates $T_2$ for any edge $e'\in E(\hat S)$  the stabilizers $G_v,G_w$ of its vertices $v,w$ stabilize a vertex in $T_2$ and hence are conjugate to $G_1$ or $G_2$, say $G_v$ is conjugate to $G_1$ and so we may assume that $G_v\leq G_1$. If $G_{e'}$ is not in a conjugate of $G_e$, then as $G_{e'}=G_v\cap G_w$, one has that $G_w$ also must be contained in $G_1$ by \cite[Corollary 7.1.5]{R}. If this is the case for all $e'$ then all vertex stabilizers of $\hat S$ are contained in $G_1$. This contradicts the minimality of $T_2$.
	
	\end{proof}


\section{Universally elliptic pro-$p$ trees}

\begin{definition}
 A  subgroup of  $G$ is called universally $\E$-elliptic if it is elliptic  in every $\E$-tree on which $G$ acts. In particular, the trivial group is universally elliptic.
\end{definition}

Note that by Corollary \ref{fixed vertex} every finite $p$-group is universally elliptic. Of course, if a pro-$p$ group  does not act on a pro-$p$ tree without global fixed point (such groups called FA-groups by Serre \cite{Serre} and so we are going to use this term in the pro-$p$ context) then it is universally elliptic. 
The class of FA pro-$p$ groups is quite large and includes many  example. All Fab pro-$p$ groups, i.e pro-$p$ groups whose open subgroups have finite abelianization are FA pro-$p$ groups. Note that Fab pro-$p$ groups include all just-infinite pro-$p$ groups  and all open pro-$p$ subgroups of $SL_n(\Z_p)$.  The pro-$p$ completion of Grigorchuk, Gupta-Sidki groups and other branch groups are FA pro-$p$ groups as well as the Nottingham pro-$p$ group.

\begin{definition} A $G$-tree is called universally $\E$-elliptic if it is elliptic  in every $\E$-tree on which $G$ acts
(i.e. every edge stabilizer  is elliptic in every $\E$-tree).\end{definition}

\begin{lemma}\label{refinement univer elliptic} Consider two $G$-trees $T_1$, $T_2$. Suppose that $T_1/G$ is finite (resp. $G$ is finitely generated).
\begin{enumerate}
\item[(1)] If $T_1$ is universally elliptic, then some refinement of $T_1$ (resp. some refinement of some $S$ from the deformation space of $T_1$) dominates $T_2$.
\item[(2)] If $T_1$ and $T_2$ are universally elliptic, any standard refinement $\hat T_1$ of $T_1$ (resp. $\hat S$ of some $S$ from the deformation space of $T_1$) dominating $T_2$
is universally elliptic. In particular, there is a universally elliptic tree $\hat T_
1$ ( resp. $\hat S$) dominating
both $T_1$ and $T_2$.\end{enumerate}\end{lemma}

\begin{proof} Both assertions follow directly from Assertions (2) and (3) of Proposition \ref{refinement} (resp. \ref{refinement for finitely generated G}).  	\end{proof}



	For many purposes, it is enough to consider one-edge splittings, i.e. trees with only one
orbit of edges.

\begin{lemma}\label{one-edge splitting} Let $G$ be a finitely generated   pro-$p$ group acting on  a pro-$p$ tree $S$ with edge stabilizers in $\E$.
	
	\begin{enumerate}
\item[(1)] $S$ is universally elliptic if and only if it is elliptic with respect to every one-edge
splitting.

\item[(2)] $S$ dominates every universally elliptic tree if and only if it dominates every universally
elliptic one-edge splitting.
\end{enumerate}
\end{lemma}

\begin{proof} For the non-trivial direction, one needs to prove that $S$ is elliptic with respect to a $(G,\E)$-tree $T$ (resp.
dominates $T$). Let $e\in E(S)$ be an edge (resp. $v\in V(S)$ be a vertex). We want to prove that $G_e$  (resp. $G_v$) is elliptic in $T$.

Let $U$ be an open normal subgroup of $G$  and $\widetilde U=\langle U_v\mid v\in V(T)\rangle$. 
By Lemma \ref{inverse limit of virtually free groups} $G=\varprojlim_{U\triangleleft_o G} G/\tilde U$ and $G_U=G/\tilde U=\Pi_1(\G_U,\Gamma_U)$ is the fundamental group of a finite reduced graph of  finite $p$-groups. Moreover, the inverse system $\{G_U, \pi_{VU}\}$ can be chosen in such a way that it is linearly ordered and for each $\{G_V\}$ of the system with $V\leq U$  there exists a natural morphism $(\eta_{VU},\nu_{VU}):(\G_V, \Gamma_V)\longrightarrow (\G_U, \Gamma_U)$ where $\nu_{VU}$ is just a collapse of edges of $\Gamma_V$ and $\eta_{VU}(\G_V(m))=\pi_{VU}(\G_V(m))$; the induced  homomorphism of the
pro-$p$ fundamental groups coincides with the canonical projection
$\pi_{VU}\colon G_V\longrightarrow G_U$. Note also that $T_U=T/\widetilde U$ is a pro-$p$ tree by Proposition \ref{mod tilde} and the vertex stabilizers of $G_U$ acing $T_U$ are finite. So $T_U$ is in the same deformation space as the standard pro-$p$ tree $S_U$ of $\Pi_1(\G_U,\Gamma_U)$. Thus the projective limit argument reduces the proof to  $G_U=\Pi_1(\G_U,\Gamma_U)$ acting on $S_U$.

We argue by induction on the number of  edges of $\Gamma_U$ to deduce that $G_e$ (resp. $G_v$) is conjugate into a vertex group of $(\G_U,\Gamma_U)$. Choose a splitting of $G_U=\Pi_1(\G_U, \Delta_U)\amalg_{(G_U)_{e_U}}$ for some $e_U\in E(\Gamma_U)$, where $\Delta_U$ is a proper connected subgraph of $\Gamma_U$. As $G_e$ (resp. $G_v$) is elliptic with respect to this splitting (by hypothesis) we may assume w.l.o.g. that $G_e\leq \Pi_1(\G_U, \Delta_U)$ (resp. $G_v\leq \Pi_1(\G_U, \Delta_U)$. Then by the induction hypothesis $G_e$  (resp. $G_v$) is conjugate into some vertex group of $ \Pi_1(\G_U, \Delta_U)$ as needed.   \end{proof}

\begin{lemma}\label{dominant splitting} Let $G$ be a pro-$p$ group and $T_1$,$T_2$ be pro-$p$ trees on which $G$ acts.
	\begin{enumerate}
\item[(1)] If $T_1$ refines $T_2$ and does not belong to the same deformation space, some $g \in G$ is
hyperbolic in $T_1$ and elliptic in $T_2$.
\item[(2)] If every $g\in G$ which is elliptic in $T_1$ is also
elliptic in $T_2$, then $T_1$ dominates $T_2$.
\item[(3)] Suppose $G$ is finitely generated and $T_2$ is minimal. If $T_1$ is elliptic with respect to $T_2$, but $T_2$ is not elliptic with respect to $T_1$, then $G$
splits over an infinite index subgroup of  an edge stabilizer of $T_2$.
\end{enumerate}
\end{lemma}
\begin{proof}  (1) Since $T_1$ refines $T_2$, every elliptic element in $T_1$ is elliptic in $T_2$. If every elliptic element in $T_2$ is elliptic in $T_1$ then $T_1$ and $T_2$ have the same set of elliptic elements and so belong to the same deformation space. Therefore  some hyperbolic element of $T_1$ becomes elliptic in $T_2$. 

\smallskip
(2) Let $v$ be a vertex of $T_1$. We need to show that $G_v$ is elliptic in $T_2$, but this follows from Proposition \ref{ellipticity}.

For (3)  let $G_e$ be
an edge stabilizer of $T_2$ which is not elliptic in $T_1$ and let $\hat S$ be the standard refinement from Proposition \ref{refinement for finitely generated G}. By   Proposition \ref{refinement for finitely generated G}(5) $G_e$ contains an edge stabilizer $G_{e_1}$  of $\hat S$ that has infinite index in $G_e$ since it is elliptic in $T_1$ and $G_e$ is not (see Corollary \ref{fixed vertex}).   By  Theorem \ref{splitting}  $G$ splits as a free pro-$p$ product with amalgamation or as a pro-$p$  HNN-extension over $G_{e_1}$.  \end{proof}



\begin{example}\label{changing jsj-tree}  Let $G$ be a finitely generated pro-$p$ group acting on a pro-$p$ tree $T$.  By Theorem \ref{pro-pbass-serre}  $G$ is the fundamental pro-$p$ group of a profinite graph of
	pro-$p$ groups $(\G, \Gamma)$.  Let $S=S(\G,\Gamma)$ be the standard pro-$p$ tree. Then $T$ and $S$ are in the same deformation space. Moreover, $S$ is universally elliptic if $T$ is.  The set of edges of $S$ however  is not compact in general; it is compact if and only if $\Gamma$ is finite. This is the case if $(G,T)$ is an $\E$-splitting and $G$ is $\E$-accessible.

	Note that $(\G, \Gamma)$ is not unique. But if $\Gamma$ is finite, then there exists a finite sequence of elementary transformations: reductions and expansions (expansion is the inverse operation to the reduction), that transforms one reduced graph of groups into another.  To see this first note that we may assume that $(\G, \Gamma)$ is reduced and the vertex groups of $(\G,\Gamma)$ are exactly the maximal vertex stabilizers of $G$ acting on $T$ and so there are finitely many of them $G_{v_1},\ldots G_{v_n}$. Therefore there exists an open normal subgroup $U$  in $G$ such that $G_{v_1}\widetilde U/\widetilde U,\ldots,
	 G_{v_n}\widetilde U/\widetilde U$ are precisely maximal finite subgroups of $G/\widetilde U$ up to conjugation and hence the vertex groups of $(\G_U,\Gamma)$ for any choice of  $(\G_U,\Gamma)$.   By \cite[Theorem 4.2]{F} the exists a finite sequence of reductions and expansions that transform one choice to the other, and it does not depend on the chosen $U$, since the vertex groups are finite. Then the same finite sequence of reductions and expansions can be applied to $(\G, \Gamma)$ using \cite[Corollary 4.4]{CZ} and its proof.
	
	\end{example}

\section{JSJ splittings}

 \begin{definition} A JSJ decomposition (or JSJ tree) of a pro-$p$ group $G$ over $\E$ is an $\E$-tree T on which $G$ acts such that:
 	
 	$\bullet$ $ T$ is universally elliptic;
 
 $\bullet$ $T$ dominates any other universally elliptic  $\E$-tree $T'$ on which $G$ acts, i.e. the stabilizer of any vertex  $v$ of $T$ stabilizes some vertex of $T'$.
 
 \smallskip	
 	If $T/G$ is finite, then $G=(\G,T/G)$ is the fundamental pro-$p$ group of a finite graph of pro-$p$ groups (see Remark \ref{cofinite action}). 
 	We call the quotient graph of groups $(\G,T/G)$
 	a JSJ splitting, or a JSJ decomposition. Also, if the standard pro-$p$ tree of the fundamental group of a profinite graph of pro-$p$ groups $(\G, \Gamma)$  is a JSJ $G$-tree, then we call  $(\G, \Gamma)$  a JSJ splitting, or a JSJ decomposition.
 	
 	\end{definition}
 	
 	\begin{remark}\label{reduced JSJ} 
 \begin{enumerate}	
\item[(i)] The second condition in the definition is a
maximality condition expressing that vertex stabilizers of $T$ are as small as possible by Lemma \ref{refinement univer elliptic}.

\item[(ii)] 	
 	If $(\G,\Gamma)$ is a JSJ-splitting with $|\Gamma|< \infty$, then after the reduction from  procedure described in Remark \ref{reduction} we still have a JSJ-splitting. Indeed, the edge groups continue to be universally elliptic, and maximal vertex stabilizers do not change, so  the obtained graph of groups still dominates any other JSJ-splitting.\end{enumerate}
 	
 	\end{remark}

 	\begin{lemma}\label{universally elliptic vertex} Any  pro-$p$ $
	\E$-tree $T$ with universally elliptic vertex stabilizers is a JSJ tree.
\end{lemma}

\begin{proof}  By assumption,  $T $ dominates every pro-$p$ $\E$-tree. In particular, $T$ is universally elliptic and
dominates every universally elliptic $\E$-tree, so it is a JSJ tree.\end{proof}

\begin{definition} (Rigid and flexible vertices). Let $H = G_v$ be a vertex stabilizer of a JSJ
pro-$p$ tree $T$ (or a vertex group of the graph of groups $(\G, \Gamma = T /G)$). We say that $H$ is rigid if it is
universally elliptic, flexible if it is not. We also say that the vertex $v$ is rigid (flexible) in this case. If
$H$ is flexible, we say that it is a flexible subgroup of $G$ (over $\E$).
\end{definition}

If $T$ is a JSJ pro-$p$ $\E$-tree, another pro-$p$ $\E$-tree $T_
0$
is a JSJ pro-$p$ tree if and only if $T_
0$ 
is universally elliptic, $T$
dominates $T_
0$ 
, and $T_
0$ dominates $T$. In other words, $T_
0$
should belong to the deformation
space of $T$ over $\E_{ell}$, where $\E_{ell}$ is the family of universally elliptic groups in $\E$.

\begin{definition} If non-empty, the set of all JSJ pro-$p$ $\E$-trees is a deformation space over $\E_{ell}$  called
the JSJ deformation space (of $G$ over $\E$). We denote it by $D_{JSJ}$.\end{definition}

\begin{remark} Let $T$ be a JSJ $\E$-tree for $G$. Then the pro-$p$ tree $S$ from Example \ref{changing jsj-tree} belongs to $D_{JSJ}$. In fact, if $\E$ is closed for finite quotients then generalizing  \cite[Definition]{wilson} of ultra filter  we may see that $S$ is the JSJ-decomposition obtained from inverse limit of ultra filter of virtually free quotients of $G$. Indeed, for any open normal subgroup $U$ of $G$ the group $G/\widetilde U$ acts on $T/\widetilde U$. Moreover,  $T_1$ dominates $T_2$ if and only if for any $T_2/\widetilde U$ there exists open normal $V$ contained in $U$ such that $T_1/\widetilde V$ dominates $T_2/\widetilde U$ and this is a definition of $\{T_1/\widetilde U\mid U\triangleleft_o G\}$ being bigger than $\{T_2/\widetilde U\mid U\triangleleft_o G\}$. Thus for $\{D/\widetilde U\mid U\triangleleft_o G\}$ (where $D$ ranges via all pro-$p$ trees dominated by $T$), $\{T/\widetilde U\mid U\triangleleft_o G\}$ is a maximal filter, i.e. an ultra filter. 

\end{remark}

\bigskip
\begin{remark}\label{action on deformation space}
If $\E$ is invariant under a group $A$ of automorphisms of $G$ (in particular if $\E$ is defined
by restricting the isomorphism type), then so is the deformation space $D_{JSJ}$. Indeed, for $T\in D_{JSJ}$ and $\alpha\in A$ we can define  a $G$-tree $T_\alpha$ with $T_\alpha=T$ and the action $g\cdot t=\alpha(g) t$.  Then the stabilizer of  $e\in T_\alpha=T$  is $\alpha^{-1}(G_e)\in\E$.
\end{remark}
Using this we can deduce the following 

\begin{theorem}\label{subida} Let $G$ be a pro-$p$ group having an open normal $\E$-accessible subgroup $H$ acting on an infinite  JSJ pro-$p$ $\E$-tree $T$. Suppose $\E$ is $G$-invariant. Then $G$ acts on an infinite  pro-$p$ $\E'$-tree $S$, where $\E'$ consists of subgroups $A$ of $G$ such that $A\cap H\in\E$. Moreover, 
$S$ is universally elliptic pro-$p$ $\E'$-tree.

\end{theorem}

\begin{proof} By Remark \ref{action on deformation space} $G$  acts on the deformation space $D_{JSJ}$ for $H$. Write $H=\Pi_1(\curlyH, \Delta)$ as the fundamental group of the corresponding reduced finite graph of pro-$p$ groups. It follows  then that $G$ leaves the conjugacy classes of all  vertex groups  of $(\curlyH, \Delta)$ invariant. By \cite[Theorem 8]{mattheus}   $G$ is the fundamental group of a reduced finite graph of pro-$p$ groups $(\G, \Gamma)$  such that its vertex and edge groups intersected with $H$ are subgroups of vertex and edge groups of $(\curlyH, \Delta)$. Thus   the standard pro-$p$ tree  $S$ of $(\G, \Gamma)$ is $(G, \E')$-tree. By Corollary \ref{fixed vertex} it is universally elliptic.



\end{proof}

Of course, if $\curlyH(v), v\in V(\Gamma)$  are universally elliptic, then Theorem \ref{amalgam} holds independently of $\E$ as it was used in the proof the $G$-invariancy of vertex stabilizers only. 

\begin{corollary}  Let $G$ be a pro-$p$ group having an open   subgroup $H$ acting on an infinite  JSJ pro-$p$ $\E$-tree $T$  whose vertex stabilizers are universally elliptic. Then $G$ acts on an infinite  JSJ pro-$p$ $\E'$-tree $S$, where $\E'$ consists of subgroups $A$ of $G$ such that $A\cap H\in\E$. 
\end{corollary}

\begin{corollary}  Let $G$ be a pro-$p$ group having an open   subgroup $H$ acting on an infinite  JSJ pro-$p$ $\E$-tree $T$  whose vertex stabilizers are FA. Then $G$ acts on an infinite  JSJ pro-$p$ $\E'$-tree $S$, where $\E'$ consists of subgroups $A$ of $G$ such that $A\cap H\in\E$. 
\end{corollary}

\begin{lemma} Assume $G$ is finitely generated and  all groups in $\E$ are universally elliptic. If $T$ is a JSJ-tree, then
	its finitely generated vertex stabilizers are universally elliptic.
	This applies in particular to splittings over finite groups.\end{lemma}

\begin{proof}  By Theorem \ref{pro-pbass-serre} $G$ is the fundamental group $\Pi_1(\G,\Gamma)$ of a profinite graph of pro-$p$ groups over $\E$ and by Example \ref{changing jsj-tree} its standard pro-$p$ tree is universally elliptic and so replacing $T$ by the standard pro-$p$ tree $S$ of $(\G, \Gamma)$ we may assume that $T$ is the standard pro-$p$ tree for this graph of groups.  If a vertex group $G(v)$ of $v$ is flexible, then there exists  a pro-$p$ tree $T_
0$
such that $G(v)$ is not elliptic in
$T_0$. Then by Theorem \ref{splitting}  $G(v)$  splits either as $G_0\ast_A G_1$ or $G_0\ast_A$, where $A\in \E$. As all edge groups are universally elliptic by hypothesis,  the edge groups of all adjacent edges to $v$ are conjugate into either $G_0$ or $G_1$ (see Corollary \ref{fixed vertex}).  Let $(\G,\hat\Gamma)$ be the standard refinement of $(\G,\Gamma)$ at $v$ (see Proposition \ref{refinement for finitely generated G}). Its standard pro-$p$ tree $T'$,
by our assumption on $A$,  is universally
elliptic, so by definition of the JSJ deformation space $T$ dominates $T'$. This implies that
$G_v$ is elliptic in $T'$, hence in $T_0$ (see Proposition \ref{refinement for finitely generated G}(2)), a contradiction.\end{proof}

\begin{theorem}\label{existence} Let   $\E$ be a continuous family of subgroups of a pro-$p$ group $G$ and suppose that $G$  is $\E$-accessible. Then JSJ pro-$p$ tree $T$ exists.

\end{theorem}

\begin{proof} Let $(G,\Gamma)$ be a maximal reduced universally elliptic $\E$-splitting  of $G$  and $T(G)$ its standard pro-$p$ tree (it might be trivial, i.e. be just a vertex). Note that $\Gamma=T/G$ is finite since $G$ is $\E$-accessible.
We show that  $(G,\Gamma)$ is JSJ-decomposition. 

  Let $T'$ be another universally elliptic pro-$p$ $\E$-tree for $G$. It suffices  to show that the stabilizer  $G_v$ of any vertex $v$ of  $T$ is elliptic in $T'$ (indeed this would mean that edge groups also stabilize a vertex in $T'$ and that $T$ dominates $T'$). 
  
  If not, then $G$ splits as an amalgamated free pro-$p$ product or pro-$p$ HNN-extension over some  edge stabilizer $A$ in $T'$ with not elliptic $G(v)$, say $G=G_1\amalg_A G_2$ (remember that $G$ is accessible, so $T'/G$ is finte).  Let $S(G)$ be a standard pro-$p$ tree with respect to this splitting. Since $T$ is universally elliptic every edge stabilizer in $T$ is elliptic in $S(G)$. It follows that the stabilizer of an edge incident to $v$
is in $G_1$ or in $G_2$ up to conjugation. Let $(\G, \hat\Gamma)$ be the standard refinement of $(\G,\Gamma)$ at $v$ (see Remark \ref{cofinite second tree}). Then it is an $\E$-splitting of $G$ dominating but not equal $(\G,\Gamma)$.  This contradicts the maximality of the decomposition. 	

\end{proof}

\begin{corollary}\label{finiteness} Let $G$ be a pro-$p$ group and $\E$ a continuous  family of subgroups of $G$. If $G$ is $\E$-accessible, then there exists JSJ pro-$p$ tree $T$ such that $T/G$ is finite. 
\end{corollary}

In particular, using $k$-acylindricity and cyclic accessibility (cf. Theorem \ref{k-acylindrical accessibility} and Theorem \ref{cyclic accessibility}) we deduce 

\begin{corollary} 
\begin{enumerate}

\item[(i)] A $k$-acylindrical JSJ-decomposition  of  a finitely generated pro-$p$ group exists.

\item[(ii)] Suppose $\E$ consists of cyclic groups. Then JSJ-decomposition exists.
\end{enumerate}

\end{corollary}

\medskip
We show now that in contrast to the abstract situation a JSJ-splitting exists for finitely generated pro-$p$ groups  regardless of $\E$-accessibility. However the JSJ pro-$p$ tree might not have compact set of edges.

\begin{theorem}\label{existence for finitely generated G} Let   $\E$ be a continuous family of subgroups of a finitely generated pro-$p$ group $G$. Then JSJ pro-$p$ $\E$-tree exists.

\end{theorem}

\begin{proof} 
Let $T_i$ be the set of all universally elliptic pro-$p$ $\E$-trees on which $G$ acts.  By Example \ref{changing jsj-tree} there exists a  universally elliptic pro-$p$ tree $S_i$ in the same deformation space as $T_i$.  We claim that $\{S_i\mid i\in I\}$ is filtered. Indeed, if say, $S_i$  does not dominate  $S_j$ then by Proposition \ref{refinement for finitely generated G} there exists a standard refinement $\hat S_i$ that dominates both $S_i$ and $S_j$. Now let $S=\varprojlim_i S_i$. We claim that $S$ dominates every universally elliptic $\E$-tree. But this is clear, since $S$ dominates every $S_i$ and hence every $T_i$.

  Now we proceed as in the proof of Theorem \ref{existence} to show that $S$ is a JSJ $\E$-tree. Let $T'$ be another universally elliptic pro-$p$ $\E$-tree for $G$. It suffices  to show that any vertex stabilizer  $G_v$ of $T$ is elliptic in $T'$. 
  
  If not, then $G_v$ acts on $T'$ non-trivially. Since $S$ is universally elliptic every edge stabilizer in $S$ is elliptic in $T'$. It follows that every edge group of an edge incident to $v$
is elliptic in $T'$. Let $ \hat S$ be the standard refinement of $S$ at $v$ (cf. Proposition \ref{refinement for finitely generated G}  and \ref{cofinite second tree}). By Proposition \ref{refinement for finitely generated G}  (3) $\hat S$ is universally elliptic, since $S$ and $T'$ are.  Then $\hat S$  dominates $S$ ( by Proposition \ref{refinement for finitely generated G}(2)), but $S$ does not dominate  $\hat S$.  This contradicts the definition of $S$. 	

\end{proof}

We illustrate this theorem in Example \ref{Wilkes} in the next section and 
 finish this section with the following simple facts for future reference.

\begin{lemma}\label{JSJ refinement} Let $G$ be a pro-$p$ group, $T$ be a JSJ pro-$p$ tree, and $S$ any $G$-tree. Suppose either $G$ is finitely generated or $T/G$ is finite.

\begin{enumerate}
\item[(i)] There is a tree $\hat T$  which refines $T$ and dominates $S$;

\item[(ii)] If $S$ is universally elliptic, it may be refined to a JSJ pro-$p$ tree.
\end{enumerate}
\end{lemma}

\begin{proof} Since $T$ is elliptic with respect to $S$, one can construct a standard refinement $\hat T$ of
$T$ dominating $S$ (Propositions \ref{refinement}, \ref{refinement for finitely generated G}). It satisfies the first assertion.
For the second assertion, since $S$ is elliptic with respect to $T$, we can consider a standard
refinement $\hat S$ of $S$ dominating $T$. It is universally elliptic by the second assertion of Lemma
\ref{refinement univer elliptic}, and dominates $T$, so it is a JSJ tree.\end{proof}

\section{Examples of JSJ-splittings}

Recall that we have fixed $\E$ and $\K$, and we only consider $(\E,\K)$-trees. Unless otherwise indicated, $G$ is only assumed to be finitely generated in this section. At the end of this section we shall give two examples of JSJ decompositions having flexible vertices, but most examples here will have all vertices rigid. The fact that they are indeed JSJ decompositions will be a consequence of Lemma \ref{universally elliptic vertex}.

\subsection{  Free pro-$p$ groups} Let $G=F_n$ be a finitely generated free pro-$p$ group, let $\E$ be arbitrary, and $\K=\emptyset$. Then the JSJ deformation space of $F_n$ over $\E$ is the space of pro-$p$ trees on which $F_n$ acts freely. Note that if we assume that $T/F_n$ is reduced then it is just a bouquet and $T$ is the Cayley graph of $F_n$.

\bigskip

\subsection{Virtually free pro-$p$ groups}
 
 More generally, if $G$ is virtually free pro-$p$, by the main result of \cite{HZ-13} $G$ splits as a the fundamental pro-$p$ group of a finite graph of finite $p$-groups. Thus if  $\E$ consists of all finite subgroups, then $D_{JSJ}$ is the space of trees with finite vertex stabilizers.  
 
 Recall that there exists a finite sequence of elementary transformations: reductions and expansions (expansion is the inverse operation to the reduction), that transforms one JSJ splitting  into another (cf. Example \ref{changing jsj-tree}). 
 
 \medskip
 Note that there is the following open 
 
 \begin{question} Is a finitely generated pro-$p$ group $G$ acting on a pro-$p$ tree $T$ with finite vertex stabilizers  accessible  and hence is virtually free pro-$p$? \end{question}
 
 \subsection{Free splittings} This is the pro-$p$ version of the Grusko JSJ-decomposition.
 
 Suppose  $\E$ consists only of the trivial subgroup of $G$, and $\K=\emptyset$. Thus $\E$-trees are trees with trivial edge stabilizers, also called free splittings. Then the JSJ deformation space exists. We call it the Grushko deformation space. It consists of trees $T$ such that edge stabilizers are trivial, and vertex stabilizers are freely indecomposable and different from $\Z_p$ (one often considers $\Z_p$ as freely decomposable since it splits as an HNN-extension over the trivial group). Denoting by $G=G_1\amalg ···\amalg G_n\amalg F_m$  a decomposition of $ G$ such that  $G_i\neq \Z_p$ are indecomposable, and $F_m$ is free pro-$p$ (such splitting always exists for a finitely generated pro-$p$ group $G$), the quotient graph of groups $T/G$  has one vertex with group $G_i$ for each $i$, and all other vertex groups are trivial. Below is the picture of the Grushko JSJ-decomposition when $n=3$, $m=2$.
 
 $$
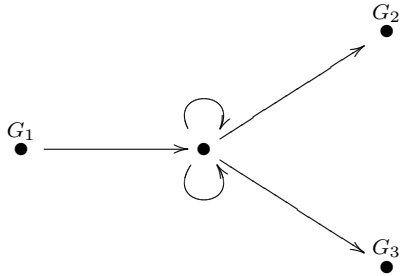
\xymatrix{&&&& {\overset{G_2 }{\bullet}}\\
{\overset{G_1}{\bullet}}\ar[rr]&& {\overset{}{\bullet}}\ar[rru] \ar[rrd]\ar@(dl,dr)\ar@(ul,ur)\\  
&&&&{\overset{G_3}{\bullet}}}$$

Note that this graph of groups is not reduced, it admits a collapse of one edge; after this it becomes reduced. 

\medskip 
 If $\K\neq \emptyset$, the JSJ deformation space is the Grushko deformation space relative to $K$. Edge stabilizers of JSJ trees are trivial, groups in $\K$ fix a vertex, and vertex stabilizers are freely indecomposable relative to their subgroups which are conjugate to a group in $\K$.
 
 \begin{proposition}\label{relative Grushko} Let $G$ be a second countable pro-$p$ group and $G=G_1\amalg ···\amalg G_n\amalg F_m$ be a Grushko decomposition relative to $\K$. Any finitely generated $\amalg$-indecomposable relative to $\K$ subgroup A of H is conjugate to a
subgroup of $G_i$ for some $i$. Moreover, if $G = A \amalg B$ for some closed
subgroup $B$ of $G$, then $A$ is conjugate to some $G_i$.\end{proposition}

\begin{proof} Follows from the pro-$p$ version
of the Kurosh subgroup theorem (see Theorem \ref{KST}).

\end{proof}
 
\subsection{ Splittings over finite groups} 

This is a pro-$p$ version of Stallings-Dunwoody JSJ-space.

\bigskip
 If $\E$ is the set of finite subgroups of $G$, and $\K=\emptyset$, then the JSJ deformation space is the set of trees whose edge groups are finite and whose vertex groups either finite  or virtually freely indecomposable (one deduces from the main result of \cite{WEZ-16}  that a pro-$p$ tree is maximal for domination if and only if its vertex stabilizers are virtually  indecomposable as a free pro-$p$ product).  If $G$ is  finitely generated, JSJ deformation space exists by Theorem \ref{existence for finitely generated G}.  If $G$ is accessible there exists a cofinite JSJ pro-$p$ tree.

\begin{remark} There is a cofinite JSJ pro-$p$ tree over $\E$ if $\E$ is a family of finite subgroups of bounded order. The reason is Wilkes pro-$p$ version of   Linnell’s accessibility \cite[Theorem 3.1]{wilk}.
\end{remark}

 If $\K\neq \emptyset$, then we talk about JSJ deformation space relative to $\K$. Edge stabilizers are finite, groups in $\K$ fix a vertex, and vertex stabilizers are virtually freely indecomposable  relative to their subgroups which are conjugate to a group in $\K$: they do not split over finite subgroups relative to their subgroups which are conjugate to a group in $\K$. As above, the relative JSJ space exists if $G$ is finitely generated.  If $\E$ consists of finite groups of bounded order (and $K$ is arbitrary) then there exists a cofinite JSJ $\E$-tree. If $\E$ contains finite groups of arbitrary large order,  a cofinite JSJ $\E$-tree exists if relative accessibility holds.
 
\medskip  
 We finish the subsection with  the example of an inaccessible finitely generated pro-$p$ group presented by Wilkes in \cite[Section 4.2]{Wilkes} that illustrates Theorem \ref{existence for finitely generated G}. 

\bigskip	
	\begin{example}\label{Wilkes} First define the map $\mu_n:\{0, . . . , p^{n+1}-1\}\longrightarrow \{0, . . . , p^n-1\}$ by sending an integer to its remainder modulo $p^n$.  Define $H_n=\mathbb{F}_p[\{0, . . . , p^n-1\}]$ to be the $\mathbb{F}_p$-vector space with basis $\{h_0, . . . , h_{{p^n}-1}\}$. There are inclusions $H_n\subseteq H_{n+1}$ given by inclusions of bases, and retractions $\eta_n:H_{n+1}\longrightarrow H_n$ defined by $h_k\rightarrow h_{{\mu_n}(k)}$. Note also that there is a natural action of $\Z/p^n\Z$ on $H_n$ given by cyclic permutation of the basis elements, and that these actions are compatible with the retractions $\eta_i$. The inverse limit of the $H_n$ along these retractions is the completed group ring $H_{\infty}=\mathbb{F}_p[[\Z_p]]$ with multiplication ignored. The continuous action of $\Z_p$ on the given basis of $H_{\infty} $ allows  to form a sort of a pro-$p$ wreath product $H_\omega=\mathbb{F}_p[[\Z_p]]\rtimes \Z_p= \varprojlim (H_n\rtimes \Z/p^n)$ which is a pro-$p$ group into which $H_{\infty}$ embeds.
		
		Next set $K_n=\mathbb{F}_p\times H_n=\langle k_n\rangle\times H_n$. Set $G_1=K_1\times\mathbb{F}_p$. For $n >1$, let $G_n$ be a finite $p$-group with presentation $G_n=\langle k_{n-1}, k_n,h_0, \ldots , h_{p^n-1} \mid k_i^p=h_i^p= 1, h_i\leftrightarrow h_j,k_{n-1}\leftrightarrow h_i\  {\rm for\ all\ } i\neq p^{n-1}$, $k_n= [k_{n-1}, h_{p^{n-1}}]\  {\rm central}\rangle$ where $\leftrightarrow$ denotes the relation ‘commutes with’. 
		
		The choice of generator names describes maps $$H_n\longrightarrow G_n, K_{n-1}\longrightarrow G_n,  K_n\longrightarrow G_n.$$  One may easily see that all three of these maps are injections. Define a retraction map $\rho_n:G_n\longrightarrow K_{n-1}$ by killing $k_n$ and by sending $h_k\rightarrow h_{{\mu_{n-1}}(k)}$. Note that $\rho_n$ is compatible with $\eta_n:H_n\rightarrow H_{n-1}$ that is, there is a commuting diagram
		
		\begin{equation*}
			\begin{tikzcd}
				K_{n-1} \arrow[hook, shift left]{r} & G_n \arrow[shift left]{l}{\rho_n}\\
				H_{n-1} \arrow[hook]{u} \arrow[hook, shift left]{r} & H_n \arrow[hook]{u} \arrow[shift left]{l}{\eta_n}  
			\end{tikzcd}
		\end{equation*}
		
		Define $\Pi_1(\G_m,\Gamma_m,v_m)$ to be the pro-$p$ fundamental group of the following graph of groups:
		
		\begin{center}
			\begin{tikzpicture}[every edge/.append style={nodes={font=\scriptsize}}]
				\node (0) at (0,0) [label=above:{\scriptsize $G_1$},point];
				\node (1) at (2,0) [label=above:{\scriptsize $G_2$},point];
				\node (2) at (3,0) {$\cdots$};
				\node (3) at (4,0) [label=above:{\scriptsize $G_{m-1}$},point];
				\node (4) at (6,0) [label=above:{\scriptsize $G_m$},point];
				
				\draw[->] (0) edge node[above] {$K_1$} (1);
				\draw[->] (3) edge node[above] {$K_{m-1}$} (4);
			\end{tikzpicture}
		\end{center}
		
		Note that the retraction $\rho_n:G_n\longrightarrow K_{n-1}$ induces the retraction $P_{m+1}\longrightarrow P_m$  represented by the collapse the last right edge of the picture. 
		
		Then $P=\varprojlim_{m\in\N} \Pi_1(\G_m,\Gamma_m,v_m)$ is the fundamental group of the following profinite graph of pro-$p$ groups
		
		\begin{center}
			\begin{tikzpicture}[every edge/.append style={nodes={font=\scriptsize}}]
				\node (0) at (0,0) [label=above:{\scriptsize $G_1$},point];
				\node (1) at (2,0) [label=above:{\scriptsize $G_2$},point];
				\node (2) at (4,0) [label=above:{\scriptsize $G_{3}$},point];
				\node (3) at (5,0) {$\cdots$};
				\node (4) at (6,0) [label=above:{\scriptsize $G_{\infty}$},point];
				
				\draw[->] (0) edge node[above] {$K_1$} (1);
				\draw[->] (1) edge node[above] {$K_{2}$} (2);
			\end{tikzpicture}
		\end{center}
		where the vertex at infinity is a one point compactification of the edge set of the graph and so does not have an incident edge to it; thus the edge set is not compact. The vertex group $G_{\infty}$ of the vertex at infinity is $G_\infty=K_{\infty}=\varprojlim_{i\in \N} K_i=H_{\infty}$. Let $J=P\amalg_{H_{\infty}} H_{\omega}$.  Then $J$ is the fundamental group of the following profinite graph of groups
		
		\begin{center}
			\begin{tikzpicture}[every edge/.append style={nodes={font=\scriptsize}}]
				\node (0) at (0,0) [label=above:{\scriptsize $G_1$},point];
				\node (1) at (2,0) [label=above:{\scriptsize $G_2$},point];
				\node (2) at (3,0) {$\cdots$};
				\node (3) at (4,0) [label=above:{\scriptsize $H_{\infty}$},point];
				\node (4) at (6,0) [label=above:{\scriptsize $H_{\omega}$},point];
				
				\draw[->] (0) edge node[above] {$K_1$} (1);
				\draw[->] (3) edge node[above] {$H_{\infty}$} (4);
			\end{tikzpicture}
		\end{center}
		
		By \cite[Section 4.3]{Wilkes}, this graph of pro-$p$ groups is injective and by \cite[Section 4.4]{Wilkes} $J=\langle G_1, H_\omega\rangle$. Since $G_1$ is finite and $H_\omega$ is 2-generated, $J$ is finitely generated (in fact for $p=2$ the group $J$ is 3-generated).  Collapsing the right edge we shall get the reduced graph of pro-$p$ groups  since no vertex group  equals to an edge group of an incident edge. Note that the latter graph of groups has a unique vertex $\infty$ whose vertex group is infinite and isomorphic to $\F_p\wr \Z_p$ which does not split over a finite $p$-group.  This is a JSJ decomposition of $J$ over finite $p$-groups.

	\end{example}
 
   \subsection*{Small groups} Recall that $G$ is small in $(\E,\K)$-trees if  for any such tree $T$ the group $G$ has no free non-abelian pro-$p$ subgroup acting freely on $T$: there always is a fixed point, or the action on the minimal  subtree has cocyclic or codihedral kernel (see \cite[Theorem 3.15]{RZ-00}).This is in particular the case if $G$ is small, i.e. contains no non-abelian free pro-$p$ group. If $ G$ does not fix a vertex, then every vertex stabilizer in the minimal subtree has a subgroup of index at most $2$ fixing an edge (the index is $2$ if $p=2$).

   \begin{lemma} If $G$ is small in $(\E,\K)$-trees, and $T$ is a non-trivial universally elliptic $(\E,\K)$-tree, then $T$ dominates every $(\E,\K)$-tree.\end{lemma}
   
   \begin{proof} Since the action of $G$ on $T$ is non-trivial, every vertex stabilizer contains an edge stabilizer with index at most 2 (by above). It follows that every vertex stabilizer is universally elliptic, since the stabilizer of the edge is (see Proposition \ref{fixed vertex}). So $T$ dominates every $(\E,\K)$-tree. \end{proof}

   \begin{corollary} If $G$ is small in $(\E,\K)$-trees, there is at most one non-trivial deformation space containing a universally elliptic tree.\end{corollary} 
   
   In this situation, the JSJ deformation space always exists: if there is a deformation space as in the corollary, it is the JSJ space; otherwise, the JSJ space is trivial.
   
   \bigskip
  \begin{example}\label{Baumslag-Solitarex} Consider for instance (non-relative) cyclic splittings of solvable pro-$p$ Baumslag-Solitar groups 
   
$$BS(1,n)=\langle a,t\mid tat^{-1}=a^{n},n\not\equiv 0 ({\rm mod}\  p)\rangle\cong \Z_p\rtimes \Z_p,$$
where $n$ is supernatural rather than natural number. 
   If $n= 1$ (so $G\cong \Z_p\times \Z_p$ ), 
    there are infinitely many deformation spaces (corresponding to epimorphisms $G\longrightarrow \Z_p$) and there is no non-trivial universally elliptic tree.
  
  \medskip  
    If $n\neq \pm 1$ there exists a unique deformation space containing  
the Bass-Serre tree $S$ of the HNN extension 
   $\langle a,t\mid tat^{-1}=a^{n}\rangle=HNN(\langle a\rangle,\langle a\rangle,t) $. This pro-$p$ tree is universally elliptic and therefore is JSJ-tree. Indeed, vertex  stabilizers of a non-trivial  $G$-tree $T$ have to be cyclic, since 2-generated subgroups are open in $G$ and if stabilize a vertex then so is $G$ by Corollary \ref{fixed vertex}.   By Theorem \ref{splitting} $G$ splits as a non-trivial (non-fictitious) free pro-$p$ product with amalgamation $G_1\amalg_{G_e} G_2$ or HNN-extension $HNN(G_1, G_e, t)$.  As $G$ is soluble, either  $G=G_1\amalg_{G_e} G_2$ with $[G_1:G_e]= 2=[G_2:G_e]$ or $HNN(G_e, G_e, t')$ . In the first case $G_1, G_2$ are cyclic and so $T/G$ has one edge only,  $p=2$ and  $G/G_e$ is infinite dihedral, which is impossible for $n\neq -1$. Thus  this case of the amalgamation $G=G_1\amalg_{G_e} G_2$ does not occur. 
    
   So $G=HNN(G_e, G_e, t')$ and  $T$ is the Bass-Serre tree of $HNN(\langle a\rangle,\langle a\rangle,t') $, where $t'$ is a unit of $\langle t\rangle\cong \Z_p$. As $\langle a \rangle$ is a unique normal subgroup $N$ of $G$ such that $G/N\cong \Z_p$, $\langle a\rangle=G_e$.  Therefore $a$ is elliptic in $T$. 
   
   Thus we proved that $S$ is universally elliptic in this case.

\medskip   
   If $n=-1$ ($p=2$, and so $G$ is pro-2 Klein bottle group), there are exactly two non-trivial deformation spaces: one contains the Bass-Serre tree of the HNN extension 
   $\langle a,t\mid tat^{-1}=a^{-1}\rangle$, the other 
    contains the tree associated to the amalgam $\langle t\rangle\amalg_{\langle t^2=v^2\rangle}\langle v\rangle$, with $v=ta$. None of these trees is universally elliptic ($t$ and $v$ are hyperbolic in the HNN extension, and $a$ is hyperbolic in the amalgam). Thus   the  JSJ deformation space of $BS(1,-1)$ is the trivial one.
   \end{example}
   
The essential feature of JSJ theory is the description of flexible
vertices, in particular the fact that flexible vertex stabilizers are
often “surface-like” (\cite{RS97, DS99, FP06} and \cite[Theorem 6.2]{GL}. The next theorem show that flexible vertices of cyclic pro-$p$ trees split as a free pro-$p$ product.   
   
   \begin{theorem} Let $G$ be a finitely generated pro-$p$ group and $\E$ be the family of infinite cyclic pro-$p$ groups. If $Q\neq \Z_p$ is a vertex stabilizer of a cyclic JSJ $G$-tree and  $Q$ does not split as a free pro-$p$ product then it is elliptic in any cyclic $G$-tree.\end{theorem}
   
   \begin{proof}  Let $T$ be a JSJ $G$-tree where $Q$ is not elliptic. By Lemma \ref{dominant splitting}(3) $Q$ splits as a free pro-$p$ product, contradicting the hypothesis.
   
   \end{proof}

  \subsection{Locally finite trees}\label{Baumslag-Solitar}

  One can generalize the previous example to locally finite trees with small edge stabilizers (i.e. $\E$ consists of small pro-$p$ groups). The example of it is a non-soluble Baumslag-Solitar pro-$p$ group, $$BS(m,n)=\langle a,t\mid ta^mt^{-1}=a^{n},)\rangle\cong \Z_p\amalg_{\Z_p} (\Z_p\rtimes \Z_p),$$
   where $m,n$ are supernatural numbers. Since $a^m$ and $a^n$ have the same order in every finite quotient, they generate the same group $\Z_p$ and  this $\Z_p$ is the amalgamated subgroup  identified with $n\Z_p=m\Z_p$ in the first factor and with the normal semidirect factor in the second factor; in particular this normal subgroup is normal in $G$.

\begin{definition}  Let be $G$ a pro-$p$ group, and $H\leq_{c} G$. The pro-$p$ commensurator of $H$ in $G$, denoted $Comm_{G}(H)$, is the set defined as:

$$Comm_{G}(H)=\overline{\{g\in G\mid H^{g}\cap H\leq_{o} H\, \textit{\rm and}\, H^{g}\cap H\leq_{o} H^{g}  \}}.$$
\end{definition}

Observe that the commensurator $Comm_{G}(H)$ is a (closed) subgroup of $G$ and $N_G(H)\leq Comm_{G}(H)$.   

\begin{theorem}\label{commensurator}(\cite[Theorem 3.4]{BZ})
Let $G$ be a pro-$p$ group acting on a pro-$p$ tree $T$ and $H$ a non-trivial  non-elliptic subgroup of $G$. Let $D$ be a minimal $C$-invariant pro-$p$ subtree of $T$. Then $Comm_{G}(H)$ acts on $D$.
\end{theorem}

  \begin{theorem}\label{theo.weak.conm}(\cite[Theorem 4.4]{BZ} Let $G=(\G,\Gamma)$ be the fundamental group of a reduced graph of pro-$p$ groups and suppose that $G$ is finitely generated. If $[G(d_0(e):G(e)]<\infty$, and $[G(d_1(e)):G(e)]<\infty$ for all $e\in E(\Gamma)$. Then $G=Comm_{G}(G(e))=N_{G}(J)$, for any $e\in E(\Gamma)$,  where $J\leq_{o} G(e)$.
\end{theorem}

\subsubsection{Generalized Baumslag-Solitar groups} 

\bigskip
   
 \begin{definition} A generalized Baumslag-Solitar pro-$p$ group (GBS) is a finitely generated pro-$p$ group which acts on a pro-$p$ tree $T$ with infinite cyclic vertex and edge stabilizers. \end{definition}
 
 Let $\E$ be the set of cyclic subgroups of $G$ (including the trivial subgroup), and $\K=\emptyset$. Unless $G$ is metacyclic,  the deformation space of $T$ is the JSJ deformation space. To prove this we need the following theorem on commensurator.  
   
\begin{theorem}\label{teoae}\cite[Theorem 4.7]{BZ}) Let $G$ be a finitely generated pro-$p$ group acting  irreducibly on a pro-$p$ tree $T$, such that $G_v$ is infinite cyclic group and $G_e\neq 1$ for all $v\in V(T)$, $e\in E(T)$. Then $Comm_{G}(G_e)=G$, for any $e\in E(T)$.
	
\end{theorem}

	\begin{theorem} Let $G$ be a finitely generated pro-$p$ group acting  irreducibly on a pro-$p$ tree $T$, such that $G_v$ is infinite cyclic pro-$p$ group for all $v\in V(T)$ and $G_e\neq 1$ for all  $e\in E(T)$. Unless $G$ is metacyclic-by $\Z/2$,  the deformation space of $T$ is the JSJ deformation space.\end{theorem}
   
  \begin{proof}  By Lemma \ref{universally elliptic vertex} it suffices to show that every vertex stabilizer $H$ of $T$ is universally elliptic.  The commensurator of $H$ is $G$ by Theorem \ref{teoae}. Then by Theorem \ref{theo.weak.conm} $H$ contains some open $J$ which is normal in $G$. If $H$ acts hyperbolically on a minimal pro-$p$ tree $T$, then so is $J$. Then by \cite[Proposition 8.2 (2)]{CZ}  $G/J$ is either cyclic or infinite dihedral.  This implies that $G$ is either metacyclic or $p=2$ and $G$ is metacyclic-by $\Z/2$. 
   
   \end{proof}
   
   \subsection{pro-$p$ RAAGs} Let $\Delta$ be a finite graph. The associated right-angled Artin pro-$p$ group $A_\Delta$ (RAAG, also called graph group, or partially commutative group) is the group presented as follows: there is one generator $a_v$ per vertex, and a relation $a_va_w=a_wa_v$ if there is an edge between $v$ and $w$ (see \cite{SZ}). The decomposition of $\Delta$ into connected components induces a decomposition of $A_{\Delta}$ as a free pro-$p$ product of freely indecomposable RAAGs (which may be infinite cyclic), so to study JSJ decompositions of $A_\Delta$ one may assume that $\Delta$ is connected. 
   
   \begin{definition}[Standard subgroups]\label{def: standard subgroup}
A subgroup of $G_\Gamma$ is called a \emph{standard subgroup} if it is the subgroup generated by a subset $V'\subseteq V(\Gamma)$. If $\Gamma = \varnothing$, by convention we set $G_\Gamma$ to be the trivial subgroup.
\end{definition}
   
   \begin{definition}[Support of an element]
Let $g$ be an element of a (pro-$p$) RAAG $G_{\Gamma}$.
The \emph{support} $\alpha(g)$ of $g$ is the set of canonical generators of the unique minimal standard subgroup of $G_{\Gamma}$ containing $g$.
\end{definition}
   
   \begin{definition}[Links and stars]
Let $g$ be an element of a (pro-$p$) RAAG $G_{\Gamma}$.
The link $\link(g)$ of $g$ is the set of vertices of $\Gamma \smallsetminus \alpha(g)$ that are adjacent to each of the vertices in $\alpha(g)$.
If $v$ is a canonical generator, we denote by $\mbox{Star}(v)$ the full subgraph generated by $\link(v)\cup v$.
\end{definition}

\begin{definition}[Hanging vertex]
We say that a vertex $v$ of $\Gamma$ is a \emph{hanging vertex} if $\mbox{Star}(v)$ is a complete graph and for every $w\in link(v)$, $\mbox{Star}(w)$ is not a complete graph.   
\end{definition}

Abusing the notation, if $S\subseteq V(\Gamma)$, we denote by $G_{S}$ the standard subgroup generated by the full subgraph generated by $S$.
   
   In  \cite{CPZ} is given a characterization of the splitting   of $G_\Gamma$ over an abelian pro-$p$ group. It is also given the following JSJ-decomposition of $G_\Gamma$ over abelian subgroups.

\begin{theorem} (\cite{CPZ})\label{thm: A JSJ decomposition}
Let $G=G_\Gamma$ be a pro-$p$ RAAG associated to a connected  finite graph $\Gamma$.

There is a (possibly trivial) decomposition of $G$ as a fundamental pro-$p$ group of a reduced finite graph of pro-$p$ groups $(\mathcal{G}_\Theta, \Theta)$ with the following properties:

\begin{itemize}
    \item the underlying graph $\Theta$ is either a tree or a tree with loops;
    \item vertex groups of $(\mathcal{G}_\Theta, \Theta)$ are standard subgroups which are either abelian or their underlying graph does not contain any disconnecting complete graph;
    \item each edge group of $(\mathcal{G}_\Theta, \Theta)$ is a standard subgroup associated with a disconnecting complete subgraph of $\Gamma$;
    \item hanging vertices do not belong to any vertex group.
\end{itemize}

Furthermore, the standard pro-$p$ tree associated to this decomposition is an abelian JSJ tree decomposition $(T_\Theta,G)$ of $G$.
\end{theorem}

\subsection{Non-rigid example}

Unlike the previous subsections, we now consider an example with flexible vertices which is a pro-$p$ version of an example from \cite{GL}.  We consider cyclic splittings with $\K=\emptyset$.

Suppose that $G_\Sigma$ is the fundamental pro-$p$ group of a closed orientable hyperbolic surface $\Sigma$. Any simple closed geodesic $\gamma$ on $\Sigma$ defines a cyclic splitting (an amalgam or an HNN-extension, depending on whether $\gamma$ separates or not). A non-trivial element $g\in G_\Sigma$, represented by an immersed closed geodesic $\delta$ , is elliptic in the splitting defined by $\gamma$ if and only if $\delta\cap  \gamma=\emptyset$  or $\delta=\gamma$. Since any $\delta $ meets transversely some simple $\gamma$, this shows that $1$ is the only universally elliptic element of $G_\Sigma$, so the JSJ decomposition of $G_\Sigma$ is trivial and its vertex is flexible.

\begin{example}\label{flexible vertices} Suppose that $G$ is the pro-$p$ completion of the fundamental group of the space consisting of three punctured tori $\Sigma_i$ attached along their boundaries. A presentation of $G$ is
$\langle a_1, b_1, a_2, b_2, a_3, b_3 | [a_1, b_1] = [a_2, b_2] = [a_3, b_3]\rangle$.
It is the fundamental group of a graph of groups $(\G,\Gamma)$ with one central vertex $v$ and three
terminal vertices $v_i$ with $e_i$ connecting $v$ to $v_i$. More precisely, putting $c=[a_1, b_1] = [a_2, b_2] = [a_3, b_3]$ the tree of groups $(\G,\Gamma)$ is as follows:

$$
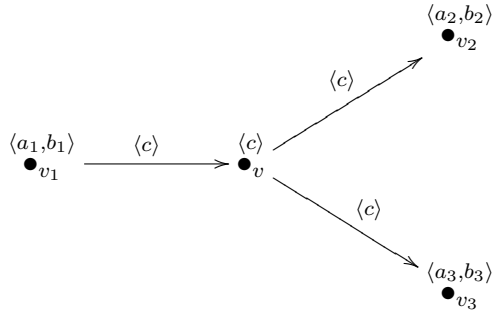
\xymatrix{&&&& {\overset{\langle a_2,b_2\rangle }{\bullet_{v_2}}}\\
{\overset{\langle a_1,b_1\rangle}{\bullet_{v_1}}}\ar[rr]^{\langle c\rangle}&& {\overset{\langle c\rangle}{\bullet_v}}\ar[rru]^{\langle c\rangle} \ar[rrd]^{\langle c\rangle}\\  
&&&&{\overset{\langle a_3,b_3\rangle}{\bullet_{v_3}}}}$$

\bigskip
All edges, as well as $v$, carry the same cyclic group $C = \langle c\rangle$. We claim
that $(\G,\Gamma)$ is a JSJ decomposition of $G$, with 3 flexible vertices $v_i$.

\medskip
Let us first show that $(\G,\Gamma)$ is universally elliptic, i.e. that $C$ is universally elliptic. Using
Lemma \ref{one-edge splitting}, we consider a one-edge cyclic splitting   of $G$ over ${\langle a\rangle}$ in which $C$ is not elliptic,
and we argue towards a contradiction. The group $G_{12}$ generated by $\langle a_1,b_1,a_2,b_2\rangle$
is the
fundamental pro-$p$ group of a closed surface (of genus 2), in particular it is freely indecomposable. It follows
that  $\langle a\rangle $   has non-trivial intersection with some
conjugate of $G_{12}$, so some $a^k$
lies in a conjugate of $G_{12}$ (where $k$ is some power of $p$), i.e. w.l.o.g. we may assume that $a^k\in G_{12}$.
We now consider the action of $a^
k$ on the standard pro-$p$ tree $T$ of $(\G,\Gamma)$. If $a^k$ fixes an edge in
$T$, then $\langle a^k\rangle$
 has finite index in some conjugate of $C$, so $C$ is elliptic in $G\amalg_{\langle a\rangle}$ by Corollary \ref{fixed vertex}, a contradiction. Hence by Proposition \ref{unique G-invariant} there exists 
the unique  minimal $a^k$-invariant subtree $D(a^
k
)$
in $T$ (if $a^k$ is elliptic, $D(a^
k
)$ is unique because $a$ does not fix an edge).
Then $D(a^
k
)$ is contained in the minimal subtree of  $G_{12}$, i.e. in the subtree $G\{e'_1,e'_2, v'_1,v'_2\}$ of $G$-translations of some lift $\{e'_1,e'_2, v'_1,v'_2\}$ of the spanning tree $[v_1,v_2]$ of $v_1,v_2$. But $G\{e'_1,e'_2, v'_1,v'_2\}$  contains no lift
of the vertex $v_3$ and hence so does not $D(a^
k
)$. Permuting indices shows that all vertices in $D(a^
k
)$ are lifts of $v$, so $D(a^
k
)$
is a single point (a lift of $v$) (since $\Gamma$ is a tree). This implies that $\langle a^k\rangle$ has finite index in some conjugate of
$C$, which is again a contradiction. Thus $C$ is universally elliptic in all cyclic splittings.

To prove maximality of the splitting (cf. Remark \ref{reduced JSJ}(i)), consider a universally elliptic tree $S$ dominating $T$. If the domination is strict (i.e. $S$ and $T$ are in different deformation spaces), some $G(v_i)$
is non-elliptic
in $S$. Being universally elliptic, $C$ is elliptic in $S$, so by a standard fact combined with \cite[Theorem 
4.6 (3)]{wilkes} the action of $G_{v_i}$
on its minimal subtree in $S$ comes from the splitting of the abstract surface group $S_{12}$ whose pro-$p$ completion is $G_{12}$. The splitting of  $S_{12}$ is associated to an essential simple closed curve
$\gamma_i \subset \Sigma_{i}$. Let $\gamma$ be a curve intersecting $\gamma_i$ non-trivially.  Considering the splitting
of  $G$ defined by a curve $\gamma \subset \Sigma_i$
 shows that $S$ is not universally
elliptic, a contradiction.
\end{example}

\bigskip
\begin{remark} In Example \ref{flexible vertices} the element $c$ can be replaced by the standard Demushkin 
relator 
$c = x_1^{p^r}[x_1, x_2] \cdots [x_{n-1}, x_n]$ 
if $p>2$ and by 
$c=x_1^{2+2^f}
 [x_1, x_2] \cdots [x_{n-1}, x_n]$ or by
 $c=x_1^2
[x_1, x_2]x_3^{2^f} [x_3,x_4]\cdots [x_{n-1}, x_n]$, or by 
$c=x_1^2x_2^{2f}[x_2, x_3] \cdots [x_{n-1}, x_n]$ for $f\geq 2$ by \cite[Theorem 5.6]{wilkes19} combined with \cite[Theorem 3.3]{wilkes}.\end{remark}

\section{A few useful facts} In this section, we first describe the behavior of the JSJ deformation space when we change the class $\E$ of allowed edge groups. We then introduce the incidence structure inherited by a vertex group of a graph of groups, and we relate JSJ decompositions of $G$ to JSJ decompositions of vertex groups relative to their incidence structure.  We also discuss relative finite presentation of vertex groups. Finally, we give an alternative construction of relative JSJ decompositions, obtained by embedding $G$ into a larger group.

\subsection{Changing edge groups} We fix two families of subgroups $\E$ and $\curlyF$, with $\E \subset \curlyF$, and we compare JSJ splittings over $\E$ and over $\curlyF$ (universal ellipticity, and all JSJ decompositions, are relative to some fixed family $\K$). 

For example: 

$\bullet$. $\E$ consists of the finitely generated abelian subgroups of $G$ , and $\curlyF$ consists of the subgroups having no non-abelian free pro-$p$ subgroups. 

$\bullet$. Groups in $\curlyF$ are locally small (their finitely generated subgroups are small), and $\E$ is the family of small subgroups.

$\bullet$ $\E$ is the family of cyclic subgroups, and $\curlyF$ is the family of  virtually cyclic subgroup.

$\bullet$ $\E$ consists of the trivial group, or the finite subgroups of $G$ (see Corollary 4.16).

\bigskip
There are now two notions of universal ellipticity, so we shall distinguish between $\E$-universal ellipticity (being elliptic in all $(\E,\K)$-trees) and $\curlyF$-universal ellipticity (being elliptic in all 
$(\curlyF,\K)$-trees). Of course, $\curlyF$-universal ellipticity implies $\E$-universal ellipticity. Recall that two trees are compatible if they have a common refinement. 

\begin{proposition}\label{jsj refinement} Assume $\E\subset \curlyF$. Let $T_\curlyF$ be a JSJ tree over $\curlyF$ with $T_\curlyF/G$ finite.
\begin{enumerate}
\item[(1)] If there is a JSJ tree $T_\E$ over $\E$, then there is a pro-$p$ JSJ $\E$-tree $\widehat T_\curlyF$ which is compatible with $T_\curlyF$. It may be obtained by refining $T_\curlyF$, and then collapsing all edges whose stabilizer is not in $\E$.

\item[(2)] If every $\E$-universally elliptic $\E$-tree is $\curlyF$-universally elliptic, the tree
$\overline T_\curlyF$ obtained from $T_\curlyF$ by collapsing all edges whose stabilizer is not in $\E$ is a JSJ tree over $\E$.\end{enumerate}
\end{proposition}

Note that (2) applies if $\E$ consists of finite groups (more generally, of groups with the pro-$p$ version of Serre’s property (FA)).

\begin{proof} (1) Let $T_\E$ be a JSJ pro-$p$ tree over $\E$. Since $\curlyF \supset \E$, the tree $T_\curlyF$ is elliptic with respect to $T_\E$. Let $\widehat T_\curlyF$ be a standard refinement of $T_\curlyF$ dominating $T_\E$ (see Proposition \ref{refinement}). Consider an edge $e$ of $\widehat T_\curlyF$ whose stabilizer is not in $\E$. Then $G_e$ fixes a unique vertex $v_e$ of $T_\E$ (if more than one, say $v,w\in V(T_\E)$) then fixes an edge by Theorem \ref{fixed geodesic}). Hence $G_{d_0(e)}, G_{d_1(e)}$ fix $v_e$ and so the group generated by them fixes $v_e$. It follows that the tree $\overline T_\curlyF$  obtained from $\widehat T_\curlyF$ by collapsing all  edges whose stabilizer is not in $\E$ dominates $T_\E$: each new vertex stabilizer is the stabilizer of some vertex in $T_\E$ (we use here that the number of such edges is finite up to translation since $T_\curlyF/G$ is finite; the argument also works if $\{e\in E(\widehat T_\curlyF)\}$ is open in $\widehat T_\curlyF$). Being $\E$-universally elliptic by Lemma \ref{refinement univer elliptic} (2), it is a JSJ tree over $\E$.

\medskip
For (2), first note that 
$\overline T_\curlyF$ is 
an $\E$-universally elliptic pro-$p$ $\E$-tree. 
If $T'$ is another one, it is $\curlyF$-universally elliptic, hence dominated by $T_\curlyF$. As above, the stabilizer of  any edge $e$ of $\widehat T_\curlyF$ not in $ \E$ fixes a unique vertex of $T_\E$, so $\overline T_\curlyF$ dominates $T'$.
Thus $\overline T_\curlyF$ is a JSJ tree over $\E$.\end{proof}

\subsection{Finiteness properties} 

 In this subsection we show that $G_v$ is finitely presented relative to incident edge groups if $G$ is finitely presented and edge stabilizers are finitely generated (Proposition \ref{relatively finitely presented} below).
 
 The part on finite generatedness of the next Proposition is \cite[Proposition 2.21]{CAZ}.

\begin{proposition}\label{reducing to trees}Let $G=\Pi_1(\G,\Gamma)$ be a finite graph of pro-$p$ groups and $D$ is a maximal subtree of $\Gamma$. Suppose $G$ is finitely generated (resp. finitely presented).  If $\G(e)$ is finitely generated (resp. finitely presented) for every $e\in E(\Gamma)$, then $G(v)$ are finitely generated (resp. finitely presented ). Moreover,  $d(\Pi_1(\G, D))\leq d(G)+ \sum_{e\in \Gamma\setminus D} (d(\G(e)-1)$.
\end{proposition}
\begin{proof} Since $\Gamma$ is finite, we can think of $G$ as $G=HNN(\Pi_1(\G, D), \G(e), t_e), e \in \Gamma\setminus D$. 
	
	Consider the Mayer-Vietoris sequence associated to $(\G,\Gamma)$ with coefficients in $\F_p$:
	
	$$ \oplus_{e\in E(\Gamma)} H_2(G(e))\rightarrow \oplus_{v\in V(\Gamma)}H_2(G(v))\rightarrow H_2(G)\rightarrow $$
	$$\rightarrow	\oplus_{e\in E(\Gamma)} H_1(G(e))\rightarrow 
		\oplus_{v\in V(\Gamma)} H_1(G(v))\rightarrow H_1(G).$$ 
	
	Since $H_1(G), H_1(G(e))$ (resp. $H_2(G), H_2(G(e))$) are finite, so are $H_1(G(v))$ (resp. $H_2(G(v)))$. This proves the first statement.
	
	The last statement is the subject of \cite[Proposition 2.21]{CAZ}.
\end{proof}

  Proposition \ref{reducing to trees} says  that, if $\Pi_1(\G,\Gamma)$ and all edge groups are finitely generated (resp. finitely presented), then so are all vertex groups. The goal of this subsection is to extend these results to relative finite generation (resp. finite presentation), as we define below.  This is not needed if all groups in $\E$ are assumed to be finitely presented. 

\bigskip
{\bf Relative finite generation and presentation}

\bigskip
If we  assume that $G$ is finitely generated, or finitely presented, these properties are not always inherited by vertex groups. We therefore consider relative finite generation (or presentation), which behave better in that respect. 

Let $G$ be a pro-$p$ group and  $\K$ be a continuous family of subgroups of $G$. 

\begin{definition} One says that $G$ is finitely generated relative to $\K$ if there exists a finite set $\Omega\subset G$ such that $G$ is generated by $\Omega\cup \K$. Such a subset $\Omega$ is a relative generating set.\end{definition} 

 If $G$ is finitely generated, then clearly it is finitely generated relative to any $\K$. If the pro-$p$ groups in  $\K$ are finitely generated, relative finite generation is equivalent to finite generation. Note  that relative finite generation does not change if one replaces the subgroups in  $\K$ by their conjugates or closes $\K$ for conjugation. 
 
 Recall that a pro-$p$ $G$-tree $T$ is relative to $\K$ if every $K\in\K$ is elliptic in $T$. We now consider relative finite presentation (see [Osi06] for the abstract case). Note that, if $\Omega$ is a relative finite generating set, then the natural morphism $F(\Omega)\amalg \coprod_{K\in \K}K\longrightarrow G$ is an epimorphism (with $F(\Omega)$ the free pro-$p$ group on $\Omega$).

 The next proposition gives the homological characterization of relative generation.
 
 \begin{proposition}\label{homological} Suppose $\K$ is a continuous family of subgroups of a pro-$p$ group $G$ closed for conjugation. Then 
$G$  is finitely generated relative to $\K$ if and only if the image of the natural map $\bigoplus_{o\in \K/G} H_1(G,\F_p[[o]])\longrightarrow H_1(G,\F_p)$ is open.

 \end{proposition}
 
 \begin{proof} Just observe that $H_1(G)=G/\Phi(G)$. 
 \end{proof}

 \begin{definition} One says that a pro-$p$ group $G$  is finitely presented relative to a continuous family of subgroups $\K$ closed for conjugation  if  the image of the natural map $\bigoplus_{o\in \K/G} H_2(G,\F_p[[o]])\longrightarrow H_2(G,\F_p)$ is open.
   \end{definition} 
 
 \bigskip
\begin{remark}\label{section} Suppose w.l.o.g. $\K$ is closed for conjugation. The definition above of relative finite presentability is given in terms of the second homology because the map $\K\longrightarrow \K/G$ does not always admit a continuous section (see \cite[Section 5.6]{RZ-10}).  When such section $\sigma:\K/G\longrightarrow \K$  exists, we can reduce $\K$ to $\widetilde \K=\sigma(\K)$ and give the following equivalent similar to the abstract case definition: $G$ is finitely presented relative to $\widetilde K$  if  there exists a finite ({\it relative generating}) set $\Omega\subset G$, such that the natural epimorphism $f:F(\Omega)\amalg \coprod_{K\in \widetilde\K} K\longrightarrow G$ has finitely generated kernel as a normal subgroup. The equivalence of the definition in this case follows directly from 5-term Hoschild-Serre spectral sequence associated to the epimorphism $f$. In this case finite presentation  
is not affected if one replaces the subgroups $K\in \sigma(\K/G)$ by their conjugates.
\end{remark} 
 
 \bigskip
 One easily checks that relative finite presentation does not depend on the choice of $\Omega$. If $G$ is finitely presented, then it is finitely presented relative to any  collection of  subgroups. 
 Conversely, if $G$ is finitely presented relative to any finite collection of finitely presented subgroups, then $G$ is finitely presented.
 
 \medskip
 The following lemma will be used below.
 
 \begin{lemma} Suppose that $G$ is finitely presented relative to $\{K_1,\ldots,K_p\}$, 
 	 and $K_p\subset K_i$ for some $i < p$. 
 	 Then $G$ is finitely presented relative to 
 	 $\{K_1,
 	 \ldots,
 	 K_{p-1}\}$.
 	  
  \end{lemma}

\begin{proof} 

 Let $f:F(\Omega)\amalg K_1\amalg\cdots\amalg K_p\longrightarrow G$ be an epimorphism corresponding to the presentation of $G$ relative to $\{K_1,\ldots,K_p\}$ and $$s: F(\Omega)\amalg K_1\amalg\cdots\amalg K_{p-1}\longrightarrow G$$ be the restriction of $f$ to  $F(\Omega)\amalg K_1\amalg\cdots\amalg K_{p-1}$.
Then $f$ factors through the epimorphism $F(\Omega)\amalg K_1\amalg\cdots\amalg K_p\longrightarrow F(\Omega)\amalg K_1\amalg\cdots\amalg K_{p-1}$ that sends $K_p$ to its copy in $K_i$ and the rest of the factors identically to their copies.  

$$\xymatrix{F(\Omega)\amalg K_1\amalg\cdots\amalg K_p\ar[rr]\ar[rd]_f&&F(\Omega)\amalg K_1\amalg\cdots\amalg K_{p-1}\ar[dl]^s\\
&G&}$$

Therefore the kernel of $f$ is mapped onto the kernel of $s$ and so the kernel of $s$ is finitely generated as a normal subgroup.
\end{proof}

Suppose that a finitely generated group $G$ splits as a finite graph of pro-$p$ groups.  It  follows from Proposition \ref{reducing to trees}  that vertex groups are finitely generated if one assumes that edge groups are finitely generated, but this is false in general without this assumption. However, vertex groups are always finitely generated relative to the incident edge groups, and there is a similar statement for relative finite presentation (see below).

The abstract version of the following lemma can be found in \cite[Lemmas 1.11, 1.12]{Gui08}. Note that it gives another proof of Proposition \ref{reducing to trees}.

 \begin{lemma}\label{ref fin gen} If a finitely generated pro-$p$ group $G$ acts on a pro-$p$ tree $T$ cofinitely, then every vertex stabilizer $G_v$ is finitely generated relative to the incident edge groups. More generally, if $G$ is finitely generated relative to $\K$, and $T$ is a tree relative to $\K$, then $G_v$ is finitely generated relative to $Inc^{\K}_v$.\end{lemma}
 
 \begin{proof}
 Let $(\G,\Gamma)$ be the corresponding graph of pro-$p$ groups with $\Gamma= T/G$. For every  vertex $v$ of $\Gamma$  put $\overline\G(v)=\G(v)/\langle \G(e)^{\G(v)}\mid e\in E(d_0^{-1}(v)\cup d_1^{-1}(e))\rangle$. Define the quotient graph of groups $(\overline \G, \Gamma)$
 by putting  $\overline\G(v)$ on top of $v$ and setting the edge groups to be trivial.  Then from the presentation \eqref{presentation} for
 $\Pi_1(\overline \G, \Gamma)$ (see Section 2.4) it follows   that $$\Pi_1(\overline \G,
 \Gamma)=\coprod_{v\in V(\Gamma)} \overline\G(v).$$ The natural morphism $(\G,\Gamma)\longrightarrow (\overline \G,\Gamma)$ induces then the epimorphism $$G=\Pi_1(\G,\Gamma)\longrightarrow \overline G=\Pi_1(\overline \G,\Gamma).$$ If $G$ is finitely generated, then so is $\overline{\G}(v)$ and therefore so is $\G(v)$ relative to $$\{\G(e)^{\G(v)}\mid e\in E(d_0^{-1}(v)\cup d_1^{-1}(e))\}.$$ 
 
 To show the second statement we use overline for the images of subgroups of $G$ in $\overline G$. Note that if $G$ is finitely generated relative to $\K$ then $\overline \G$ is finitely generated relative to $\overline \K$ and therefore so are $\overline{\G}(v)$.  Hence   $\overline\G_v$ are finitely generated  relative to $$\{\G(e)^{\G(v)}\mid e\in E(d_0^{-1}(v)\cup d_1^{-1}(e))\}\cup Inc^{\K}_v.$$
 \end{proof}
 
 \begin{remark} If $\K=\{K_1,...,K_p\}$, then the family $\K_{|G_v}$ consists of at most $p$ groups, each conjugate to some $K_i$. \end{remark}
 
 There is a similar statement for relative finite presentation.
 
 \begin{proposition}\label{relatively finitely presented} Let $G=\Pi_1(\G, \Gamma)$ be the fundamental pro-$p$ group of a finite graph of pro-$p$ groups with finitely generated edge groups. If $G$ is finitely presented, then every vertex group $G(v)$ is finitely presented relative to the incident edge groups. More generally, if $G$ is finitely presented relative to $\K$, and $G=\Pi_1(\G, \Gamma)$ is a splitting  relative to $\K$, then $G_v$ is finitely presented relative to $Inc_v^\K$.\end{proposition}
 
 \begin{proof} W.l.o.g we assume that $\K$ is closed for conjugation in $G$.  Consider the Mayer-Vietoris sequence:
  $$\bigoplus_{e\in E(\Gamma)} H_2(G_e)\longrightarrow  \bigoplus_{v\in V(\Gamma))}H_2(G_v)\longrightarrow H_2(G)\longrightarrow \bigoplus_{e\in E(\Gamma)} H_1(G_e).$$
  
  As in the previous lemma for every  vertex $v$ of $\Gamma$  put $$\overline\G(v)=\G(v)/\langle \G(e)^{\G(v)}\mid e\in E(d_0^{-1}(v)\cup d_1^{-1}(e))\rangle.$$ Define the quotient graph of groups $(\overline \G, \Gamma)$
 by putting  $\overline\G(v)$ on top of $v$.  Then from the presentation \eqref{presentation} (Section 2.4) for
 $\Pi_1(\overline \G, \Gamma)$ it follows  that $$\Pi_1(\overline \G,
 \Gamma)=\coprod_{v\in V(\Gamma)} \overline\G(v).$$ The natural morphism $(\G,\Gamma)\longrightarrow (\overline \G,\Gamma)$ induces then the epimorphism $$G=\Pi_1(\G,\Gamma)\longrightarrow \overline G=\Pi_1(\overline \G,\Gamma)$$ with the kernel $\langle\langle G(e)\mid e\in E(\Gamma)\rangle \rangle$. Let 
$$H_2(G)\longrightarrow H_2(\overline G)\longrightarrow  \bigoplus_{e\in E(\Gamma)} H_1(G_e)$$ be the part of 
  the 5-term Hoschild-Serre sequence  for this epimorphism.  The right part of it is finite. So from the commutative diagram
$$\xymatrix{  \bigoplus_{K\in \K} H_2(K)\ar[d]\ar[r]&H_2(G)\ar[d]\\
\bigoplus_{K\in \K} H_2(\overline K)\ar[r]&H_2(\overline G)}$$
 it follows that 
   $\bigoplus_{K\in \K} H_2(K)\longrightarrow H_2(G)$ has open image if and only if $$\bigoplus_{ K\in \K} H_2(\overline K)\longrightarrow H_2(\overline G)$$ has open image. Similarly  $\bigoplus_{K\in \K_{|G(v)}} H_2(K)\longrightarrow H_2(G(v))$ has open image if and only if $\bigoplus_{ K\in \K_{|G(v)}} H_2(\overline K)\longrightarrow H_2(\overline G(v))$ has open image. Thus it suffices to prove the proposition when edge groups are trivial. 
  
  Assuming this the natural map $\bigoplus_{v\in V(\Gamma))}H_2(G_v)\longrightarrow H_2(G)$ is an isomorphism. As the image  of $\bigoplus_{ K\in \K} H_2(\overline K)$ is open in $H_2(\overline G)$, the image of  $\bigoplus_{ K\in \K_{|G(v)}} H_2(\overline K)$ must be open in $H_2(\overline G(v))$ for every $v$. This finishes the proof.

 \end{proof}
 
 \subsection{Incidence structures of vertex groups} Given a vertex stabilizer $G_v$ of a tree $T$, it is useful to consider splittings of $G_v$ relative to incident edge stabilizers, as they extend to splittings of $G$ if $T/G$ is finite (Lemma \ref{refinement at v} below). In this subsection we give definitions. We assume in this subsection until the rest of the section that $\K$ does not have subgroups contained in edge stabilizers.
  
  \bigskip
  {\bf Definitions.}  Let $T$ be a tree (minimal, relative to $\K$, with edge stabilizers in $\E$). Let $v$ be a vertex, with stabilizer $G_v$.
  
  \begin{definition} (Incident edge groups $Inc_v$). Given a vertex $v$ of a tree $T$ 
  we define $Inc_v$ (or $Inc_{G_v}$) as the family of stabilizers $G_{e}$ of edges incident to $v$. We call $Inc_v$ the set of incident edge groups. It is a continuous family of subgroups of $G_v$ closed for $G_v$-conjugation. \end{definition}
  
  If $T/G$ is finite, alternatively, one can define $Inc_v$ from the quotient graph of groups $\Gamma =T/G$ as the image in $G_v$ of the groups carried by all oriented edges  incident to $v$. 
  
  \begin{definition}  Given $v$, consider the family of conjugates of groups in $\K$ that fix $v$ and no other vertex of $T$. We define the restriction $\K_{|G_v}$. 
  \end{definition}
  
  \begin{definition} ($Inc^{\K}_v$). We define $Inc^{\K}_v=Inc_v\cup \K_{|G_v}$. We will sometimes write $Inc_{|Q}$ and $Inc_{\K|Q}$ rather than $Inc_v$ and $Inc_{G_v}$, with $Q=G_v$.\end{definition}
  
We also view $\K_{|G_v}$ and $Inc^{\K}_v$ as families of subgroups of $G_v$, each well-defined up to conjugacy.

\begin{remark} We emphasize that $\K_{|G_v}$ only contains groups having $v$ as their unique fixed point. Two such groups are conjugate in $G_v$ if and only if they are conjugate in $G$. In particular, the number of $G_v$-conjugacy classes of groups in $\K_{|G_v}$ is bounded by the number of $G$-conjugacy classes of groups in $\K$. Also note that any subgroup of $G_v$ which is conjugate to a group of $\K$ is contained (up to conjugacy in $G_v$) in a group belonging to $Inc^{\K}_v$, so is elliptic in any splitting of $G_v$ relative to $Inc^{\K}_v$. \end{remark}
 
 \bigskip
 \subsection{JSJ decompositions of vertex groups.}
 
 \bigskip 
 Given an $(\E,\K)$-tree $T$ with $T/G$ finite, we compare splittings of $G$ and relative splittings of vertex stabilizers $G_v$. Recall that $Inc^{\K}_v$ is the family of incident edge stabilizers together with $\K_{|G_v}$ (see Subsection 8.2), and assuming that $G$ is finitely generated one has that $G_v$ is finitely generated relative to $Inc^{\K}_v$, by Lemma \ref{ref fin gen}. In particular, whenever $G_v$ acts on an $(\E,\K)$-tree $S_v$ relative to $Inc_v$ with no global fixed point, there is a unique minimal $G_v$-invariant subtree $D(G_v)\subset S_v$ by Proposition \ref{unique G-invariant}. We view $D(G_v)$ as a tree with an action of $G_v$; if $G_v$ is elliptic in $S_v$, we let $D(G_v)$ be any fixed point. 
 
 \begin{definition} ($\E_v$). We denote by $\E_v$ the family consisting of subgroups of $G_v$ belonging to $\E$. All splittings of $G_v$ will be over groups in $\E_v$.\end{definition} 
 
 \begin{lemma}\label{refinement at v} Let $G_v$ be a vertex stabilizer of an $(\E,\K)$-tree $T$ and suppose that $T/G$ is finite.  Any splitting of $G_v$ relative to $Inc^{\K}_v$ extends (non-uniquely) to a splitting of $G$ relative to $\K$. More precisely, given an $(\E_v,Inc^{\K}_v)$-tree $S_v$, there exist an $(\E,\K)$-tree $\hat T$ and a collapse map $p:\hat T\longrightarrow T$ such that $p^{-1}(v)$ is $G_v$-equivariantly isomorphic to $S_v$.\end{lemma}  
 
 We say that $\hat T$ is obtained by refining $T$ at $v$ using $S_v$. More generally, if $T/G$ is finite,  one may choose a splitting for each orbit of vertices of $T$, and refine $T$ using them. Any refinement of $T$ may be obtained by this construction (possibly with non-minimal trees $S_v$).
 
 \begin{proof} We construct $\hat T$ as in the proof of Proposition \ref{refinement} , with $Y_v=S_v$. It is relative to $\K$ because any group in $K$ which is conjugate to a subgroup of $G_v$ is conjugate to a subgroup of a group belonging to $Inc^{G_v}$. Non-uniqueness comes from the fact that there may be several ways of attaching edges of $T$ to $S_v$.\end{proof}

 \begin{lemma}\label{ellipticity in stabilizers} Let $G_v$ be a vertex stabilizer of a universally elliptic $(\E,\K)$-tree $T$ with $T/G$ finite. 
 
 $\bullet$ The groups in $Inc^{\K}_v$ are elliptic in every $(\E,\K)$-tree $S$.
 
 $\bullet$ A subgroup $K< G_v$ is $(\E,\K)$-universally elliptic (as a subgroup of $G$) if and only if it is $(\E_v,Inc^{\K}_v)$-universally elliptic (as a subgroup of $G_v$).\end{lemma} 
 
  The second assertion says that $K$ is elliptic in every $(\E,\K)$-tree on which $G$ acts if and only if it is elliptic in every $(\E_v,Inc_v\cup K_{|G_v})$-tree on which $G_v$ acts. If this holds, we simply say that $K$ is universally elliptic.
  
  \begin{proof} The first assertion is clear: $S$ is relative to $\K_{|G_v}$, and also to $Inc_v$ because $T$ is universally elliptic. 
  
  Suppose that $K$ is $(\E_v,Inc^{\K}_v)$-universally elliptic, as a subgroup of $G_v$. Let $S$ be any $(\E,\K)$-tree. It is relative to $Inc^{\K}_v$ by the first assertion. Let  $S_v\subset S$  be  a minimal $G_v$-invariant  subtree. The action of $G_v$ on $S_v$ is an $(\E_v,Inc^{\K}_v)$-tree, so by assumption $K$ fixes a point in $S_v$, hence in $S$. We have proved the “if” direction in the second assertion, and the converse follows from Lemma \ref{refinement at v}.\end{proof}

  \begin{corollary}\label{flexible or rigid}  Let $G_v$ be a vertex stabilizer of a JSJ tree $T_J$ such that $T_J/G$ is finite.
  
  \begin{enumerate}
  \item[(1)] $G_v$ does not split over a universally elliptic subgroup relative to $Inc^{\K}_v$.
  \item[(2)] $G_v$ is flexible if it splits relative to $Inc^{\K}_v$, rigid otherwise.
\end{enumerate}
\end{corollary} 
  
\begin{proof} If there is a splitting as in (1), we may use it to refine $T_J$ to a universally elliptic tree (see Lemma \ref{refinement at v}). This tree must be in the same deformation space as $T_J$, so the splitting of $G_v$ must be trivial. 

(2) follows from Lemma \ref{ellipticity in stabilizers} applied with $K=G_v$\end{proof}

\begin{proposition}\label{interconnection} Let $T$ be a universally elliptic $(\E,\K)$-tree with $T/G$ finite.

\begin{enumerate}

\item[(1)] Assume that every vertex stabilizer $G_v$ of $T$ has a JSJ tree $T_v$ relative to $Inc^{\K}_v$. One can then refine $T$ using these decompositions so as to obtain a JSJ tree of $G$ relative to $\K$.

\item[(2)] Conversely, if $T_J$ is a JSJ tree of $G$ relative to $\K$, and $G_v$ is a vertex stabilizer of $T$, one obtains a JSJ tree for $G_v$ relative to $Inc^{\K}_v$ by considering the action of $G_v$ on its minimal subtree $T_v=D_{T_J}(G_v)$ in $T_J$ (with $T_v$ a point if $G_v$ is elliptic in $T$).
\end{enumerate}
\end{proposition}

\begin{proof} To prove (1), let $\hat T$ be the tree obtained by refining $T$ using the $T_v$’s as in Lemma \ref{refinement at v}. It is relative to $\K$, and universally elliptic by Lemma \ref{refinement at v} since its edge stabilizers are edge stabilizers of $T$ or of some $T_v$. To show maximality, we consider another universally elliptic $(\E,\K)$-tree $T'$, and we show that any vertex stabilizer $H$ of $\hat T$ is elliptic in $T'$. It is in a vertex stabilizer of some $T_v$, with $v$ a vertex of $T$. If $G_v$ is not elliptic in $T'$, its minimal subtree $Y_v$ is a universally elliptic $(\E_v,Inc^{\K}_v)$-tree (by Lemma \ref{ellipticity in stabilizers}). Since $T_v$ is a JSJ tree, it dominates $Y_v$ so $H$ is elliptic in $T'$. This proves (1).

Now let $T_J$ and $T_v\subset T_J$ be as in (2). By Lemma \ref{refinement at v} $T_v$ is relative to $Inc^{\K}_v$, and it is $(\K_v,Inc^{\K}_v)$-universally elliptic because its edge stabilizers are contained in edge stabilizers of $T_J$. To prove maximality of $T_v$, consider another universally elliptic pro-$p$ tree $S_v$ with an action of $G_v$ which is relative to $Inc^{\K}_v$ and universally elliptic. Use it to refine $T$ to a tree $\hat T$ as in Lemma \ref{refinement at v}.
As above, $\hat T$ is relative to $\K$, and universally elliptic by Lemma \ref{refinement at v}. Being a JSJ tree, $T_J$ dominates $\hat T$. Vertex stabilizers of $T_v$ are elliptic in $T_J$, hence in $\hat T$, hence in $S_v$, so $T_v$ dominates $S_v$. This proves (2).\end{proof}

The following corollary says that one may usually restrict to unsplittable pro-$p$ groups over a finite subgroup when studying JSJ decompositions.

\begin{corollary}\label{refining Grushko} Suppose that $\E$ contains all finite subgroups of $G$. If $\K=\emptyset$, refining a Grushko decomposition of $G$ using JSJ decompositions of free factors yields a JSJ decomposition of $G$. Similarly, refining a  decomposition of $G$ over finite groups using JSJ decompositions of the vertex groups $G_v$ yields a JSJ decomposition of $G$. If $\K\neq \emptyset$, one must use relative free product JSJ decompositions or  JSJ decompositions over finite groups, and JSJ decompositions of vertex groups $G_v$ relative to $\K_{|G_v}$. Every flexible subgroup of $G$ is a flexible subgroup of some $G_v$. As mentioned in Section 4, JSJ decompositions always exist under some accessibility assumption (see Section 4 for examples of accessibility).\end{corollary} 

 \begin{proof} This follows from Proposition \ref{interconnection} since $T$ is in the deformation space with $\E$ consisting of either trivial or finite groups: finite groups are universally elliptic, and every splitting of $G_v$ is relative to $Inc_v$. The assertion about flexible subgroups follows from Corollary \ref{flexible or rigid}.\end{proof}
 
 \begin{remark} The results of this subsection remain true if one replace the hypothesis $|T/G|<\infty$ by $G$ being finitely generated. Indeed, in the proof of Lemma \ref{refinement at v} one just needs to use the proof of Proposition \ref{refinement for finitely generated G} instead of the proof of Proposition \ref{refinement}. We conclude this section stating it.
 
 \end{remark}
 
 We begin restating Lemma \ref{refinement at v}
 
 \begin{lemma}\label{refinement at v for fintely generated G} Let $G_v$ be a vertex stabilizer of an $(\E,\K)$-tree $T$ and suppose that $G$ is finitely generated.  Any splitting of $G_v$ relative to $Inc^{\K}_v$ extends (non-uniquely) to a splitting of $G$ relative to $\K$. More precisely, given an $(\E_v,Inc^{\K}_v)$-tree $S_v$, there exist an $(\E,\K)$-trees $\hat S$ and a collapse map $p:\hat S\longrightarrow S$ to a pro-$p$ tree $S$ in the deformation space of $T$ such that $p^{-1}(v)$ is $G_v$-equivariantly isomorphic to $S_v$.\end{lemma}
 
 Now we restate Proposition \ref{interconnection}.
 
 \begin{proposition}\label{interconnection for finitely generated G} Let $G$ be finitely generated and $T$  a universally elliptic $(\E,\K)$-tree.

\begin{enumerate}

\item[(1)] Assume that every vertex stabilizer $G_v$ of $T$ has a JSJ tree $T_v$ relative to $Inc^{\K}_v$. One can then refine some pro-$p$ tree $S$ from the deformation space of $T$ using these decompositions so as to obtain a JSJ tree of $G$ relative to $\K$.

\item[(2)] Conversely, if $T_J$ is a JSJ tree of $G$ relative to $\K$, and $G_v$ is a vertex stabilizer of $T$, one obtains a JSJ tree for $G_v$ relative to $Inc^{\K}_v$ by considering the action of $G_v$ on its minimal subtree $T_v=D_{T_J}(G_v)$ in $T_J$ (with $T_v$ a point if $G_v$ is elliptic).
\end{enumerate}
\end{proposition}

\section{Trees of cylinders}

\begin{definition} Let $X$ be a profinite space. An equivalence relation on $X$ is called profinite if it is an intersection of clopen equivalence relations on $X$.\end{definition}

\begin{lemma}\label{closedness} Let $\E$ be a family of subgroups of a profinite group $G$ closed in the \'etale topology of $Subgr(G)$ (see Subsection 2.3) and  for conjugation. If $G$ acts on a profinite space  $X$, then $\{x\in X\mid G_x\in \E\}$ is a closed subset of $X$. 

\end{lemma}

\begin{proof} By Lemma \ref{Ribes 5.2.2} $\{G_x\mid x\in X\}$ is a continuous family. Then by  \cite[Lemma 5.2.1 (c)]{R} 
$\varphi:X\longrightarrow Subgr(G)$ is continuous. Hence $\varphi^{-1}(E)=\{x\in X\mid G_x\in \E\}$ is a closed subset of $X$. 

\end{proof}

\begin{remark}\label{strict topology} There is a natural topology 
 on $Subg(G)$ that we (following \cite{HJ}) shall call strict topology. Its base of open subsets is given by inverse images of elements in the discrete space $Subgr(G/U)$ for all open normal subgroups of $G$. This topology is stronger than   the \'etale topology and so the statement of Lemma \ref{closedness} holds also under assumption that $E$ is closed in the strict topology. \end{remark}

\begin{lemma}\label{inverse limit of relations} Let $G$ be a profinite group acting on a profinite space $X$  and $\sim$ a profinite $G$-invariant equivalence relation on $X$. Let $X=\varprojlim_i X_i$ be a decomposition as a surjective inverse limit of finite $G$-quotient spaces $X_i$ of $X$. Define an equivalence relation $\sim_i$ on $X_i$ putting $m_i\sim_i f_i$ if there are elements $m\sim f$ in the corresponding preimages  of $m_i$ and $f_i$ in $X$ respectively and then taking the transitive hull. Then $\sim=\varprojlim_i \sim_i$.
	
	\end{lemma}
	
	\begin{proof}  	As $X_i$ is finite, $\sim_i$ is clopen. Then its preimage $_i\sim$ in $X$ is also clopen.  We just need to show that for not equivalent  $e,f$ there exists $i$ such that $e_i\not\sim_i f_i$, where $e_i,f_i$ are the images of $e$ and $f$ in $X_i$. 

As  $\sim$ is profinite, it is the intersection of clopen equivalence relations $R$ on $X$.  Since $X=\varprojlim_i X_i$, for any such $R$ there exists $i$ such that for any element $m_i$ in $X_i$  and any elements $m,e$ of its preimage in $X$ one has  $e R m$. This means that two not $R$-equivalent  points of $X$ have distinct images in $X_i$.  Hence if $m_i\sim_i f_i$ then  for some elements $m,f$ of their preimages in $X$ one has $m R f$ (indeed, there exists a finite sequence of elements $m_i, m_{i1}, \ldots m_{in}, f_i$ such that any pair of consecutive elements have  a pair of equivalent elements in their preimages in $X$ and hence belong to the same $R$-equivalence class). Then by the choice of $i$, for any such elements $m,f$ of  preimages of $m_i, f_i$ in $X$ one has $m R f$. 
	\end{proof}

\begin{definition}\label{admissibleontree} (Admissible equivalence relation). Let $G$ be a pro-$p$ group acting on a pro-$p$ tree  $T$ with closed $E(T)$.  A profinite equivalence relation  $\sim$ on $E(T)$  is called
admissible  if the following axioms hold:

(1) If $e \sim e'$ and $g \in G$, then $ge \sim ge'$
(invariance under translation).

(2) If $G_e \subset G_{e'}$, then $e\sim e'$ (nesting of stabilizers implies equivalence).

(3)  If $f \sim f'$, 
 then for each edge $e \subset [f, f']$ one has $e \sim f \sim f'$.

\end{definition}

\begin{lemma}\label{cylinders of trees} Let $G$ be a pro-$p$ group acting on a profinite graph $T$. Let $\sim$ be a $G$-invariant profinite  equivalence relation on $E(T)$.  Then
	
	\begin{enumerate}

\item[(i)] The union of all elements of $E(T)$ in an
equivalence class of it  together with their vertices is a profinite subgraph $D$ of $T$; a connected component $Y$ of it is a connected profinite graph and $E(Y)$ is called a cylinder of $T$. 

\item[(ii)] If $T$ is a pro-$p$ tree  and $\sim$ is admissible then

\begin{enumerate}

\item[(a)]  $D$ is a pro-$p$ subtree of $T$ and so coincides with $Y$;

\item[(b)]  two
distinct cylinders meet in at most one point. 
\end{enumerate}
\end{enumerate}
\end{lemma} 

\begin{proof}

	(i)  Let $G_m$ be  the stabilizer of an element  $m\in T$. Since  $\{ f\in E(T)\mid f\sim m \}$ is closed,  $Y=\{ f\cup d_0(f)\cup d_1(f)\mid f\in E(T),  f\sim m \}$ is a subgraph of $T$.

	(ii) Suppose $T$ is a pro-$p$ tree. 
	
	(a) Let $A$ and $B$ be subgroups in $\E$. Then  by axiom (3) for vertices $v,w$ of $D$ stabilized by $A$ and $B$ we have $[v,w]\subseteq D$ and so     $D$ is a pro-$p$ subtree.
	
	(b) If $v,w$ are in a cylinder $Y$, then by Theorem \ref{fixed geodesic}  so is $[v,w]$; therefore if they belong to any other cylinder  they can not be distinct, since otherwise any edge of $[v,w]$ belong to different equivalence classes, which is impossible.
	
	\end{proof}

\begin{definition} (Graph of cylinders). Let $G$ be a pro-$p$ group. Given a connected profinite $G$-graph $T$ with $E(T)$ compact  and a $G$-invariant profinite equivalence relation $\sim$ on $E(T)$, its graph of
	cylinders $T_c$ is the bipartite profinite graph with vertex set $V_0(T_c) \cup V_1(T_c)$, where $V_0(T_c)$ is the set of
	vertices $v$ of $T$, $V_1(T_c)= E(T)/\sim$ is the space of cylinders $E(Y)$ of
	$T$, and there is an edge $\epsilon = (v, E(Y) )$ between $v$ and $E(Y)$ in $T_c$ if and only if $v \in Y$ in $T$. Thus $V_1(T_c)$ is homeomorphic to $E(T)/\sim$ and the incident maps $d_0, d_1$ defined as $d_0(\epsilon)=v$, $d_1(\epsilon)=Y$ are continuous, because $E\longrightarrow E/\sim$ and $d_i$ are continuous.
	Equivalently, one obtains $T_c$ from $T$ by replacing each cylinder $Y$ by the cone over the
	set of vertices $v \in Y$.
	
	\end{definition}

\begin{remark}  If $v\in V_0(T)$ then its stabilizer equals $G_v$. If $v\in V_1(T_c)$, i.e. $v=Y$ then its stabilizer is the set stabilizer $G_Y$ of $Y$. Hence if $e=(v,Y)$ then its stabilizer is $G_v\cap G_Y$.

\end{remark}

\begin{proposition}\label{finite diameter}  Suppose $T/G$ is finite. Then $T_c/G$ is finite.

\end{proposition}

\begin{proof} Indeed,   $|V_1(T_c)/G|\leq |V(T)/G|=|(V_0(T_c/G)|$ and $V_0(T/G)=V(T/G)$.  So $V(T_c)/G$ is finite and hence so is $T_c/G$.
	
\end{proof}

\begin{lemma}\label{embedding}  Let $\Gamma$ be a finite connected graph and $\sim$ an  equivalence relation  on $E(\Gamma)$ such that for each equivalence class $Y$  of it $d_0(Y)\cup d_1(Y)$  is a connected subgraph of $\Gamma$.  Let $\Gamma_c$ be the graph of cylinders with respect to this relation. Then $\pi_1(\Gamma)\geq \pi_1(\Gamma_c)$.
\end{lemma}

\begin{proof} Let $Y$ be a cylinder of $\Gamma$ with respect to $\sim$.  By \cite[Proposition 3.5.7]{R} $$\pi_1(\Gamma)=\pi_1(Y)\amalg L.$$  Observe that replacing the cylinder by its cone $\Gamma_Y$ we replace   $\pi_1(Y)$ by the fundamental group of the cone which is trivial and so the fundamental group of the resulting graph is $L$. Doing this for every cylinder and taking into account that $\Gamma$ is finite we deduce the result. 

\end{proof}

\begin{corollary}\label{free product of cylinders}
$$\pi_1(\Gamma)=\coprod_Y(\pi_1(Y))\amalg \pi_1(\Gamma_c).$$\end{corollary}

\bigskip
Given an admissible equivalence relation $\sim$ on $E(T)$, we now associate a tree of cylinders
$T_c$ to any $(G,\E)$-tree $T$ with compact $E(T)$.

\begin{definition} (Tree of cylinders). Let $G$ be a pro-$p$ group. Given a pro-$p$ $(G,\E)$-tree $T$  with $E(T)$ compact and an admissible  equivalence relation $\sim$ on $E(T)$, its pro-$p$ tree of
	cylinders (see Theorem \ref{tree of cylinders}) $T_c$ is the bipartite tree with vertex set $V_0(T_c) \cup V_1(T_c)$, where $V_0(T_c)$ is the set of
	vertices $v$ of $T$, $V_1(T_c)=E(T)/\sim$ is the space of cylinders $E(Y)$ of
	$T$, and there is an edge $\epsilon = (v, E(Y) )$ between $v$ and $E(Y)$ in $T_c$ if and only if $v \in Y$ in $T$.
	Equivalently, one obtains $T_c$ from $T$ by replacing each cylinder $E(Y)$ by the cone over the
	set of vertices $v \in Y$.
	 (Warning: $T_c$ is not always an $\E$-tree).
	\end{definition}

\begin{theorem}\label{tree of cylinders}    $T_c$ is a  pro-$p$ tree.
	
\end{theorem}

\begin{proof}  Let $T=\varprojlim_i \Gamma_i$ be a decomposition as an inverse limit of finite $G$-quotient graphs $\Gamma_i$ of $T$. Define an equivalence relation $\sim_i$ on $E(\Gamma_i)$ putting $m_i\sim_i f_i$ if there are elements $m\sim f$ in the corresponding preimages of $m_i$ and $f_i$ respectively and then taking the transitive hull.  By Lemma \ref{inverse limit of relations} $\sim=\varprojlim_i \sim_i$.

  \medskip
  We deduce that $T/\sim=\varprojlim_i \Gamma_i/\sim_i$ and so $\pi_1(T)=\varprojlim_i \pi_1(\Gamma_i)$. Note that morphisms of the inverse system  $\{\Gamma_i\}$ map cylinders into cylinders and so  the homomorphisms of the inverse system $\{\pi_1(\Gamma_i)\}$ send the fundamental groups of cylinders into the fundameltal groups of the cylinders. It follows that free factors corresponding to cylinders in Corollary \ref{free product of cylinders} form the inverse system obtained by restricting inverse system $\{\pi_1(\Gamma_i)\}$ restricted to them.  Then   by Corollary \ref{free product of cylinders}  $\pi_1((\Gamma_i)_{c}) $ is a free factor of  $\pi_1(\Gamma_i)$ and  $\pi_1(T_c)=\varprojlim_i \pi_1((\Gamma_i)_{c})$ is a surjective inverse limit. Thus  we have the following commutative diagram
 
 $$\xymatrix{1=\pi_1(T)\ar[r]\ar[d]&\pi_1(\Gamma_i)\ar[d]\\
        \pi_1(T_c)\ar[r]&\pi_1((\Gamma_i)_{c})}$$
        Since $1=\pi_1(T)=\varprojlim_i \pi_1(\Gamma_i)$ we deduce that the the right vertical maps are surjective, so is the left vertical map. Hence $\pi_1(T_c)=1$, i.e. $T_c$ is a pro-$p$ tree.
\end{proof}

\begin{definition} If the set of edges of $T_c$ with edge stabilizers not in $\E$ is closed, then we construct the  collapsed tree of cylinders $T_c^{*}$ which 
	is the pro-$p$ tree obtained from $T_c$ by collapsing all edges
	whose stabilizer does not belong to $\E$.

\end{definition}

\begin{proposition} Let $T$ be a pro-$p$  $(G,\E)$-tree.
	
	\begin{enumerate}
	\item[(1)] $T$ dominates $T_c$;
	\item[(2)] If $H\le G$ is a vertex stabilizer of $T_c$ which is not elliptic in $T$, it is the stabilizer of
	the equivalence class $[Y ]$ associated to some cylinder $Y \subseteq T$. If $e$ is an edge of $Y$ ,
	the equivalence class of $G_e$ is $H$-invariant
	 ($hG_e
	 h^{-1} 
	\sim G_e$ 
	 for $h \in H$);
	
	\item[(3)] Suppose that the stabilizer of every equivalence class $[Y]$ belongs to $\E$. Then edge
	stabilizers of $T_c$ belong to $\E$.
	\end{enumerate}
	
	\end{proposition}

\begin{proof} Consider a vertex $v$ of $T$. Then it defines a vertex in $V_0(T_c)$,
	and this vertex is fixed by $G_v$. This shows that $T$ dominates $T_c$.
	The second assertion follows from the previous sentence: the stabilizer of a vertex in
	$V_0(T_c)$ is a vertex stabilizer of $T$, the stabilizer of a vertex in $V_1(T_c)$ is the stabilizer of an
	equivalence class $[Y ]$. Also note that (3) is clear since the stabilizer of an edge $e = (v, Y )$
	of $T_c$ is contained in the stabilizer of $[Y ]$.
	
\end{proof}

\begin{example} \label{equality} Suppose $\sim$ is defined by equality stabilizers, i.e.  $e\sim f$ iff $G_e=G_f$ for $e,f\in E(T)$.
	\begin{enumerate}
	\item[(a)] Let $G=G_1\amalg_H G_2$ be a free pro-$p$ product with amalgamation. Hence $G$ is the fundamental group of a tree with one edge. 
	$$\xymatrix{
{\overset{G_1}{\bullet}}\ar[rr]^H&& {\overset{G_2}{\bullet}}}$$

Then $T_c/G$ is a tree with 3 vertices.
$$\xymatrix{
\bullet\ar[rr]&& \bullet&&\bullet\ar[ll]}$$
 The underlying graph of pro-$p$ groups of the $G$-tree $T_c$ is 
 
$$\xymatrix{
{\overset{G_1}{\bullet}}\ar[rrr]^{N_{G_1}(H)}&&& {\overset{N_G(H)}{\bullet}}&&&{\overset{G_2}{\bullet}}\ar[lll]_{N_{G_2}(H)}}$$ 
  Indeed, denoting by $S=S(G)$ the standard pro-$p$ tree of $G=G_1\amalg_H G_2$, we see that the cylinder for $H$ is $\{e\in E(S)\mid G_e=H\}$. But $G_e$ is a conjugate $gHg^{-1}$ of $H$ and so if it contains $H$ it equals $H$ (one sees it looking at finite quotients).  Thus the splitting of $G$ acting on $T_c$ is 
  \begin{equation}\label{cylindrical splitting}G=G_1\amalg_{N_{G_1}(H)} N_G(H)\amalg_{N_{G_2}(H)} G_2.
  \end{equation}
  
  By Proposition \ref{normalizer} $N_G(H)=N_{G_1}(H)\amalg_H N_{G_2}(H)$.  
  Hence  $G$ is the fundamental group of the following tree of pro-$p$ groups    

$$\xymatrix{
{\overset{G_1}{\bullet}}\ar[rr]^{N_{G_1}(H)}&& {\overset{N_{G_1}(H)}{\bullet}}\ar[rr]^H&&{\overset{N_{G_2}(H)}{\bullet}}&&{\overset{G_2}{\bullet}}\ar[ll]_{N_{G_2}(H)}}$$ 
and then collapsing the left and right edges we arrive at the original decomposition. 

\bigskip
If $G_1, G_2$ are finite then the last and original splittings are JSJ splittings. This illustrates Subsection 7.2.  Note also that if $H$ is normal in $G_1$ and $G_2$ then $T_c$ is totally reduced, i.e. after reduction we get the trivial graph of groups. This shows that if $G_1,G_2$ are finite, then the sequence $((T_c)_c)_{\ldots}$ gives eventually graph of  groups that admits complete reduction to the trivial graph of groups.
	
	\bigskip
	\item[(b)] Suppose now that $G=G_1\amalg_H$ is a pro-$p$ HNN-extension. Then   $G$ is the fundamental pro-$p$ group of the following graph of pro-$p$ groups:
	
$$\xymatrix{
{\overset{G_1}{\bullet}}\ar@(dl,dr)_{H}}$$	

	In this case $T_c/G$ is a graph with two vertices but the number of edges can be 1   and two. Denote by $t$ the stable letter of $G$.
	
	Case 1. $H$ and $H^t$ are not conjugate in $G_1$. Then $T_c/G$ has two edges.
	$$\xymatrix{
\bullet\ar@/^/[rr]\ar@/_/[rr]&& \bullet}$$
The underlying graph of pro-$p$ groups of the $G$-tree $T_c$ then  is 

$$\xymatrix{
{\overset{G_1}{\bullet}}\ar@/^/[rrr]^{N_{G_1}(H)^t}\ar@/_/[rrr]_{N_{G_1}(H)}&&& {\overset{N_G(H)}{\bullet}}}$$
where upper arrow embedding is the conjugation by the stable letter $t$. 

By Proposition \ref{normalizer},
in this case the normalizer splits as in the previous case: $N_G(H)={N_{G_1}(H)}\amalg_H {N_{G_1}(H^t)}$ and so we can decompose the graph of groups as follows:
$$
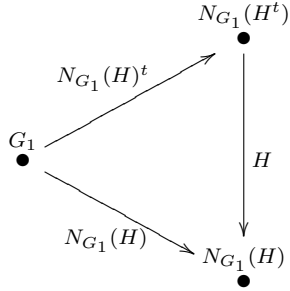
\xymatrix{&&{\overset{N_{G_1}(H^t)}{\bullet}}\ar[dd]^H\\
{\overset{G_1}{\bullet}}\ar[rru]^{N_{G_1}(H)^t}\ar[rrd]_{N_{G_1}(H)}&&\\
&&{\overset{N_{G_1}(H)}{\bullet}}}$$
and then collapsing the upper and lower edges we arrive at the original decomposition.

\bigskip

Case 2. $H$ and $H^t$ are  conjugate in $G_1$ say $H^{g_1}=H^t$ for some $g_1\in G_1$. Then $tg_1^{-1}$  normalizes $H$.  In this case $T_c/G$ has one edge.
	$$\xymatrix{
\bullet\ar[rr]&& \bullet}$$
The underlying graph of pro-$p$ groups of the $G$-tree $T_c$ then  is 

$$
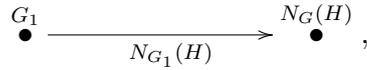
\xymatrix{
{\overset{G_1}{\bullet}}\ar[rrr]_{N_{G_1}(H)}&&& {\overset{N_G(H)}{\bullet}}},$$

By Proposition \ref{normalizer}
in this case $N_G(H)={N_{G_1}(H)}\amalg_H $ is an HNN-extension. So the graph of groups can be transformed as follows:
$$
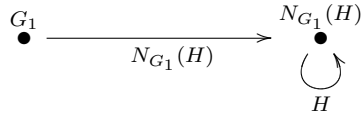
\xymatrix{{\overset{G_1}{\bullet}}\ar[rrr]_{N_{G_1}(H)}&&& {\overset{N_{G_1}(H)}{\bullet}}\ar@(dl,dr)_{H}}$$
Collapsing the horizontal arrow we can get the original decomposition.

\bigskip
If $G_1$ is finite then in both cases the original and the last splitting is JSJ. This is another illustration of Subsection 7.2.

\item[(c)] In general, if we have the following graph of groups  

$$
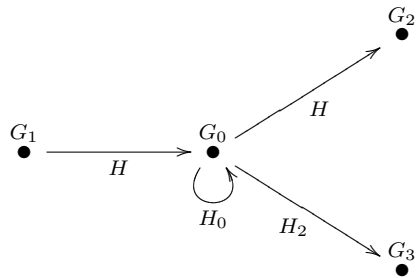
\xymatrix{&&&& {\overset{G_2 }{\bullet}}\\
{\overset{G_1}{\bullet}}\ar[rr]_H&& {\overset{G_0 }{\bullet}}\ar[rru]_H \ar[rrd]_{H_2}\ar@(dl,dr)_{H_0}\\  
&&&&{\overset{G_3}{\bullet}}}$$
and $G$ is the fundamental group of it and if the images of $H_0$ in $G_0$ are not conjugate, 
then $T_c/G$ is as follows:
$$
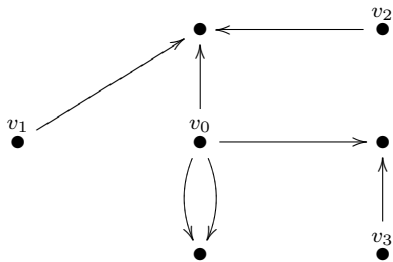
\xymatrix{&&\bullet&& {\overset{v_2 }{\bullet}}\ar[ll]
\\
{\overset{v_1}{\bullet}}\ar[rru]&& {\overset{v_0 }{\bullet}}\ar[u] \ar[rr]\ar@/^/[d]\ar@/_/[d]&&\bullet\\  
&&\bullet&&{\overset{v_3}{\bullet}}\ar[u]}$$
where $v_i$ belong to $V_0(T_c)$ and are the vertices that correspond to vertices under $G_i$ in the original graph above.

If the images of $H_0$ in $G_0$ are conjugate then two vertical edges from $v_0$ down are replaced by one edge.

\bigskip The vertex groups of $v_i$ are $G(v_i)$, the edge group of an  edge $e$ incident to $v_i$ is $N_{G(v_i)}(G(e))$ the vertex group of the cyclinder of $e$ is $N_G(G(e))$.  

 Similarly as in case (a) that if all vetrex groups are finite, then the sequence $((T_c)_c)_{\ldots}$ gives eventually graph  of groups that admits complete reduction to the trivial graph of groups.
	
	\end{enumerate}
	
	\end{example}

\subsection{Algebraic Interpretation}

Fix $\E$ now closed in the \'etale  topology of $subgr(G)$ and we as usual restrict to $\E$-trees.

 It  is sandwich closed: if $A, B \in \E$  and $A < H < B$, then $H \in \E$.

\begin{definition}\label{admissible} (Admissible equivalence relation). An  equivalence relation  $\sim$ on $\E$ (relative to $\mathcal H$) is
admissible  if the following axioms hold for any $A, B \in \E$:

(0) $\sim$ is an intersection of clopen equivalence relations.

(1) If $A \sim B$ and $g \in G$, then $gAg^{-1} \sim gBg^{-1}$
(invariance under conjugation).

(2) If $A \subset B$, then $A\sim B$ (nesting implies equivalence).

(3) Let $T$ be a $(\E,\mathcal H)$-tree. If $A \sim B$, and $A, B$ fix
$v, w \in  T$ respectively, then for each edge $e \subset [v, w]$ one has $G_e \sim A \sim B$.

The equivalence class of $A \in \E$ will be denoted by $[A]$. We let $G$ act on $\E/\sim$ by
conjugation, and the stabilizer of $[A]$ will be denoted by $G_{[A]}$.\end{definition}

\begin{lemma}\label{3'} If $\sim$ satisfies 1,2 and axiom (3’) below, then it is admissible.

(3’). Let $T$ be an $\E$-tree. If $A\sim B$, and $A, B$ are elliptic in $T$, then $\langle A, B\rangle$ is also elliptic in $T$.\end{lemma}

\begin{proof} We show axiom 3. Let $v,w$ be vertices fixed by $A$ and $B$ respectively and let $u$ be a vertex fixed by $\langle A, B\rangle$. Since $[v, w]\subset [v, u]\cup [u, w]$, one may assume $e\in [v,u ]$. One has $A\leq G_e$ because $A$ fixes $[v, u]$ (see Theorem \ref{fixed geodesic}), so that $G_e\sim A$ by axiom 2. \end{proof}

Here are some examples. 

(1) If $G$ is a  CSA pro-$p$ group (for instance a pro-$p$ limit group, see Section 3 for definitions), we can take for $\E$ the set of  abelian subgroups,
and for $\sim$ the commutation relation: $A \sim B$ if $\langle A, B\rangle$ is abelian). Note that this relation is the intersection of open equivalence relations containing it, since the space of subgroups $$subgr(G)=\varprojlim_{N\triangleleft_o G}subgr(G/N)$$ with {\it strict} topology (i.e. the topology coming from this inverse limit which is stronger than \'etale topology, see Remark \ref{strict topology}) is profinite, and it is preserved by finite quotients. The set $\E$ in this case is closed in the  strict topology of $subgr(G)$. The group $G_{[A]}$
is
the maximal abelian subgroup containing $A$.

\bigskip
(2) $\E$ is a set of infinite  cyclic-by bounded order subgroups, and $\sim$ is the commensurability
relation ($A \sim B$ if and only if $A \cap B$ has finite index in $A$ and $B$). Note that $\E$ is compact. Here again it is the intersection of clopen equivalence relations as one can check directly. The group $G_{[A]}$
is the closure of the commensurator of $A$. Note that the commensurator $G_{[A]}$ is closed in $G$ if $G$ is finitely generated (see Theorem \ref{commensurator}) and otherwise it might not be.

\bigskip
(3) $\E$ is a set of finite subgroups, and $\sim$ is the equality
relation ($A \sim B$ if and only if $A=B$). Suppose $\E$ is compact (for example  $G$ is virtually torsion free). Here again it is the intersection of clopen equivalence relations as one can check directly. The group $G_{[A]}$
is the normalizer of $A$. 

\bigskip
In fact, we are interested in restricting the equivalence relation on edge stabilizers, or more to the point to induce this relation on edges of $T$. 

\bigskip
Given an admissible equivalence relation $\sim$ on $\E$, we now associate a tree of cylinders
$T_c$ to any $(G,\E)$-tree $T$ with compact $E(T)$.
We declare elements $e, f\in E(T)$ to be equivalent (and write $e\approx f$) if $G_e \sim G_f$.

\begin{lemma}\label{cylinders} Let $G$ be a pro-$p$ group acting on a profinite graph $T$ with closed $E(T)$ and $\E$ be a family of subgroups of $G$ closed in the \'etale topology of $Subgr(G)$ and  for conjugation. Let $\sim$ be an admissible equivalence relation on $\E$.  Then
	
	\begin{enumerate}
		\item[(i)] $\{e\in E(T)\mid G_e\in \E\}$ is a closed subset of $E(T)$.

\item[(ii)] The equivalence relation $\approx$ on $E(T)$ is an admissible profinite equivalence relation.

\end{enumerate}
\end{lemma} 

\begin{proof}

(i) Follows from Lemma \ref{closedness}.
	
	(ii)  By hypothesis $\sim$ is the intersection of clopen equivalence relations $R$ on $Subgr(G)$. Then $R$ determines on $E$ the equivalence relation $\sim_R$ by saying $e\sim_R f$ if $G_e R G_f$. Since $\{G_e\mid e\in E(T)\}$ is a continuous family, every equivalence class of $\sim_R$ is clopen.  Clearly $\approx=\bigcap_R \sim_R$. This shows that $\approx$ is profinite.

	Axiom \ref{admissibleontree}(1) is clear. 
	
	 
	Axiom \ref{admissibleontree}(2) follows from  
	Axiom \ref{admissible}(2) as well as Axiom \ref{admissibleontree}(3) follows from  
	Axiom \ref{admissible}(3.)
	\end{proof}

The equivalence class in $\E/\sim$ containing $Y$
 will be denoted by $[Y ]$.


\begin{corollary}  $ E(T)/\approx$ is a profinite space.
\end{corollary}

\begin{proof} As $\approx$ is profinite by Lemma \ref{cylinders}, $\approx$ is the intersection of clopen equivalence relations $R$.  Hence $E(T)/\approx=\varprojlim_R E(T)/R$ is profinite.
	
\end{proof}

 We record the following for future references.

\begin{remark}\label{edges vs subgroups} Given an edge $\epsilon= (x, Y)$ of $T_c$, let $e$ be an edge of $Y$ adjacent to $x$ in $T$. Then $G_e$ is a representative of the class $C$ of groups which are contained in $G_\epsilon=G_x\cap G_Y$. If $G_\epsilon$ belongs to $\E$, then it is in $C$ by axiom 2 of admissible relations. Also note that $G_Y$ represents $C$ if $G_Y\in \E$. \end{remark}

\begin{theorem}\label{trees of cylinders}  If $T, T'$ are reduced  $\E$-trees  belonging to the same deformation space such that $T/G$ and $T'/G$ are finite, then there is a canonical equivariant isomorphism between their trees of cylinders.\end{theorem}

\begin{proof}  If $v\in V(T)$ is a  vertex of $T$ fixed by $G_v$, then $G_v$ fixes a unique vertex $v'\in V(T')$ up to conjugation. Indeed, if $w'\neq v'$ is another vertex fixed by $G_v$ then $G_{v'}$ and $G_{w'}$ must fix $gv$ and $hv$, since  $T$ is reduced; but then  $ G_{v'}^g,G_{w'}^h=G_v$ so $v'=kw'$ for some $k\in G$ since $T'$ is reduced.

Thus we have a bijection $\bar f:V(T)/G\longrightarrow V(T')/G$ between finite sets that sends the orbit $vG$ of $v$  to the orbit $v'G$ of a vertex in $T'$ fixed by $G_v$. It induces $G$-equivariant continuous map $V(T)\longrightarrow V(T')$. 

We show now one-to-one correspondence between $E(T)/\sim$ and $E(T')/\sim$.
If $e$ is an edge of $T$ with vertices $v,w$ then $G_v, G_w$ fix  unique distinct vertices $v',w'$ of $T'$ and so $G_e$ fixes $[v',w']$ (see Theorem \ref{fixed geodesic}). But similarly, any edge $e'$ of the geodesic $[v',w']$ fixes an edge $e_0$ in $T$, so  by condition (2) of Definition \ref{admissible} $e\sim e_0$.  It follows that the correspondence $e$ to an edge in $T'$ stabilized by $G_e$ defines a bijection $\varphi$ on the spaces of cylinders $\overline E(T)/\sim$ and $\overline E(T')/\sim$.

Now let $v\in V_0(T_c), v'\in V_0(T'_c)$ such that $G_v=G_{v'}$. Let $Y\in V_1(T_c)$ connected to $v$ by $\epsilon$. Let $Y'$ be the corresponding cylinder of $V_1(T'_c)$, i.e. corresponding to the same class in $\E$. Then the map $(f,\varphi):(v, Y)\longrightarrow (v',Y')$ is a homeomorphism. This gives the required isomorphisms between $T_c$ and $T'_c$ which is clearly $G$-equivariant.\end{proof}

Theorem \ref{trees of cylinders} says that if $T$ is reduced then $T_c$    depends only on the deformation space $D$ containing $T$ (we sometimes say
	that $T_c$ is the tree of cylinders of $D$).
	
	\medskip
	Theorem \ref{trees of cylinders} also shows that the graph $T_c$ can be defined purely in terms of the maximal elliptic subgroups (the  vertex stabilizers), which are the same for $T$ and $T'$. 

\medskip	
If $G$ is finitely generated then by Example \ref{changing jsj-tree} there  is a reduced $\E$-tree  in the same deformation space (but possibly we lose compactness of the space of edges).  To keep compactness of the space of edges we need to assume that $G$ is $\E$-accessible. In this case $T_c/G$ is finite as we record in the following

\begin{corollary}\label{cofiniteness of tree of cylinders} Suppose $G$ is $\E$-accessible. Then $T_c/G$ is finite.

\end{corollary}

\begin{proof} By Corollary \ref{finiteness} there exists a JSJ-tree $T$ such that $T/G$ is finite. Then so is $T_c/G$ by Proposition \ref{finite diameter}.

\end{proof}

\begin{remark}\label{via equality} Note that any admissable equivalence relation contains diagonal, i.e. the equality relation. Therefore $T_c$ always factors through the pro-$p$ tree  of cyliders of the equality relation $T_=$. Thus costructing $T_c$ we may assume that   $T/G$ has no loops (see Example \ref{equality}(b), Case 2). This shows that the canonical $G$-isomophism constructed in the proof of  Theorem \ref{trees of cylinders} is unique.\end{remark}

\bigskip
Depending on the context, it may be convenient to think of a cylinder either as set of edges $E(Y)$ of a subtree $Y$ of $T$, or a vertex of $T_c$, or an equivalence class $C$. Similarly, there are several ways to think of $x\in V_0(T_c)$: as a vertex of $T$, a vertex of $T_c$, or an elliptic subgroup $G_x$. If $x$ is a vertex of $T$ belonging to two cylinders, then its stabilizer $G_x$ is not contained in a group of $\E$: otherwise all edges of $T$ incident to $x$ would have equivalent stabilizers by axiom 2, and $x$ would belong to only one cylinder. More generally, let $v$ be any vertex of $T$ whose stabilizer is not contained in a group of $\E$. Then $G_v$ fixes $v$ only, and is a maximal elliptic subgroup.

Conversely, let $H$ be a subgroup of $G$ which is elliptic in $T$,  not contained in a group of $\E$, and  maximal for these properties. Then $H$ fixes a unique vertex $v$ and equals $G_v$.
 
 \bigskip
 In the next theorem we are going to make the construction of  the pro-$p$ tree of cylinders in the case of non-compact space of edges assuming that $G$ is finitely generated. We shall call it also the tree of cylinders as it will coincide with the tree of cylinders defined in this section for the compact set of edges.  It has the same features in the sense that the space of vertices of it consists of vertices of original tree and the profinite space of cylinders in which  the set of original cylinders is dense. Recall again that in this case there is always a reduced pro-$p$ tree in the same deformation space with possibly not compact set of edges. 
 
 \begin{theorem}\label{canonical} Let $G$ be a finitely generated  pro-$p$ group acting on a pro-$p$ $\E$-tree $T$ (with not obligatory compact $E(T)$). Then there exists a canonical  pro-$p$ tree of cylinders $T_{cc}$ that depends on   the deformation space of $T$ only.  
 
 \end{theorem}
 
 \begin{proof} By Theorem \ref{pro-pbass-serre}  $G$ is the fundamental pro-$p$ group of a profinite graph of
	pro-$p$ groups $(\G, \Gamma)$ with vertex groups and edge groups  being conjugates of vertex and edge stabilizers in $G$.   By Lemma \ref{inverse limit of virtually free groups} $G=\varprojlim_{U\triangleleft_o G} G/\widetilde U$ and $G_U:=G/\widetilde U=\Pi_1(\G_U,\Gamma_U)$ is the fundamental group of a finite reduced graph of  finite $p$-groups. Moreover, the inverse system $\{G/\widetilde U, \pi_{VU}\}$ can be chosen in such a way that it is linearly ordered and for each $\{G/\widetilde V\}$ of the system with $V\leq U$  there exists a natural morphism $$(\eta_{VU},\nu_{VU}):(\G_V, \Gamma_V)\longrightarrow (\G_U, \Gamma_U)$$ where $\nu_{VU}$ is just a collapse of edges of $\Gamma_V$ and $\eta_{VU}(\G_V(m))=\pi_{VU}(\G_V(m))$; the induced  homomorphism of the
pro-$p$ fundamental groups coincides with the canonical projection
$\pi_{VU}\colon G/\widetilde V\longrightarrow G/\widetilde U$ and $(\G,\Gamma)=\varprojlim_{U\triangleleft_o G}(\G_U,\Gamma_U)$.  Let $S_U$ be the standard pro-$p$ tree for $G_U=\Pi_1(\G_U,\Gamma_U)$ and recall that $E(S_U)$ is compact (since $S_U/G_U$ is finite). Then $S=\varprojlim_{U\triangleleft_o G} S_U$ and the projection $S\longrightarrow S_U$ is a collapse of connected components of certain profinite subgraph of $S$.

Note that $G_U$ has only finitely many edge stabilizers up to conjugation. 
Define an equivalence relation $\sim_U$ on $E(S_U)$ putting $m_U\sim_U f_U$ if either the stabilizer of one of them contains the other or there are edges $m, f\in E(T)$ in the preimages of $m_U,f_U$ respectively such that $G_m\sim G_f$. We show that this is admissible equivalence relation on $E(S_U)$. First observe that it is obviously $G$-invariant and so (i) of Definition \ref{admissibleontree} holds. Since $S_U/G_U$ is finite, this implies that $\sim_U$ is profinite. The nesting property (2) follows from the definition of $\sim_U$.  To prove (iii) suppose $f_U\sim f'_U$  and $e_U\in [f_U,f'_U]$.  If the stabilizer of $f_U$ contains the stabilizer of $f'_U$ or vice versa, then either $ (G_U)_{f'_U}\leq  (G_U)_{e_U}\leq  (G_U)_{f_U}$  or $ (G_U)_{f_U}\leq  (G_U)_{e_U}\leq  (G_U)_{f'_U}$   by Theorem \ref{fixed geodesic}, so that $e_U\sim f_U\sim f'_U$ by the nesting property; otherwise there are edges $f, f'\in E(T) $ such that $G_{f}\sim G_{f'}$. In the latter case, since $S\longrightarrow S_U$ is a collapse map, there exists an edge $e\in [f,f']$  mapped to $e_U$ and so $G_{f}\sim G_{e}\sim G_{f'}$; in this case then      $f_U\sim e_U\sim f'_U$ by the definition of $\sim_U$. Thus $\sim_U$ is admissible.

Let $(T_{U})_c$ be the  pro-$p$ tree of cylinders for $S_U$ (it exists as $E(T_U)$ is compact). Then  for each $V\leq_o U$ normal in $G$, $e\sim_V e'$ implies $\pi_{VU}(e)\sim_U \pi_{VU}(e')$ and so we have a unique $G$-epimorphism $(T_{V})_c\longrightarrow (T_{U})_c$ induced by a morphism $S(G_V)\longrightarrow S(G_U)$ . This means that  $\{(T_{U})_c\mid U\triangleleft_o G\}$ is an inverse system. The pro-$p$ tree $T_{cc}=\varprojlim_{U\triangleleft_o G} (T_{U})_c$ is the desired   pro-$p$ tree of cylinders.

Now we shall prove the uniqueness. Suppose $T'$ is from the same deformation space. Then similarly to the above  $G=\varprojlim_{U\triangleleft_o G} G/\overline U$, where $\overline U$ means the subgroup of $U$ generated by the  stabilizers of vertices of $T'$,  and $G/\overline U=\Pi_1(\G'_U,\Gamma'_U)$ is the fundamental group of a finite reduced graph of  finite $p$-groups. But $G/\widetilde U$ and $G/\overline U$ have the same set of vertex stabilizers and so $T/\widetilde U$ and $T/\overline U$ belong to the same deformation space.
Hence by Theorem \ref{trees of cylinders} there exists a canonical isomorphism $(T_c)_U\longrightarrow (T'_C)_U$ (as vertex stabilizer determines uniquely the vertex by Remark \ref{via equality}) for every such $U$. This gives the isomorphism $T_{cc}\cong T'_{cc}$.
	\end{proof}
	
	
	
	\begin{example} 
	 If $G$ is a torsion free CSA pro-$p$ group (for instance a pro-$p$ limit group, we can take for $\E$ the set of  abelian subgroups,
and for $\sim$ the commutation relation: $A \sim B$ if $\langle A, B\rangle$ is abelian.  The group $G_{[A]}$
is
the maximal abelian subgroup containing $A$. The $T_c$ has maximal abelian subgroups as its edge stabilizers.

	\end{example} 
	
\begin{example}\label{of Wilkes} We shall illustrate Theorem \ref{canonical} recalling Example \ref{Wilkes}.  

There was constructed the following JSJ graph of pro-$p$ groups

\begin{center}
			\begin{tikzpicture}[every edge/.append style={nodes={font=\scriptsize}}]
				\node (0) at (0,0) [label=above:{\scriptsize $G_1$},point];
				\node (1) at (2,0) [label=above:{\scriptsize $G_2$},point];
				\node (2) at (4,0) [label=above:{\scriptsize $G_{3}$},point];
				\node (3) at (5,0) {$\cdots$};
				\node (4) at (6,0) [label=above:{\scriptsize $H_{\omega}$},point];
				
				\draw[->] (0) edge node[below] {$K_1$} (1);
\draw[->] (1) edge node[below] {$K_2$} (2);				
				
			\end{tikzpicture}
		\end{center}

Let $\sim$ be the equality. Then the underlying profinite graph of pro-$p$ groups for $T_{cc}$ is 

\begin{center}
			\begin{tikzpicture}[every edge/.append style={nodes={font=\scriptsize}}]
				\node (0) at (0,0) [label=above:{\scriptsize $G_1$},point];
				\node (1) at (3,0) [label=above:{\scriptsize $N_{\langle G_1,G_2\rangle}(K_1)$},point];
				\node (2) at (5,0) [label=above:{\scriptsize $G_{2}$},point];
				\node (3) at (6,0) {$\cdots$};
				\node (4) at (7,0) [label=above:{\scriptsize $H_{\omega}$},point];
				
				\draw[->] (0) edge node[below] {$N_{G_1}(K_1)$} (1);
\draw[->] (1) edge node[below] {$N_{G_2}(K_1)$} (2);				
				
			\end{tikzpicture}
		\end{center}
The situation is similar as in Example \ref{equality}.
But in this case $N_{G_i}(K_i)$ coincides with $G_i$ since $K_i$ is normal in $G_i$. Hence we can reduce this graph of pro-$p$ groups

\begin{center}
			\begin{tikzpicture}[every edge/.append style={nodes={font=\scriptsize}}]
				\node (0) at (0,0) [label=above:{\scriptsize $N_{\langle G_1,G_2\rangle}(K_1)$},point];
				\node (1) at (3,0) [label=above:{\scriptsize $N_{\langle G_2,G_3\rangle}(K_2)$},point];
				\node (2) at (6,0) [label=above:{\scriptsize $N_{\langle G_3,G_4\rangle}(K_1)$},point];
				\node (3) at (7,0) {$\cdots$};
				\node (4) at (8,0) [label=above:{\scriptsize $H_{\omega}$},point];
				
				\draw[->] (0) edge node[below] {$N_{G_2}(K_1)$} (1);
\draw[->] (1) edge node[below] {$N_{G_3}(K_2)$} (2);				
				
			\end{tikzpicture}
		\end{center}

 Note that all vertex groups except the limit vertex group were finite in the original graph of groups, but in the reduced graph of cylinders all vertex groups are infinite. All edge groups remain finite.

\end{example}
	
	\begin{proposition}\label{properties of tree of cylinders} Let $G$ be a finitely generated or $\E$-accessible pro-$p$ group acting on a pro-$p$ tree $G$. Then 
	
	\begin{enumerate}
	\item[(1)] $T$ dominates $T_c$;
	\item[(2)] If $H\le G$ is a vertex stabilizer of $T_c$ which is not elliptic in $T$, it is the stabilizer of
	a vertex in $V_1(T_c)$ associated to an equivalence class $[Y ]$ of $\E$. If $e$ is an edge of $Y$ ,
	the equivalence class of $G_e$ is $H$-invariant
	 ($hG_e
	 h^{-1} 
	\sim G_e$ 
	 for $h \in H$);
	\item[(3)]  $T_{cc}$ constructed in Theorem \ref{canonical}  depends only on the deformation space $D$ containing $T$ (we sometimes say
	that $T_{cc}$ is the canonical pro-$p$ tree of cylinders of $D$);
	\item[(4)] Suppose that the stabilizer of every equivalence class $[A]$ belongs to $\E$. Then edge
	stabilizers of $T_c$ belong to $\E$.
	\end{enumerate}
	
	\end{proposition}

\begin{proof} Consider a vertex $v$ of $T$. By Theorem \ref{pro-pbass-serre} $G=\Pi_1(\G, \Gamma)$ is the fundamental group of a profinite graph of pro-$p$ groups and the standard pro-$p$ tree $S=S(G)$ lie in the same deformation space as $T$. Moreover, according to Lemma \ref{inverse limit of virtually free groups} $(\G,\Gamma)=\varprojlim_{U\triangleleft_o G}(\G_U,\Gamma_U)$, that induces $\Pi_1((\G,\Gamma)=\varprojlim_{U\triangleleft_o G}\Pi_1(\G_U,\Gamma_U)$ and $S=\varprojlim_{U\triangleleft_o G} S_U$, where $S_U$ is the  standard pro-$p$ tree of $\Pi_1(\G_U,\Gamma_U)$.  
So $V(S)=V_0(T_c)$ and so $S$ and therefore $T$   dominates $T_c$.

	The second assertion follows from remarks made above: the stabilizer of a vertex in
	$V_0(T_c)$ is a vertex stabilizer of $T$, the stabilizer of a vertex $v$ in $V_1(T_c)$ fixes the vertex $v_U$ in $V_1(T_{cU})$ that corresponds to the equivalence class of  $\sim_U$ in $G_U$ (see the proof of Theorem \ref{canonical}). Therefore to deduce the assertion about $V_1(T_c)$ we just need to show that $\sim$ is the inverse limit of $\sim_U$. By \cite[Construction 4.3]{HZ} there exists a $G$-space  $Y_\E$ such that the set of point stabilizers coincides with $\E$ and similarly $G_U$-space $Y_U$ whose set of point stabilizers coincides  with the image $\E_U$ of $\E$ in $Subgr(G_U)$ and those spaces are functorial. This means that $Y_\E=\varprojlim Y_U$. Then the result follows from Lemma \ref{inverse limit of relations}.

	 Also note that (4) is clear since the stabilizer of an edge $e = (v, Y )$
	of $T_c$ is contained in the stabilizer of $[Y ]$.
	
	(3) is the subject of Theorem \ref{canonical}
\end{proof}

\subsection{Automorphisms}
	
	\begin{lemma}\label{vertex stabilizers invariant} Let $G$ be a pro-$p$ group acting on a reduced pro-$p$ $\E$-tree $T$. Suppose the set $\{G_v\mid v\in V(T)\}$  and $\sim$ are invariant for any automorphism of $G$. 
\begin{enumerate}
\item[(i)]  If $G$ is $\E$-accessible, then $Aut(G)$ acts on $T_c$.	
\item[(ii)] If $G$ is a finitely generated, then  $Aut(G)$ acts on $T_{cc}$.
\end{enumerate}	
	\end{lemma}
	
	\begin{proof}

	\medskip
	(i) By Remark \ref{reduction} we may assume that  $T$ is reduced. Then by Remark \ref{via equality} we may assume that $T/G$ have no loops and so a vertex stabilizer determines uniquely the vertex fixed by it.  It follows that for $\alpha\in Aut(G)$, defining $\alpha(v)$ to be the unique vertex whose stabilizer is $\alpha(G_v)$ if $v\in V_0(T_c)$ we get the action of $Aut(G)$ on $V_0(T_c)$. Let $e$ be an edge of a  cylinder $Y$, i.e. $G_e$ belongs to the equivalence class $[Y]$ corresponding to $Y$ (see Proposition \ref{properties of tree of cylinders}). Then $\alpha(G_e)\in \alpha([Y])$ and so this defines the action of $Aut(G)$ on $V_1(T_c)$. Moreover, $\alpha(G_e)\in \alpha([Y])$ means that if the edge $\epsilon$ connects $v$ and $Y$ then there is an edge that connects $\alpha(v)$ with $\alpha(Y)$ that we define as $\alpha(\epsilon)$. Thus  we get an  action of $Aut(G)$ on $T_c$ and we need to show that it is continuous.
	
	   We show first that it is continuous on $V_0(T_c)$.  It suffices to show that for given $v\in V(T)$ the map $v\longrightarrow G_v$ is a homeomorphism  from $V_0(T_c)$ to 
	   $\{G_v\mid v\in V_0(T_c)\}$ considered as a subspace in the strict topology of $Subgr(G)$. It is clearly a bijection, so one needs to show continuity only. Taking an open normal subgroup $U$ of $G$ we deduce that the preimage of $\{G_w\mid G_w\leq G_vU\}$ is open in $V_0(T_c)$ since $\{G_v\mid v\in V_0(T_c)\}$ is a continuous family by Lemma \ref{Ribes 5.2.2}. This shows continuity of the map $v\longrightarrow G_v$. 
	 
	 Considering now $E(T)$ as a subspace of $V(T)\times V(T)$  and so $\sim$ as the $Aut(G)$-invariant equivalence relation on $V(T)\times V(T)$  we deduce that $Aut(G)$ acts continuously on $(V(T)\times V(T))/\sim$ and hence on the subspace of cylinders $E(T)/\sim$.  
	
\medskip
(ii) 	Suppose now that $G$ is finitely generated. Let $\mathcal U=\{U\triangleleft_o G\}$ be the family of all open characteristic subgroups of $G$. Then $\widetilde U=\{U\cap G_v\mid v\in V(T)\}$ is characterisrtic in $G$ and we can consider the action of $G_U=G/\widetilde U$ on a pro-$p$ tree $T/\widetilde U$ (see Proposition \ref{mod tilde}). Note that in Theorem \ref{canonical} $T_{cc}$ was defined as the inverse limit of $(T_U)_c$ arised from th action of $G_U$ on standard pro-$p$ trees $S(G_U)$ from the same deformation space as $T_U$. Since $S(G_U)/G_U$ are finite,  we deduce from (i) that $Aut(G_U)$ acts continuously on   $(T_U)_c$.

	 But $Aut(G)=\varprojlim_{U\in\mathcal U}(Aut(G_U)$ and $T_{cc}=\varprojlim_{U\in\mathcal U} (T_U)_c$. Thus one has a continuous  action of $Aut(G)$ on $T_{cc}$.

	\end{proof}

	\begin{theorem}\label{Aut} Let $G$ be a finitely generated (resp. $\E$-accessible) pro-$p$ group acting on a JSJ pro-$p$ $\E$-tree $T$. Suppose $\E$  is invariant for any automorphism of $G$. Then $Aut(G)$ acts on $T_{cc}$ (resp. on $T_c$).
	
	\end{theorem}
	
	\begin{proof} Let $\alpha\in Aut(G)$ be an automorphism. We can define another action of $G$ on $T$ by setting $g\cdot m=\alpha(g)m$ for $g\in G$, $m\in T$. We call the new $G$-tree $T'$ and observe that $T'$ is JSJ $\E$-tree.  Hence $\alpha(G_v)$ must stabilize a vertex of $T$. Thus $Aut(G)$ leave vertex stabilizers invariant and so the result follows from Lemma \ref{vertex stabilizers invariant}.
	
	\end{proof} 
	
	\begin{corollary}  If $G$ is $\E$-accessible (resp. finitely generated) and $T$ 
is a JSJ $(\K,\E)$-tree, then $T_c$ (resp. $T_{cc}$) and $T^*_c$ (resp. $T^*_{cc}$) are invariant under any automorphism of $G$ which preserves $\E$ and $\K$.\end{corollary}

In the next theorem for a subgroup $K\leq G$ we use notation $Aut_G(K)$ to be the group of all automorphisms of $G$ that leave $K$ invariant.	
	
	\begin{theorem}\label{amalgam} Let $G=G_1\amalg_H G_2$ be a non-fictitious free pro-$p$ product with amalgamation. Suppose that $G_1$ and $G_2$ are rigid. Then $Aut(G)$ is either 
$$Aut(G)= Aut_G(G_1)\amalg_{ Aut_G(G_1)\cap Aut_G(N_G(H))}Aut_G(N_G(H))\amalg_{Aut_G(G_2)\cap Aut_G(N_G(H))} Aut_G(G_2)$$ 
or
$$Aut(G)= Aut_G(G_1)\amalg_{ Aut_G(G_1)\cap Aut_G(N_G(H))}Aut_G(N_G(H)),$$
where 1 can be replaced by 2.

	\end{theorem}
	
	\begin{proof} Let $S=S(G)$ be the standard pro-$p$ tree for $G=G_1\amalg_H G_2$. We define $\sim$ as in Example \ref{equality}(a):  edges are equivalent if their stabilizers are equal. Then $T_c/G$ is a tree with 3 vertices. The underlying graph of pro-$p$ groups of the $G$-tree $T_c$ is 
 
$$\xymatrix{
{\overset{G_1}{\bullet}}\ar[rrr]^{N_{G_1}(H)}&&& {\overset{N_G(H)}{\bullet}}\ar[rrr]^{N_{G_2}(H)}&&&{\overset{G_2}{\bullet}}}$$ (Example \ref{equality}(a)).  By Lemma \ref{vertex stabilizers invariant} $Aut(G)$ acts on $T_c$.  The outer automorphism group $Out(G)$ acts either trivially on $V(T_c)/G$ or by swapping two vertices of $V_0(T_c)/G$  (leaving the middle vertex $V_1(T_c)/G$ fixed).  It follows that $T_c/Aut(G)$ is either $T_c/G$ or a tree with two vertices only. 

\medskip
Case 1. 	$T_c/Aut(G)=T_c/G$. Lift it to a subtree $T'=\{v_1,v_2,w\in V(T_c), e_1, e_2\in E(T_c)\}$ of $T_c$  with $G_{v_1}=G_1$, $G_{v_2}=G_2, G_w=N_G(H)$ ($w$ is the cylinder of $H$) $e_1$ connecting $v_1$ and $w$, $e_2$ connecting $v_2$ and $w$.  Since $G=G_1\amalg_H G_2$ is  non-fictitious, $S$ is reduced and so  $Aut(G)_{v_1}=Aut_G(G_1)$, $Aut(G)_{v_2}=Aut_G(G_2)$. Moreover, the cyclinder is $N_G(H)$  and its stabilizer $Aut(G)_w=Aut_G(N_G(H))$. Hence $ Aut(G)_{e_1}= Aut_G(G_1)\cap Aut_G(N_G(H))$, $ Aut(G)_{e_2}= Aut_G(G_2)\cap Aut_G(N_G(H))$. So $$Aut(G)= Aut_G(G_1)\amalg_{ Aut_G(G_1)\cap Aut_G(N_G(H))}Aut_G(N_G(H))\amalg_{Aut_G(G_2)\cap Aut_G(N_G(H))} Aut_G(G_2).$$

\medskip

Case 2. $T_c/Aut(G)\neq T_c/G$ (in this case $G_1$ and $G_2$ are isomorphic). Lift it to a subtree $T'=\{v_1,w\in V(T_c), e\in E(T_c)\}$ of $T_c$  with $G_{v_1}=G_1$, $ G_w=N_G(H)$ ($w$ is the cylinder of $H$) $e$ connecting $v_1$ and $w$. Then $Aut(G)_{v_1}=Aut_G(G_1)$,  $Aut(G)_w=Aut_G(N_G(H))$, $ Aut(G)_{e_1}= Aut_G(G_1)\cap Aut_G(N_G(H))$. Thus 
$Aut(G)= Aut_G(G_1)\amalg_{ Aut_G(G_1)\cap Aut_G(N_G(H))}Aut_G(N_G(H)).$

	\end{proof}
	
	\begin{remark} 
	If $H$ is normal in $G$ then the splitting of the previous theorem is fictitious, since $Aut_G(N_G(H))=Aut(G)$. 
	
	\end{remark}
	
	\begin{corollary} Let $G=G_1\amalg_H G_2$ be a free pro-$p$ product of finite $p$-groups with amalgamation.  Then $Aut(G)$ is either 
$$Aut(G)= Aut_G(G_1)\amalg_{ Aut_G(G_1)\cap Aut_G(N_G(H))}Aut_G(N_G(H))\amalg_{Aut_G(G_2)\cap Aut_G(N_G(H))} Aut_G(G_2)$$ 
or
$$Aut(G)= Aut_G(G_1)\amalg_{ Aut_G(G_1)\cap Aut_G(N_G(H))}Aut_G(N_G(H)).$$	The second possibility occurs only if there exists an automorphism of $G$ that swaps $G_1$ and a conjugate of $G_2$.

	\end{corollary}
	
	\begin{proof} We need to prove only the last statement of the corollary, because finite groups are rigid. But this is the subject of \cite[Proposition 3.3]{BPZ}. 
	
	\end{proof}

Next  we refine Theorem \ref{amalgam} assuming that the amalgamated subgroup is malnormal. 

\medskip
Suppose first that $H$ is malnormal, say in $G_2$, and so $N_{G_2}(H)=H$. Then $$N_G(H)=N_{G_1}(H)$$ and so  the statement of Theorem \ref{amalgam} reads as follows:
$Aut(G)$ is either 
$$Aut(G)= Aut_G(G_1)\amalg_{ Aut_G(G_1)\cap Aut_G(N_{G_1}(H))}Aut_G(N_{G_1}(H))\amalg_{Aut_G(G_2)\cap Aut_G(N_{G_1}(H))} Aut_G(G_2)$$ 
or
$$Aut(G)= Aut_G(G_1)\amalg_{ Aut_G(G_1)\cap Aut_G(N_{G_1}(H))}Aut_G(N_{G_1}(H)),$$
where 1 can be replaced by 2.	

However if $H$ is malnormal in both $G_1$ and $G_2$ then $N_G(H)=H$. Moreover, $Aut_G(H)\leq Aut_G(G_1)\cap Aut_G(G_1)$ since $G_1^{g_1}\cap G_2^{g_2}=H$ if and only if $g_1\in G_1,g_2\in G_2$. On the other hand if for $\alpha\in Aut(G)$ one has $\alpha(G_i)=G_i$ for $i=1,2$,  then $\alpha(H)=\alpha(G_1\cap G_2)=\alpha(G_1)\cap \alpha(G_2)=G_1\cap G_2=H$. Thus  $Aut_G(H)= Aut_G(G_1)\cap Aut_G(G_1)$ and   the statement in this case reduces to the following

\begin{corollary}\label{malnormal amalgam}
Let $G=G_1\amalg_H G_2$ be a non-fictitious free pro-$p$ product with malnormal amalgamation. Suppose that $G_1$ and $G_2$ are rigid. Then $Aut(G)$ is either 
$$Aut(G)= Aut_G(G_1)\amalg_{Aut_G(H)}Aut_G(G_2)$$ 
or
$$Aut(G)= Aut_G(G_1),$$
where 1 can be replaced by 2.
\end{corollary} 

\bigskip

Note that since $H$ is malnormal, $G$ has trivial center and so $Inn(G)$ is naturally isomorphic to $G$, i.e. we can identify them in $Aut(G)$. Thus $G\cap Aut_G(G_i)=G_i$. Hence $G\cap Aut_G(H)=G\cap Aut_G(G_1)\cap Aut_G(G_1)=G_1\cap G_2=H$.

Thus we can deduce the following 

\begin{corollary} If $H$ is malnormal and $G_1,G_2$, then $$Out(G)=Aut_G(G_1)/G_1\amalg_{ Aut_G(H)/H} Aut_G(G_2)/G_2$$ or  $$Out(G)=Aut_{G}(G_1)/G_1$$
	
\end{corollary}

\bigskip
\begin{remark} All results for amalgamated free product above hold the same way also for abstract amalgamated free product and profinite amalgamated free products (see the next section for details).\end{remark}

\subsection{Applications}

\begin{proposition}\label{tree of cylindres acylindrical}
Let $G$ be a CSA pro-$p$ group and suppose that $\E$ consists of abelian subgroups. Let $\sim$ be the commutation relation. Then $T_c$ is 2-acylindrical.\end{proposition}

\begin{proof} As it is observed in Example (1) of Section 9.1 the cylinders are maximal abelian subgroups which are malnormal by hypothesis. Hence for any abelian group $A$ containing in the maximal $A_{max}$ the pro-$p$ subtree fixed by $A$ is the star of the vertex $v_{A_{max}}$ where $A_{max}$ is the cylinder.

\end{proof}

Proof of Theorem \ref{CSA}. Let    $G=\Pi_1(\G,\Gamma)$ be a splitting of $G$ as the fundamental group of a finite reduced graph of pro-$p$ groups with abelian edge groups. Let $T$ be a maximal subtree of $\Gamma$.  By \cite[Proposition 3.4]{CZ} $|\Gamma\setminus T|\leq d(G)$ and so  w.l.o.g. we may assume that $\Gamma=T$ is a tree. As $|E(T)|=|G(T)|-1$ it suffices to bound $|V(T)|$. As $G$ is SCA, every centralizer is abelian and so any abelian vertex group of  $(\G, T)$ has non-abelian neighbouring vertex group. Hence it suffices to bound the number of vertices of $T$ whose vertex groups are non-abelian.  By Proposition \ref{tree of cylindres acylindrical} the tree of cylinders $T_c$ of the standard pro-$p$ tree $S(G)$ is 2-acylindrical.  Recall that  $V(\Gamma)=V_0(T_c/G)$. Let $T_{cr}$ be the reduced tree obtained from $T_c$. By Corollary \ref{finiteness} $T_c/G$ is a finite and hence so is $T_{cr}/G$. As edge stabilizers of $T_c$ are abelian the set of vertices $V_{NA}$ of $T$ whose vertex stabilizers are not abelian embeds into $T_{cr}$, so it suffices to bound   $|V(T_{cr}/G)|$.   But this  follows from Theorem \ref{k-acylindrical accessibility}.

\begin{theorem} Let $G$ be a finitely generated  CSA pro-$p$ group. Suppose $G$ is indecomposable into a free pro-$p$ product and  $\E$ consists of  abelian subgroups. Let $T$ be a JSJ $\E$-tree. Then  $T_c$ is    $Aut(G)$-invariant JSJ $\E$-tree relative to not virtually cyclic abelian subgroups.
\end{theorem}

\begin{proof} By Theorem \ref{CSA} $G$ is $\E$-accessible and so by Corollary \ref{finiteness}  $T/G$ is finite.  Then by Corollary \ref{cofiniteness of tree of cylinders} $T_c/G$ is finite. 

 Let us first show that $T_c$ is universally elliptic.  Let $A$ be an edge stabilizer of $T_c$.  We note that in this case $G_Y$ is a maximal abelian subgroup. 
		Let $T'$ be a universally elliptic $\E$-tree for $G$ relative to non cyclic abelian subgroups. If $G_Y$ is not cyclic,  then it stabilizes a vertex of $T'$ since $T'$ is relative to non-cyclic abelian subgroups and therefore so is $A$.
		
		If $G_Y$ is cyclic, then it is commensurable with an edge stabilizer of $T$ (edge stabilizers of $T$ cannot be trivial since $G$ does not split as a free pro-$p$ product) and so  by Proposition \ref{ellipticity} is elliptic in $T'$ since $T$ is universally elliptic.  Thus $T_c$ is universally elliptic.

		 We need to show now that $T_c$ dominates $T'$. If $v\in V_0(T_c)$ then $v\in V(T)$ and so $G_v$ is elliptic in $T'$ since $T$ is a JSJ pro-$p$  $\E$-tree. 
		  
	Suppose $v\in V_1(T_c)$.  Then $v=T_Y$ and $G_v$ is the set-stabilizer $G_Y$ of the cylinder $Y$. But $G_Y$ in this case is  maximal abelian. So as in the preceding paragraph if it is  cyclic it is commensurable with the stabilizer of an edge in $T$ and so is elliptic. If it is abelian 
	non-cyclic subgroup of $G$, then it stabilizers a vertex of $T'$ by hypothesis.

\end{proof}

As limit pro-$p$ groups are CSA we deduce the following 

\begin{corollary} Let $G$ be a pro-$p$ limit group   and $\E$ be the set of  abelian subgroups of $G$. Suppose that $G$ is indecomposable into a free pro-$p$ product. Then the pro-$p$ tree of cylinders $T_c$ is    $Aut(G)$-invariant JSJ $\E$-tree relative to not virtually cyclic abelian subgroups. 

\end{corollary}

\begin{example} Let $H$ be an  infinite $FA$ pro-$p$ group and $C$ is a malnormal infinite cyclic subgroup of $H$. Let $G$ be a free pro-$p$ product $H\amalg_C C\times \Z_p$.  Note that there are even analytic pro-$p$ $CSA$ such groups $H$ (see \cite{MWZ}; also open pro-$p$ subgroups of $SL_2(\Z_p)$  are FA and CA). 

Then $G$ is the pro-$p$ fundamental group of the 
following graph of pro-$p$ groups.

$$\xymatrix{{\overset{H}{\bullet}}\ar[rrrr]^{C}&&&& {\overset{C}{\bullet}}\ar@(rd,ru)_{t}}$$
where $t$ is a generator of $\Z_p$.

The quotient graph $T_c/G$ then is the following graph of pro-$p$ groups.

$$\xymatrix{{\overset{H}{\bullet}}\ar[rrrr]^{C}&&&& {\overset{C\times\Z_p}{\bullet}}\ar[rrrr]^{C}&&&&{\overset{C}{\bullet}}}$$

Of course, the right edge of this graph of groups is fictitious and so  can be  collapsed.

\bigskip
We also can consider the extension of the centralizer above as HNN-extension $HNN(H,C,t)$. Then it is the fundamental pro-$p$ group of the following graph of pro-$p$ groups:

$$\xymatrix{{\overset{H}{\bullet}}\ar@(rd,ru)_{C}}$$

It follows that $T_c/G$ is the following graph of pro-$p$ groups:

$$\xymatrix{{\overset{H}{\bullet}}\ar[rrrr]^{C}&&&& {\overset{C\times\Z_p}{\bullet}}}$$

As $H$ is FA, $T_c$ is universally elliptic, and it is JSJ pro-$p$ tree relative to non-cyclic abelian subgroups. Note that this tree of groups coincides with the previous tree of groups after collapsing the right edge.

\end{example}

\section{Automorphisms}

In this section we extend Corollary \ref{malnormal amalgam}   to the pro-$\C$ case, where $\C$ is a class of finite groups closed for subgroups, quotients and extensions, assuming that the amalgamated subgroup is finite.

A pro-$\C$ group will be called an $OE$-group if  whenever it acts   on a  pro-$\C$ tree  with finite edge stabilizers, then it fixes a vertex. 

We shall reformulate here Proposition 3.4 and Corollary 3.5 \cite{BPZ} for our purpose adapting the proofs.

\begin{proposition} \label{compart1} Let $G = G_{1}\amalg_{H}\,G_2$  be a proper   pro-$\C$  amalgamated free products of  pro-$\C$ groups  
with finite amalgamation.  Suppose  $G_1$ is an OE-group.   Then  for any  automorphism $\psi\in Aut(G)$ there exists an inner automorphism $\tau$ such that  $\tau\psi(G_1) =G_{1}$ (up to possibly interchanging $G_1$ and $G_2$ in $G$ if $G_1\cong G_2$).
	\end{proposition}

\begin{proof}  By \cite[Example 6.3.1]{R}   $\psi(G_1)$ is conjugate into $G_1$ or $G_2$, say $\psi(G_1)\leqslant {G_{1}}^{w}$ for some $w\in G$. Thus denoting by    $\tau_{w^{-1}}$ the inner automorphism of conjugation by $w^{-1}$, if necessary,  we get $\tau_{w^{-1}}\psi(G_1) \leqslant G_1$.  But symmetrically  $G_1=\psi^{-1}(G_1^w)$ is in $G_1^g$ for some $g\in G$ and so $G_1=G_1^g$ and hence $\psi(G_1)={G_{1}}^{w}$.  Thus $\tau_{w^{-1}}\psi(G_1) = G_1$. 
\end{proof}

\begin{corollary} \label{2b}  Suppose in addition that   $G_2$ is a pro-$\C$ $OE$-group. Then for any automorphism $\psi\in Aut(G)$ there exists an inner automorphism $\tau$ such that $\tau\psi(H) = H$, $\tau\psi(G_1) =G_{1}$ and $\tau\psi(G_2) ={G_{2}}^{g}$ for some $g\in G$ (up to possibly interchanging $G_1$ and $G_2$ in $G$ if $G_1\cong G_2$).
	
\end{corollary}

\begin{proof} By  Proposition \ref{compart1} we may assume that $\psi(G_1)=G_1$.  Similarly as in the proof of  Proposition \ref{compart1} $\psi(G_2)$ is conjugate into $G_1$ or $G_2$.  But $ G_1$ and $ G_2$ are not conjugate,  since otherwise by \cite[Corollary 7.1.5 (b)]{R} or \cite[Corollary 3.13]{ZM-88} $G_1$ is conjugate to $H$ and so $G$ is fictitious. Thus  $\psi(G_2)\leqslant {G_{2}}^{b}$ for some $b\in G$ (cf. \cite[Example 6.3.1]{R}).  Then symmetrically  $\psi^{-1}(G_2^b)$ is in $G_2^g$ for some $g\in G$.  But $\psi^{-1}\psi(G_2)=G_2$ and so we deduce as in the previous proposition that 
	  $\psi(G_2)=G_2^b$.

	By \cite[Corollary 7.1.5 (b)]{R} or \cite[Corollary 3.13]{ZM-88} we have: $$\psi(H) = G_{1} \cap  G_{2}^{b} \leqslant H^{b_1} \,(b_1\in G_1)\eqno{(1)}$$ Similarly we have $\psi^{-1}(H)\leqslant H^{g_1}$ for some $g_1 \in G_1.$ Since $H$ is finite we have $\psi(H)=H^{b_1},\, b_1\in G_1.$ Thus we get $\tau_{{b_1}^{-1}} \circ \psi(H)=H$.
	
\end{proof}

\begin{corollary}\label{malnormal amalgamation} Suppose $H$ is malnormal. Then $\tau$ can be chosen such that in addition $\tau\psi(G_2)=G_2$ (up to possibly interchanging $G_1$ and $G_2$ in $G$ if $G_1\cong G_2$).
\end{corollary} 

\begin{proof} By Corollary \ref{2b} we may assume that $\psi(H) = H$, $\psi(G_1) =G_{1}$ and $\psi(G_2) ={G_{2}}^{g}$ for some $g\in G$. Then Equation (1) reads as $$H=\psi(H) = G_{1} \cap  G_{2}^{g} \leqslant H^{b_1} \,(b_1\in G_1, g\in G).$$ From malnormality of $H$ we deduce $b_1\in H$. So denoting by $e$ the edge of the standard profinite tree $S(G)$ stabilized by $H$ and by $v_1,v_2$ its vertices stabilized by $G_1$ and $G_2$ respectively we have that $H$ fixes $[v_1,gv_2]$ and so this geodesic contains only one edge $e$, i.e. $gv_2=v_2$ and $g\in G_2$. 
	
	\end{proof}

 We denote by  $Aut_{G_i}(H)$ and $Aut_{\widehat G_i}(H)$ the subgroups that leave  $H$ invariant  and by $\overline{Aut}_{G_i}(H)$, $\overline{Aut}_{\widehat G_i}(H)$ their images in $Aut(H)$. 

\begin{theorem}\label{Aut(G) for amalgam} Let $G=G_1\amalg_H G_2$ be a free pro-$\C$ product of OE pro-$C$ groups $G_1\not\cong G_2$ with finite malnormal amalgamation.  Then $$Aut(G)=Aut_{G}(G_1)\amalg_{Aut_G(H)} Aut_{G}(G_2).$$
\end{theorem}

\begin{proof} Since $G_1\not\cong G_2$ an automorphism interchanging $G_1$ and $G_2$ does not exist.  So by  Corollary \ref{malnormal amalgamation}   for any $\alpha\in Aut(G)$ one has $\alpha(G_1)=G_1^{g_1},  \alpha(G_2)=G_2^{g_2}$ and $\alpha(H)=H^{g_3}$ for some $g_1,g_2,g_3\in G$.  But $G_1^{g_1},  \G_2^{g_2}$ stabilize unique vertices $v_1, v_2$ in the standard pro-$\C$ tree $S=S(G)$ for 
$G=G_1\amalg_H G_2$ and so $\alpha(H)=G_1^{g_1}\cap  \G_2^{g_2}$ fixes the geodesic $[v_1,v_2]$.   Since $H$ is malnormal in $G$ this geodesic has only one edge $e'$ that connects  $v_1$ and $v_2$. Moreover if for $\alpha\in Aut(G)$ one has $\alpha(G_i)=G_i$ for $i=1,2$,  then $\alpha(H)=\alpha(G_1\cap G_2)=\alpha(G_1)\cap \alpha(G_2)=G_1\cap G_2=H$ and so $\overline{Aut}_{G}(G_1)\cap \overline{Aut}_{G}(G_2)=H$.

 Thus if $w_1,w_2,$ are vertices stabilized by $G_1$ and $G_2$ respectively, we can define the action of $Aut(G)$ on $S$ by setting $\alpha(w_1)=v_1$, $\alpha(w_2)=v_2$ and $\alpha(e_0)=e'$ where $e_0$ is an edge connecting $w_1$ and $w_2$ (stabilized by $H$).  We need to show that this action is continuous.  

Since $d_1,d_2$ are continuous, it suffices to show that the action of $Aut(G)$ on $E(S)$ is continuous.  Since $Inn(G)$ acts  transitively on $E(S)$, so is $Aut(G)$. If $e$ is the edge stabilized by $H$ then the stabilizer of it in $Aut(G)$ is  $Aut_G(H)$.  Hence $Aut(G)$-space $E(S)$ is naturally equivariantly homeomorphic to $Aut(G)$-space $Aut(G)/Aut_G(H)$. This can be checked by factoring both spaces modulo open subgroups $U$ of $Aut(G)$ containing $Aut_G(H)$ and then taking the inverse limit of the natural equivariant bijections. This proves that the action of $Aut(G)$ on $E(S)$ is continuous as needed.

Finally observing as in the previous section that the vertex stabilizer of $w_i$ in $Aut(G)$ is $Aut_G(G_i)$ for $i=1,2$ we get the result. 

\end{proof}

Note that since $H$ is malnormal then $H$ is self-normalized and so $N_{G_i}(H)=H$. Since $H$ is malnormal, $G$ has trivial center and so $Inn(G)$ is naturally isomorphic to $G$, i.e. we can identify them in $Aut(G)$. Thus $G\cap Aut_G(G_i)=G_i$.   Thus  we can deduce the following 

\begin{corollary}   $$Out(G)=Aut_{G}(G_1)/G_1\amalg_{Aut_G(H)/H} Aut_{G}(G_2)/G_2.$$
	
\end{corollary}

\bigskip

Observe that all results above hold the same way also for abstract amalgamated free product. Hence we deduce the following 

\begin{corollary} Suppose $G_1\not\cong G_2$ are finite and $H$ is malnormal. Let $G^{abs}=G_1*_H G_2$ be a free amalgamated product. Then $Aut(\widehat G^{abs})=\widehat{Aut(G^{abs})}$ and $Out(\widehat G^{abs})=\widehat{Out(G^{abs})}$.
\end{corollary}

\begin{remark}  The congruence problem for the automorphism group of a finitely generated residually finite group asks whether $\widehat{Aut(G)}\longrightarrow Aut(\widehat G)$ is injective and if it is not to describe the kernel (the congruence kernel).  The previous corollary solves it for free product of finite groups with malnormal amalgamation, i.e. this homomorphism is injective. 

In general, the subgroup theorem for normal subgroups (cf. \cite{Za}) gives the structure of the congruence kernel for $Aut(G)$ as in Theorem \ref{Aut(G) for amalgam} \end{remark}

\section{Compatibility}

Recall  that two $G$-trees $T_1$ and $T_2$ are compatible if they have a common refinement. In other words, there exists a tree $T$ with collapse maps $T\longrightarrow T_i$, $i=1,2$.

\begin{definition} (Universally compatible). A pro-$p$ tree $T$ is universally compatible (over $\E$ relative to $\K$) if it is compatible with every $(\E,\K)$-tree. \end{definition}

\medskip
In particular, this means that any pro-$p$ $(\E,\K)$-tree $T'$ can be obtained from $T$ by refining and collapsing. When $T'$ is a one-edge splitting, either $T'$ coincides with the splitting associated to one of the edges of $T/G$, or one can obtain $T'/G$ by refining $T/G$ at some vertex $v$ using an one-edge splitting of $G_v$ relative to the incident edge groups, and collapsing all the original edges of $T/G$.

\begin{definition} We say that a $G$-tree $T$ is maximal for domination if any $G$-tree dominating it belongs to the same deformation space as $T$.\end{definition}

\begin{definition} (Compatibility JSJ deformation space). If,  among deformation spaces containing a universally compatible tree, there is one which is maximal for domination, we denote it by $D_{co}$ and  called it the compatibility JSJ deformation space of $G$ over $\E$ relative to $\K$.  \end{definition}

\begin{theorem}\label{co}  $D_{co}$ exists and unique.

\end{theorem}

\begin{proof} 
Existence. Compatible $(\E,\K)$-trees form naturally an inverse system. Hence inverse limit of them is the $(\E,\K)$-tree maximal for domination. 

Uniqueness.
Let $T_1$, $T_2$ be  maximal for domination universally compatible $(\E,\K)$-trees.  As they are compatible there exists a tree $T$ with collapse maps $T\longrightarrow T_i$, $i=1,2$. As they are maximal for domination they belong to the same deformation space. 

\end{proof}

The uniqueness of an inverse limit  shows that a maximal for domination universally compatible $(\E,\K)$-tree is unique. We state it as a 

\begin{corollary}\label{T_{co}} A maximal for domination universally compatible $(\E,\K)$-tree $T_{co}$ is unique. It is fixed under any automorphism of $G$ that leaves $D_{co}$ invariant. In particular, any pro-$p$  subgroup of $Aut(G)$ that preserves $\E$ and $\K$  acts on $T_{co}$.  

\end{corollary}

Note that if $T_{co}/G$ is finite or $G$ is finitely generated then $T_{co}$ may be refined to a JSJ $(\E,\K)$-tree (Proposition  \ref{JSJ refinement}).

\medskip
Clearly, a universally compatible tree is universally elliptic.  This implies that $D_{co}$ is dominated by $D_{JSJ}$. Also note that, if $T$ is universally compatible and $J$ is an edge stabilizer in an arbitrary $(\E,\K)$-tree, then $J$ is elliptic in $T$ (i.e. any $(\E,\K)$-tree is elliptic with respect to $T$).

\medskip
Examples in the next subsection show that in general it  is  difficult to find non-trivial example of a splitting having only one reduced  JSJ tree. We give below several examples where $T_{co}$ does not exists  contrasting corresponding   abstract examples, where such a tree exists. In the pro-$p$ case, for such examples one should look at $k$-acylindrical splittings. 

We shall need the following 

\begin{proposition}\label{cofinite} Let $T_1$, $T_2$ be compatible $G$-trees with $|T_i/G|< \infty$ for $i=1,2$. Then there exists a finite refinement $T$ of $T_i$ such that $|T/G|<\infty$. \end{proposition}

\begin{proof} Let $S$ be a common refinement of $T_1$ and $T_2$. Then the corresponding  collapse maps $f_i:S\longrightarrow T_i$ induce the collapse maps $\bar f_i:S/G\longrightarrow T_i/G$ on the quotinets. Since $T_i/G$ is finite the corresponding equivalence relations are clopen, i.e. there are  clopen subgraphs $\Gamma_i$  of $S/G$  with finitely many connected componente such that $\bar f_i$ is the collapse of these finitely many connected components of $\Gamma_i$. Then $\Gamma_1\cap \Gamma_2$ is a clopen subgraph of $S/G$  with finitely many connected component. Thus the collapse of these connected components  $S/G\longrightarrow \Gamma$ gives a finite graph $\Gamma$. Let  $\Delta$ be a preimage of $\Gamma_1\cap \Gamma_2$ in $S$. Then collapse of all connected components of $\Delta$ is the $G$-equivariant collapse map $S\longrightarrow T$ with $T/G=\Gamma$. Thus $T$ is the desired refinement.

\end{proof}

\subsection{Examples}

\subsubsection{Free pro-$p$ groups}

When $\E$ is $Aut(G)$-invariant, the compatibility JSJ tree $T_{co}$ is $Aut(G)$-invariant (cf. Corollary \ref{T_{co}}).  This sometimes forces it to be trivial.

\begin{proposition}\label{free pro-p} If $G$ is a free pro-$p$ group and $\E$ is $Aut(G)$-invariant, then $T_{co}$ is trivial. \end{proposition}

\begin{proof} If $F=F(x_1\, \ldots x_n)$ is a free pro-$p$ group on the basis $x_1, \ldots x_n$ then the automorphisms that send $x_i$ to $x_ix_j$ ($i\neq j$) and  leave the rest of generators give the set of $F$-trees such that the only compatible tree for all of them is the trivial one. 

\end{proof}

\subsubsection{Free pro-$p$ products}

 Here we assume that $G$ is finitely generated. In this case a Grushko decomposition always exists and $Aut(G)$ is profinite and virtually pro-$p$ (cf. \cite[Corollary 4.4.4 and Lemma 4.5.5]{RZ-10}.

\begin{proposition}\label{free product} Let $\E$ consist only of the trivial group. Let $$G=G_1\amalg \cdots ···\amalg G_p\amalg F_q$$ be a Grushko decomposition ($G_i$ is non-trivial, not $\Z_p$, and freely indecomposable, $F_q$ is free pro-$p$ of rank $q$).
Then $T_{co}$ is trivial.

\end{proposition}

\begin{proof} 
 Note that  $T_{co}$ is, of course, trivial also if $G$ is freely indecomposable or free of rank $\geq 2$ (cf. Proposition \ref{free pro-p}). 
 
 Write $G=A\amalg B$, where $A,B$ are non-trivial and  $A=G_1$. Let $T$ be a standard pro-$p$ tree of this decomposition. Given a non-trivial $g\in G$, the subgroup $g Ag^{-1}\amalg B$ is the image of $A\amalg B$ by an automorphism $\alpha_g$ that sends $A$ to $gAg^{-1}$ and fixes $B$. Let $S$ be the stadard pro-$p$ tree of this decomposition. 
 If there exists a common refinement of $T$ and $S$, then  by Proposition \ref{cofinite}  there exists a common refinement $D$ of $T$ and $S$ with $D/G$ finite. Then collapsing fictitious edges we may assume that $D$ is reduced. But reduced $G$-tree whose vertex groups are two free factors up to conjugation is a pro-$p$ tree with one edge and two vertices up to translation.  This contradicts that $D/G$ is a refinement.

\end{proof}

This contrasts with the abstract situation where $T_{co}$ exists for Grusko decompositions whose tree have one edge only up to translation (see \cite[Section 11.2]{GL}).

\begin{example} Let $G\cong \Z_p$ and $\E$ consists of the trivial group. The unique reduced $G$-tree is the Cayley graph $\Gamma(G)$ with respect to a generator. However $Aut(G)=Z_p\times C_{p-1}$ for $p$ odd and $\Z_2\times C_2$ for $p=2$ (the group of units). This group can not act on $\Gamma(G)$ because any automorphism fixes trivial element of $G$  and  sends the generator to another unit of $\Z_p$. \end{example}

\subsubsection{ Free product with amalgamation} 

\begin{example}\label{example amalgamation} Let $G=G_1\amalg_H G_2$ be a non-fictitious free pro-$p$ product of FA pro-$p$ groups with  amalgamated subgroup $H$. By Proposition \ref{normalizer} $N_G(H)=N_{G_1}(H)\amalg_HN_{G_2}(H)$. Assume that $H$ is not self-normalized in each factor (this is always the case if $G_i$ are finite for example).  
For any $g\in  N_G(H)\setminus G_2G_1$ putting $G'_2=G_2^g$ we get the splitting $G=G_1\amalg_H G'_2$ 
(cf. \cite[Proposition 5.1]{BPZ}). To see this just observe that denoting by $H_1\leq G_1, H_2\leq G_2$ the identified copies of $H$ we can change identification $h_1=h_2$ by $h_1=h_2^g\in G_2^g$ to obtain the equivalent presentation for $G$. We show that the standard pro-$p$ tree $T'$ for the splitting $G=G_1\amalg_H G'_2$  is not compatible with $T$. As in the case of  free pro-$p$ products if there exists a common refinement, then  by Proposition \ref{cofinite}  there exists a common refinement $D$ of $T$ and $T'$ with $D/G$ finite. Then collapsing fictitious edges we may assume that $D$ is reduced. But reduced $G$-tree whose vertex groups are two free factors (of an amalgamated free pro-$p$ product) up to conjugation is a pro-$p$ tree with one edge and two vertices up to translation.  This contradicts that $D/G$ is a refinement.

Next we show that if on the contrary   $H$ is malnormal in, say, $G_2$,  then $T_{co}$ is actually the standard pro-$p$ tree $S(G)$ for $G=G_1\amalg_H G_2$. Indeed, as in preceding paragraph one shows that any non-trivial reduced pro-$p$ $G$-tree  $T$ can be idetified with standard pro-$p$ tree of the splitting $G$ as a free product with amalgamation $G'_1\amalg_{H'}G'_2$, where $G'_i$ are conjugates of $G_1$ and $G_2$, and so we may assume w.l.o.g that $G_1=G'_1$. As  $G'_2=G_2^g$ for some $g\in G$,  $G'_2$ fixes a vertex  $gG_2$ in $S(G)$. But $S(G)$ is 2-acylidrical, because $H$ is malnormal in $G_2$. Therefore,   if $g\notin G_2G_1$ then $G_1\cap G'_2=1$ and so $G\cong G_1\amalg G_2'$, a contradiction. Hence $g=g_1g_2\in G_1G_2$ and so $T$ is $G$-isomorphic to $S(G)$. Indeed, the $G$-isomorphism $S(G)\longrightarrow T$ is  defined by sending $1H$ to    to the edge that connects the vertex stabilized by $G_1$ and the vertex  stabilized by $g_1G'_2g_1^{-1}$.

\end{example}

  \subsubsection{Poincar\'e duality pro-$p$ groups} 
  
  Let $G$ be a Poincar\'e duality group of dimension $n$. Recall that a pro-$p$ group is $VPC_n$ (resp. $VPC_{\leq n}$) if it is virtually polycyclic of Hirsch length $n$ (resp $\leq n$). We  consider the family $\E$ consisting of malnormal $VPC_{\leq n-1}$-subgroups. By \cite[Theorem 1.2]{CAZ} there exists a unique JSJ $\E$-tree for $G$. So by Corollary \ref{T_{co}} it coincides with $T_{co}$.
  
  Examples of JSJ-decompositions of $PD^3$ pro-$p$ groups can be obtained by the pro-$p$ completion of   abstract JSJ-decomposition of some 3-manifolds (see \cite{wilkes17}). The pro-$p$ completion of $PD^n$ groups in general were studied in  \cite{KZpoincare,W, HK,HKL}.
  
    Note that by \cite[Lemma 1.6 and Theorem 1.7] {WZ}, 
    if $PD^n$ pro-$p$ group $G$ splits over a virtually solvable subgroup $H$,
     then 
     $H$ is 
     $VPC_{n-1}$ 
     (solvable pro-$p$ groups of finite cohomological dimension are polycyclic). 
    This  is a pro-$p$ version of the Kropholler \cite[Theorem A2]{K}. The abstract analogue of existence and uniqueness of the JSJ-decomposition for $PD^n$ group is the Kropholler theorem  \cite[Theorem A2]{K} that gives also information  on vertex groups of a JSJ-splitting.

\section{Compatibility and free constructions of free pro-$p$ groups}

\subsection{ Free pro-$p$ groups}

We start this section by showing that any splitting of a free pro-$p$ group admits a  free pro-$p$ product refinement. The next proposition is a pro-$p$ version of the result of Shenitzer \cite{She}.

\begin{definition} Let $F$ be a free pro-$p$ group and $\overline A\leq H_1(F,\F_p)=F/\Phi(F)$ be a subgroup of  its Frattini quotient. Let $B_A$ be a basis of $\overline A$ and $s:F/\Phi(F)\longrightarrow F$ a continuous section. Then $A=\langle s(B_A)\rangle$ is a free pro-$p$ group on $s(B_A)$ and it is a free factor of $F$ that we shall call a lift of $\overline A$ to $F$. 

\end{definition}

\begin{proposition} \label{free}  Let $G=F_1\amalg_C F_2$ be a free amalgamated product  of free pro-$p$ groups and suppose $G$ is free pro-$p$. Then $C$ splits as a free pro-$p$ product $C=C_1\amalg C'_2=C'_1\amalg C_2=C'_1\amalg C'_2\amalg C'$ such that $C_i$ is a free factor of  $F_i$ and $C'_i$ is a (possibly trivial) free factor of $C_i$.
	
\end{proposition}

\begin{proof} Denote by  $cor_i:H_1(C,\F_p)\longrightarrow H_1(F_i,\F_p)$ the corestriction, $i=1,2$. Since $G$ is free, $H_2(G,\F_p)=0$ and so one deduces from the Mayer-Vietoris sequence (cf. \cite[Proposition 9.2.13]{RZ-10}) that  the natural map $$cor_1\oplus cor_2:H_1(C,\F_p)\longrightarrow H_1(F_1,\F_p)\oplus H_1(F_2,\F_p)$$ is injective. Then  one has   $ker(cor_1)\cap ker(cor_2)=0$. Therefore $(cor_2)_{|ker(cor_1)}$ and $(cor_1)_{|ker(cor_2)}$  are  injections. It follows that  $$H_1(C,\F_p)=ker(cor_1)\oplus \overline C_2\  {\rm and}\  H_1(C,\F_p)=ker(cor_2)\oplus \overline C_1,$$ where $\overline C_2\cong im(cor_1)$ is the complement of $ker(cor_1)$ in $H_1(C, \F_p)$ containing $ker(cor_2)$ and $\overline C_1\cong im(cor_2)$ is the complement of $ker(cor_2)$ in $H_1(C, \F_p)$ containing $ker(cor_1)$.   Let $ C_i$ be  a lift of $\overline C_i$ to $C$ and $C'_i$ be a lift  of $ker(cor_i)$ to $C_i$. Note that $C'_i$ is a free factor of $C_i$ and   $C_i$ is a free factor of $F_i$.  
Then $$C=C_1\amalg C'_2=C'_1\amalg C_2=C'_1\amalg C'_2\amalg C'$$ as 
$$H_1(C,\F_p)=\overline C_1\oplus ker(cor_2)=\overline C_2\oplus ker(cor_1)=ker(cor_1)\oplus ker(cor_2)\oplus \overline C'$$ for some $\overline C'$.

\end{proof}

\begin{corollary}\label{cyclic amalgamation of free} Let $G=F_1\amalg_C F_2$ be a free pro-$p$ product  of free pro-$p$ groups $F_1,F_2$ with cyclic amalgamation $C$. Then  $G$ is free pro-$p$ if and only if  $C$  is a free factor of  $F_1$ or $F_2$.

\end{corollary}

\begin{proof} 'only if '. In this case  either $C_1=C$ or $C_2=C$ in Proposition \ref{free} and so the result follows from this proposition.

\medskip
'if'. If $C$ is a free factor of, say, $F_1$, i.e. $F_1=L\amalg C$,  then $G= L\amalg F_2$ is a free pro-$p$ product of free pro-$p$ groups and therefore is free pro-$p$. 
 \end{proof}

The next proposition is the pro-$p$ analogue of \cite[Theorem 1]{Swa}.

\begin{proposition} \label{HNNfree}  Let $G=HNN(F,C, t)$ be a pro-$p$ HNN-extension of  a free pro-$p$ group and suppose $G$ is free pro-$p$. Then $F$ and $C$ split as a free pro-$p$ product $F=F_1\amalg F_2$,  $C=C_1\amalg C_2$ such that $C_1$ is a free factor of  $F_1$ and $C_2^t$ is a free factor of $F_2$ (where one of $C_1,C_2$ can be trivial).
	
\end{proposition}

\begin{proof} Let $cor:H_1(C,\F_p)\longrightarrow H_1(F,\F_p)$, $cor_t:H_1(C^t,\F_p)\longrightarrow H_1(F,\F_p)$ be the corestrictions. Since $G$ is free, $H_2(G,\F_p)=0$ and so one deduces from Mayer-Vietoris sequence (\cite[Proposition 9.4.2]{RZ-10}) that the natural map $f:H_1(C,\F_p)\longrightarrow H_1(F,\F_p)$ given by $cor(c)-cor_t(c^t)$ is injective and so $ker(cor)\cap ker(cor_t\circ \tau)=0$, where $\tau$ is the conjugation by $t$. Hence $H_1(C,\F_p)=\overline C_1\oplus \overline C_2$, where $\overline C_2$ is the kernel of   the corestriction $cor:H_1(C,\F_p)\longrightarrow H_1(F,\F_p)$  and $\overline C_1$ is its complement. As $C$ is free pro-$p$,  $C=C_1\amalg C_2$ is a free pro-$p$ product of lifts  of  $\overline C_i$ in $C$.

Let $H_1(F,\F_p)=\overline F_1\oplus \overline F_2$ be a direct sum decomposition with $cor(\overline C_1)\leq \overline F_1$ and $cor(\overline C_2^t)\leq \overline F_2$. Then one can choose lifting  of $\overline F_i$ containing $C_i$. So $F=F_1\amalg F_2$ is a free pro-$p$ product   and $C_1$ is a free factor of $F_1$ and $C_2^t$ is a free factor of 
$F_2$   (possibly trivial). 
\end{proof}

\begin{corollary}\label{cyclic free} Let $G=HNN(F,C, t)$ be a pro-$p$ HNN-extension of  a free pro-$p$ group $F$ with cyclic associated subgroup $C$. Then $G$ is free pro-$p$ if and only if either $C$ or $C^t$ is a free factor of $F$ and $C\Phi(F)\neq C^t\Phi(F)$.
\end{corollary} 

\begin{proof} 'Only if'. In this case  either $C_1=C$ or $C_2=C$ in Proposition \ref{HNNfree} say $C=C_1$. Then by Proposition  \ref{HNNfree} $C_1$ is a free factor of $F$.  As was already mentioned in the proof of Proposition \ref{HNNfree} the natural map $f:H_1(C,\F_p)\longrightarrow H_1(F,\F_p)$ given by $cor(c)-cor_t(c^t)$ is injective and so $C\Phi(F)\neq C^t\Phi(F)$ as required.

\medskip
'if'. Suppose say $C$ is a free factor of $F$, i.e. $F$ decomposes as a free pro-$p$ product $F=C\amalg L$.

Case 1. Both $C, C^t \not\leq \Phi(F)$. Since $C\Phi(F)\neq C^t\Phi(F)$, $L$ can be chosen such that $C^t\leq L$.  In this case  $G=L\amalg \langle t\rangle$ and  the result follows.

\medskip

Case 2.  One of $C, C^t$ is contained in $\Phi(F)$, say $C^t\leq \Phi(F)$.  Then $G=\langle L, t\rangle$. Denote by $c$ a generator of $C$, so that $c^t=c'\in \Phi(F)$ and let $X$ be a basis of $L$. Since   $[t,c]\not\in \Phi(F)$ the presentation $G=\langle C, X,t\mid c^t=c'\rangle$ just identifies the free generator $c^t c'^{-1}=[t,c]$ with the trivial element. So $G$ is a free pro-$p$ group  $L\amalg \langle t\rangle$ again.  

\end{proof}

The next example shows that \cite[Theorem 18]{Henry} as well as \cite[Corollary 1.5]{Toi} does not hold in the pro-$p$ case.

\begin{example}\label{cyclicHNNfree} Let $G=HNN(F(a,b),\langle b\rangle, t)$ with $b^t=a^p[a,b]$.

1.  $G$ is a two generated free pro-$p$ group generated by $a$ and $t$. This follows from Corollary \ref{cyclic free}. 

2. $G=HNN(F(a,b),\langle b\rangle, t)$ is proper. Indeed, if $G$ were not proper, then the image of $F(a,b)$ in $F(a,t)$  would be cyclic, equal to $\langle a\rangle$ and so $b$ would be a super natural power  $a^{\alpha}$ of $a$. Putting it into the relation we get  $(a^{-\alpha}) ^ta^p[a,a^{\alpha}]=(a^{-\alpha})^ta^p=1$. It follows then that $G$ is the Bamslag-Solitar pro-$p$ group (see Section \ref{Baumslag-Solitar}) that have a non-trivial normal subgroup, contradicting that $G$ is non-abelian free pro-$p$. 

3. $F(a,b)$ does not split relative to   $\langle a^p[a,b]\rangle$. Indeed, if $\langle a^p[a,b]\rangle$ is contained in a free factor $C$ of $F$, then $F/\langle\langle a^p[a,b]\rangle\rangle \cong C/\langle a^p[a,b]\rangle\amalg \Z_p$. But $F/\langle\langle a^p[a,b]\rangle\rangle=\Z_p\rtimes \Z_p$ is a soluble Demushkin group that is not cyclic and  torsion free, a contradiction. 
\end{example}

\begin{definition} Suppose $\E$ consist of the trivial subgroup and $\K$ a set of subgroups of a second countable pro-$p$ group $G$. If $T$ is an $(\E,\K)$-tree, and there exists a continuous section $V(T)/G\longrightarrow V(T)$ then by \cite[Theorem 9.6.1]{R} $G$ splits as a free pro-$p$ product  of its vertex stabilizers. In this case we shall say that $G$ splits as a free pro-$p$ product relative to $\K$.\end{definition}

\begin{proposition}\label{star general} Let $F$ be a free pro-$p$ group written as the fundamental group $\Pi_1(\curlyF, \Gamma)$ of a finite star  of pro-$p$ groups, i.e. $d_0(\Gamma)=v_0$ is one vertex.    Then each edge group $F(e)$ splits as a free pro-$p$ product $F(e)=F_0(e)\amalg F_1(e)$ (with one of them possibly trivial as well as possibly $F_0(e)=F_1(e)$) such that  each pending vertex group $F(v)$ splits as a free pro-$p$ product relative to  $F_1(e)$ for $e$ incident to $v$ and   $F(v_0)$ splits as a free pro-$p$ product relative to  $F_0(e)$ for all $e\in E(\Gamma)$ and to $F_1(e)$ for all $e$ that are loops. 
$$
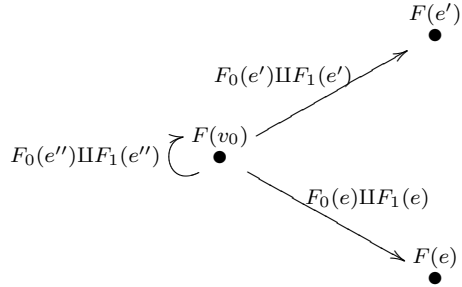
\xymatrix{&&&& {\overset{F(e') }{\bullet}}\\
&& {\overset{ F(v_0)}{\bullet}}\ar[rru]^{F_0(e')\amalg F_1(e')} \ar[rrd]^{F_0(e)\amalg F_1(e)}\ar@(ld,lu)^{ F_0(e'')\amalg F_1(e'')}\\  
&&&&{\overset{F(e)}{\bullet}}}$$

\end{proposition}

\begin{proof} Let $cor_i:H_1(F(e),\F_p)\longrightarrow H_1(G(d_i(e)),\F_p)$ be the corestriction.  Since $F$ is free, $H_2(F,\F_p)=0$ and so the natural map $$f:\bigoplus_{e\in E(\Gamma)} H_1(F(e),\F_p)\longrightarrow \bigoplus_{v\in V(\Gamma)} H_1(F(v),\F_p), f=cor_1(g)-cor_0(g)$$ from the Mayer-Vietoris sequence (cf. \cite[Section 9.4]{R})  is injective;  so $$ker(cor_0)\cap ker(cor_1)=0$$ for each $e$.   Then $H_1(F(e), \F_p)$ splits as a direct sum $H_1(F(e),\F_p)=\overline C_0(e)\oplus \overline C_1(e)$ with $\overline C_1(e)=ker(cor_0)$ and $\overline C_0(e)$ its complement. Since   $ker(cor_0)\cap ker(cor_1)=0$ for each $e$, $cor_1$ restricts injectively on $\overline C_1(e)$. Lift $\overline C_i(e)$ to $F_i(e)\leq F(e)$. Then $F(e)=F_0(e)\amalg F_1(e)$.

Consider the composite  $$\alpha:\bigoplus_{e\in E(\Gamma)} H_1(F(e),\F_p)\longrightarrow \bigoplus_{v\in V(\Gamma)} H_1(F(v),\F_p)\longrightarrow H_1(F(v_0),\F_p)$$ of $f$ with the natural projection of $\bigoplus_{v\in V(\Gamma)} H_1(F(v),\F_p)$ onto $H_1(F(v_0),\F_p)$.

Let $E(\Gamma)=E_n\cup E_l$ be a union, where $E_n$ is the set of not loops and $E_l$ is the set of loops.  
Note that $\bigoplus_{e\in E_n} H_1(F_1(e),\F_p)$ is the kernel of $\alpha$ and so $\alpha$  is injective on $$\bigoplus_{e\in E(\Gamma)} H_1(F_0(e),\F_p)\oplus \bigoplus_{e\in E_l} H_1(F_1(e),\F_p).$$

   Consider a free pro-$p$ product $$G= \coprod _{e\in E_n} F_0(e)\amalg \coprod_{e\in E_l} F(e)$$ and 
note that $$H_1(G,\F_p)= \bigoplus_{e\in E_n} H_1(F_0(e),\F_p)\oplus \bigoplus_{e\in E_l} H_1(F_1(e),\F_p).$$  
    Define $\beta_e:F_0(e)\longrightarrow F$ to be the natural embedding. Let $\beta:G\longrightarrow F(v_0)$  be a homomorphism induced by $\beta_{e}, e\in E(\Gamma)$,  and $\beta_e(f)^{t_e}$ for $f\in F_1(e)$ if $e\in E_l$. Then the following diagram 

$$\xymatrix{G\ar[rr]^{\beta}\ar[d]&& F(v_0)\ar[d]\\
  H_1(G,\F_p)\ar[rr]^{\alpha}&&H_1(F(v_0),\F_p)}$$
 is commutative. As $\alpha$ is split injective and $G$, $F(v_0)$ are free pro-$p$, $\beta$ is spit injective and so we may consider $G$ as a free factor of $F(v_0)$, unless $\beta$ is surjective. This finishes the proof for $F(v_0)$.
 
 Similarly $F_1(e)$ embeds as a free factor to its pending vertex group $F(v)$ if $e$ is not a loop. This finishes the proof.

\end{proof}

\begin{remark}\label{triviality of free factors} Note that in Proposition \ref{star general}  $F_i(e)$ is trivial if and only if $\partial_i(F(e))\leq \Phi(F_{d_i(e)})$.  

\end{remark}

\begin{corollary}\label{star} Let $F$ be a free pro-$p$ group written as the fundamental group $\Pi_1(\curlyF, \Gamma)$ of a star  of pro-$p$ groups, i.e. $d_0(\Gamma)=v_0$ is one vertex. Suppose all edge groups $F(e)$ are infinite cyclic.   Then $F(v_0)$ splits as a free pro-$p$ product relative to $\{\partial_i(G(e)\mid d_i(e)=v_0, \partial_i(G(e)\not\leq \Phi(G(v)), i=0,1\}$.

\end{corollary}

\begin{proof} Follows from Proposition \ref{star general} and Remark \ref{triviality of free factors}

\end{proof}

\begin{theorem}\label{free pro-p general} Let $F$ be a non-abelian free pro-$p$ group written as the fundamental group $\Pi_1(\curlyF, \Gamma)$ of a finite injective graph of pro-$p$ groups.  Then for a vertex $v$ of $\Gamma$  each incident edge group $F(e)$ splits as a free pro-$p$ product $F(e)=F_0(e)\amalg F_1(e)$  and   $F(v)$ splits as a free pro-$p$ product relative to  $F_0(e)$ for all $e\in E(\Gamma)$ with $v=d_0(e)$ and to $F_1(e)$ for all $e$ with $d_1(e)=v$.  (Here $F_0(e)$ and $F_1(e)$ can coincide or be trivial).

\end{theorem}

\begin{proof} 
 Let $\Sigma$ be the star of $v$ and $(\curlyF, \Sigma)$ the subgraph of pro-$p$ groups restricted to $\Sigma$. By changing orientation if necessary we may assume that $d_0(e)=v$ for every edge $e$ incident to $v$. Then by  Proposition \ref{star general}  each incident edge group $F(e)$ to $v$  splits as a free pro-$p$ product $F(e)=F_0(e)\amalg F_1(e)$  and   $F(v_0)$ splits as a free pro-$p$ product relative to  $F_0(e)$ for all $e\in E(\Gamma)$ and to $F_1(e)$ for all $e$ that are loops.  

$$
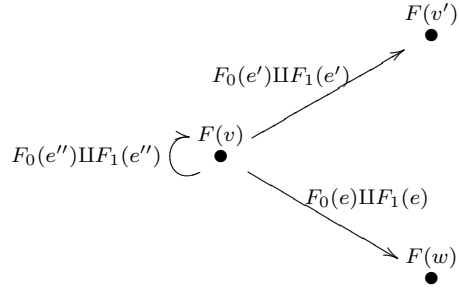
\xymatrix{&&&& {\overset{F(v') }{\bullet}}\\
&& {\overset{ F(v)}{\bullet}}\ar[rru]^{F_0(e')\amalg F_1(e')} \ar[rrd]^{F_0(e)\amalg F_1(e)}\ar@(ld,lu)^{ F_0(e'')\amalg F_1(e'')}\\  
&&&&{\overset{F(w)}{\bullet}}}$$

\end{proof}

\begin{corollary} Any decomposition of a free pro-$p$ group as the fundamental group of a finite graph of pro-$p$ groups $\Pi_1(\curlyF, \Gamma)$ is compatible with some nontrivial free pro-$p$ product decomposition.

\end{corollary}

\begin{proof} We first insert artificial edges for each $F_i(e)$ with $e$ incident to $v$ and then   using Theorem \ref{free pro-p general} and its notation, we can 
 blowup  the vertex $v$ and vertex group $F(v)$ into a free pro-$p$ product.

$$
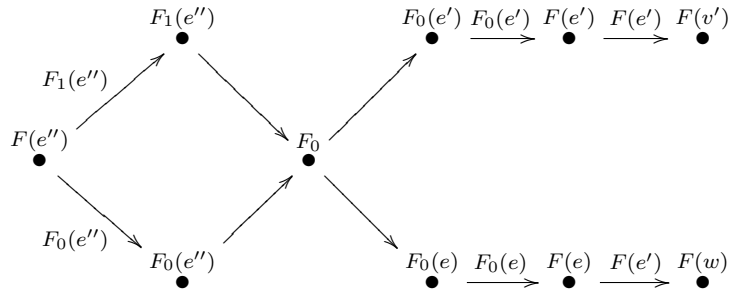
\xymatrix{&{\overset{F_1(e'')}{\bullet}}\ar[rd]&&{\overset{F_0(e')}{\bullet}}\ar[r]^{ F_0(e')}& {\overset{F(e') }{\bullet}}\ar[r]^{F(e')}&{\overset{F(v') }{\bullet}}\\
{\overset{F(e'')}{\bullet}}\ar[rd]_{ F_0(e'')}\ar[ru]^{ F_1(e'')}&& {\overset{ F_0}{\bullet}}\ar[ru] \ar[rd]\\  
&{\overset{F_0(e'')}{\bullet}}\ar[ru]&&{\overset{F_0(e)}{\bullet}}\ar[r]^{ F_0(e)}&{\overset{F(e)}{\bullet}}\ar[r]^{F(e')}&{\overset{F(w) }{\bullet}}}$$
Here $F_0$ is the remaining free factor that can be trivial.





Then the collapse of new edges gives back the graph of groups $(\curlyF, \Gamma)$ and the collapse of edges with non-trivial edge groups gives a free pro-$p$ product decomposition.

\end{proof}

Using Corollary \ref{star} we deduce

\begin{corollary}\label{cor:free pro-p} Let $F$ be a non-abelian free pro-$p$ group written as the fundamental group $\Pi_1(\curlyF, \Gamma)$ of a finite injective graph of pro-$p$ groups with  cyclic edge groups. Then  for any vertex $v$ of $\Gamma$  the vertex group $G(v)$ splits as a free pro-$p$ product relative to $\{\partial_i(G(e)\mid d_i(e)=v, \partial_i(G(e)\not\leq \Phi(G(v))\}$.

\end{corollary}

Example \ref{cyclicHNNfree} shows that to get a  splitting of a vertex group relative to incident edge groups we need to put the hypothesis that underline graph has no cycles. We do it in the next

\begin{corollary}\label{free pro-p over tree} Let $F$ be a non-abelian free pro-$p$ group written as the fundamental group $\Pi_1(\curlyF, \Gamma)$ of a finite injective tree of pro-$p$ groups with  cyclic edge groups. Then   there exists a vertex $v$ of $\Gamma$ such that the vertex group $G(v)$ splits as a free pro-$p$ product relative to the incident edge groups of $v$.

\end{corollary}

\begin{proof} 

We use induction on $\Gamma$. 
If $\Gamma$ has one edge only the result is the subject of Corollary  \ref{cyclic amalgamation of free}.

Note that  $G$ splits  as a free amalgamated pro-$p$ product  $G=G_1\amalg_{G(e)} G(w)$ over an edge group $G(e)$, where  $e$ is an edge incident to a pending vertex $w$ and we assume w.l.o.g that $d_1(e)=w$. Let $\Gamma_1=\Gamma\setminus \{e,w\}$. By the induction hypothesis there exists $v\in V(\Gamma_1)$  such that $G(v)$ splits as a free pro-$p$ product relative to incident edge groups of $v$ in $\Gamma_1$.  If $v$ is not incident to $e$ we are done. 

Suppose $e$ is incident to $v$. If $\partial_1(G(e))\not \leq \Phi(G(w)$ then  $\partial_1(G(e))$ is a free factor of $\Phi(G(w))$ and so $w$ is the needed vertex.

Since $H_2(F)=0$ the map $$f:\bigoplus_{e\in E(\Gamma)} H_1(F(e),\F_p)\longrightarrow \bigoplus_{v\in V(\Gamma)} H_1(F(v),\F_p), f=cor_1(g)-cor_0(g)$$ from the Mayer-Vietoris sequence (cf. \cite[Section 9.4]{R})  is injective and so if $\partial_1(G(e)) \leq \Phi(G(w)$ then $\partial_0(G(e)) \not\leq \Phi(G(w)$. Then the result follows from Corollary \ref{cor:free pro-p}.

\end{proof}

\subsection{Cyclic amalgamations or HNN-extensions of free pro-$p$ groups decomposable as free pro-$p$ products}

\medskip
If we drop the assumption of freeness of $G$ but assume that $C$ is cyclic in Corollary \ref{cyclic amalgamation of free}  we get the same conclusion. 

\begin{proposition}\label{pro-p} Let $G=F_1\amalg_C F_2$ be a free amalgamated pro-$p$ product of free pro-$p$ groups over infinite cyclic subgroup $C$.  Then $G$    splits as a non-trivial free pro-$p$ product if and only if  $C$  belongs to a free factor of $F_1$  or $F_2$.
\end{proposition}

\begin{proof}
Suppose $G$    splits as a non-trivial free pro-$p$ product $G=\coprod_{i=1}^n H_i$.  If $G$ is  free pro-$p$ then the result follows from Corollary \ref{cyclic amalgamation of free}. Otherwise,  at least one $H_i$, say $H_1$, is not free pro-$p$ and so  by \cite[2.7]{ZM-88} $H_1$  intersects a conjugate of $C$ non-trivially (since if $H_1$ acts on  the standard pro-$p$ tree $S(G)$ of $G=F_1\amalg_C F_2$ with trivial edge stabilizers it is free pro-$p$). Thus we may assume that $H_1\cap C\neq 1$.  But the centralizer of any element of $H_1$ is in $H_1$ (see \cite[Theorem 9.1.12]{RZ-00}), so $C\leq H_1$. Note that $F_i\cap H_1$ is a free factor of $F_i$, $i=0,1$ by \cite[Corollary 9.1.10]{RZ-00}. Write $F_i=(F_i\cap H_1)\amalg L_i$, $i=1,2$. If  $L_1=L_2=1$, then $G\leq H_1$, a contradiction. Thus $C$ is in a non-trivial free factor $F_i\cap H_1$ of $F_i$ for $i=1$ or 2.  
	
Conversely,	if $C$ is in a non-trivial free factor $L_1$ of $F_1$, i.e. $F_1=K_1\amalg L_1$ then   $G=K_1\amalg  (L_1\amalg_C F_2)$ is a free pro-$p$ product.   
\end{proof}

An element $g$ of a pro-$p$ group $G$ is called a test element if  any endomorphism $\alpha$ of $G$ with $\alpha(g)=g$ is an automorphism. It is proved by Snop\c{c}e and Tanushevski \cite{ST} that an element of a free pro-$p$ group is test if and only if it does not belong to any proper free factor. Thus we can deduce the following

\begin{corollary} Let $F_1, F_2$ be free pro-$p$ groups and $c_1\in F_1, c_2\in F_2$ be test elements. Then 
$G=F_1\amalg_{c_1=c_2} F_2$ does not split as a free pro-$p$ product.
\end{corollary}

\begin{proposition}\label{Hnn pro-p} Let $G=HNN(F,C,t)$ be an $HNN$-extension of a free pro-$p$ group over infinite cyclic subgroup $C$.  If $G$    splits as a non-trivial free pro-$p$ product, then  $C$ or $C^t$ belongs to a free factor of $F$.
\end{proposition}

\begin{proof} Suppose $G$    splits as a non-trivial free pro-$p$ product $G=\coprod_{i=1}^n H_i$.  If $G$ is  free pro-$p$ then the result follows from Corollary \ref{cyclic free}. Thus we may assume that some of $H_i$, say $H_1$ is not free. Let $S(G)$ be a standard pro-$p$ tree for $G=HNN(F,C,t)$. If $H_1$ intersects all conjugates of $C$ trivially then by \cite[2.7]{ZM-88} $H_1$ is free pro-$p$, so   $H_1$  intersects a conjugate of $C$ non-trivially. Therefore  we may assume that $H_1\cap C\neq 1$ and prove that statement for $C$.  But the centralizer of any element of $H_1$ is in $H_1$ (see \cite[Theorem 9.1.12]{RZ-00}, so $C\leq H_1$. 
Note that $F\cap H_1$ is a free factor of $F$,  by \cite[Corollary 9.1.10]{RZ-00}. Write $F=(F\cap H_1)\amalg L$. If  $L=1$, then $F\leq H_1$ and $C=F\cap F^{t}\leq H_1\cap H_1^t\neq 1$ implying that $t\leq H_1$ by \cite[Theorem 9.1.12]{RZ-10}. So  $G\leq H_1$, a contradiction. Thus $C$ is in a non-trivial free factor $F\cap H_1$ of $F$. 

 \end{proof}

\subsection{ Virtually free pro-$p$ products}

We shall extend the  situation above to virtually free pro-$p$ products of pro-$p$ groups. To do this we need the  generalization of \cite[Theorem 3.6]{WEZ-16}. We shall make before the following

\begin{remark}\label{homo natural mod tilde} Let $G=\Pi_1(\G,\Gamma,v_0)$ be the pro-$p$ fundamental group of a finite graph of finite
$p$-groups, and let $U$ be an open and normal subgroup of $G$. Then 
$\tU=\langle\,U\cap {}\G(v)^g\mid g\in G,\ v\in
V(\Gamma)\,\rangle$ is a closed normal subgroup of $G$. By
\cite[Prop.~1.10]{RZ-14}, one has a natural decomposition of
$G/\tU$ as the pro-$p$ fundamental group $G/\tU=\Pi_1(\G_{U},\Gamma, v_0)$
of a finite graph of finite $p$-groups $(\G_{U},\Gamma)$, where the vertex and edge groups satisfy
$\G_{U}(x)=\G(x)\tU/\tU$, $x\in V(\Gamma)\sqcup E(\Gamma)$. Thus we have a
morphism $\eta\colon (\G,\Gamma)\longrightarrow (\G_{U},\Gamma)$
of graphs of groups such that the induced homomorphism on the
pro-$p$ fundamental groups coincides with the canonical projection
$\varphi_{U}\colon G\longrightarrow G/\tU$.

\end{remark}

\begin{theorem} \label{relative virtfreeprod} Let $G$ be a finitely generated pro-$p$ group containing an open normal subgroup
$H$ and a $G$-invariant collection $\K$ of subgroups of $H$ such that $H$ has a non-trivial decomposition as a free pro-$p$ product $H=F\amalg H_1\amalg\cdots\amalg H_s,$
 where $F$ is a free pro-$p$ subgroup of rank $r$  and all subgroups of $\K$ are conjugate into $H_i$ for some $i$. Suppose  $\{ H_i^h\mid h\in H, i=1, \ldots,k\}$ is $G$-invariant.   
Then $G$ splits as the pro-$p$ fundamental group of
a finite graph $(\G, \Gamma)$ of pro-$p$ groups relative to $\K\cup \{ H_i^h\mid h\in H, i=1, \ldots,k\}$ with finite edge stabilizers of order not exceeding $[G:H]$.
\end{theorem}

\begin{proof}
 By hypothesis,  $s+r\geq 2$. By
construction, one has for all $g\in G$ and for all
$i\in\{1,\ldots,s\}$ that $H_i^g$ is a free factor of $H$. 

\noindent
{\bf Step 1:} Let $\B$ be a basis of neighbourhoods of $1_G\in G$ consisting of open normal
subgroups $U$ of $G$ which are contained in $H$ with $H_i\not\subseteq U$ for every $i=1,\ldots s$. 
For $U\in\B$ put
\begin{equation}
\label{eq:tU}
\widetilde U=\langle\,U\cap {}H_i^g\mid g\in G,\ 1\leq i\leq s\,\rangle=
\langle\,U\cap {}H_i^h\mid h\in H,\ 1\leq i\leq s\,\rangle.
\end{equation}
Then $\widetilde U$ is a closed normal subgroup of $H$, and
\begin{equation}
\label{eq:defHtU}
H/\widetilde U= F\amalg H_1\widetilde U/\widetilde U\amalg\cdots\amalg H_s\widetilde U/\widetilde U
\end{equation}
(cf. \cite[Prop.~1.18]{Me}). The group $G/\tU$ contains the
open normal subgroup $H/\widetilde U$ which is a finitely generated,
virtually free pro-$p$ group (since $U/\widetilde U$ is free pro-$p$ by
 in \cite[Theorem 2.6]{ZM-88}), and thus $G/\widetilde U$ is a
finitely generated, virtually free pro-$p$ group.

\noindent
{\bf Step 2:} By \cite[Theorem 1.1]{HZ-13}, $G/\tU$  is isomorphic to the pro-$p$ fundamental group
$\Pi_1(\G_U,\Gamma_U, v_U)$ of a finite graph
of finite $p$-groups. Although neither the finite graph $\Gamma_U$ nor the finite graph of finite 
$p$-groups $\G_U$ are uniquely determined by $U$ (resp. $\tU$), the index $U$ in the notation
shall express that both these objects are depending on $U$. 
By Remark \ref{reduction} 
we may assume that $(\G_U,\Gamma_U)$ is
reduced. Hence from now on we may assume that for every $U\in \B$ the vertex groups of
$G/\tU=\Pi_1(\G_U,\Gamma_U, v_U)$ are representatives of the
$G/\tU$-conjugacy classes of maximal finite subgroups.
Note that by \cite[Proposition 3.2 (a)]{WEZ-16}, one has $\G_U(e)\cap H/\tilde U=1$ and so $|\G_U(e)|\leq [G:H]$. Moreover, as all subgroups $K$ of $H$ from $\K$ are elliptic in $H$, $K\tU/\tU$ are elliptic (and so are finite) in $H/\tU$. As virtually elliptic group is elliptic, for any subgroup $K$ of $G$ from $\K$ the quotient group $K\tU/\tU$ is elliptic (and so is finite) in $G/\tU$.

\noindent {\bf Step 3:} By Remark \ref{homo natural mod tilde},
for $V\subseteq U$ both open and normal in
$G$ the decomposition $G/\tV=\Pi_1(\G_V,\Gamma_V, v_V)$ gives rise
to a natural decomposition of $G/\tU$ as the fundamental group
$G/\tU=\Pi_1(\G_{V,U},\Gamma_V, v_V)$ of a finite graph of pro-$p$ groups
$(\G_{V,U},\Gamma_V)$, where vertex and edge groups are natural images of the vertex and edge groups of $G/\tV=\Pi_1(\G_V,\Gamma_V, v_V)$ in $G/\widetilde U$. Moreover, by \cite[Lemma 3.4]{WEZ-16},
if $(\G_V,\Gamma_V)$ is reduced, then $(\G_{V,U},\Gamma_V)$ is reduced. Thus in this case one has 
a morphism
$\eta\colon (\G_{V},\Gamma_V)\longrightarrow (\G_{V,U},\Gamma_V)$
of reduced graph of groups such that the induced homomorphism on
the pro-$p$ fundamental groups coincides with the canonical
projection $\varphi_{UV}\colon G/\tV\longrightarrow G/\tU$.

\noindent {\bf Step 4:}  By \cite[Proposition~3.2]{WEZ-16}, the
number $|V(\Gamma_U)|+|E(\Gamma_U)|$ is bounded by $4(r+s)-1$. So we have only finitely many graphs $\Gamma_U$ up to isomorphism, when $U$ runs. It follows that there is a finite graph $\Gamma$ such that $\Gamma_U$ is isomorphic to $\Gamma$ for infinitely many $U$'s.
Therefore, by passing to a cofinal system $\curlyC$ of $\B$ if
necessary, we may assume that $\Gamma_U=\Gamma$ for each $U\in
\curlyC$. Then, by \cite[Corollary 3.3]{WEZ-16},  the number of
isomorphism classes of finite reduced graphs of finite $p$-groups
$(\G^\prime_U,\Gamma)$ which are based on $\Gamma$ satisfying
$G/\tU\simeq\Pi_1(\G^\prime,\Gamma,v_0)$ is finite. Suppose that
$\Omega_U$ is a set containing a copy of every such isomorphism
class. For $V\in\curlyC$, $V\subseteq U$, one has a map
$\omega_{V,U}\colon \Omega_V\to\Omega_U$ (cf. Step~3). Hence
$\Omega=\varprojlim_{U\in\curlyC}\Omega_U$ is non-empty. Let
$(\G^\prime_U,\Gamma)_{U\in\curlyC}\in\Omega$. Then
$(\G^\prime,\Gamma)$ given by $\G^\prime(x)=\varprojlim
\G^\prime_U(x)$ if $x$ is either a vertex or an edge of $\Gamma$,
is a reduced finite graph of finitely generated pro-$p$ groups
satisfying $G\simeq\Pi_1(\G^\prime,\Gamma,v_0)$. By
\cite[Proposition~3.2 (a)]{WEZ-16}, $\G^\prime(e)$ is finite and in fact $|\G^\prime(e)|\leq [G:H]$ for
every edge $e$ of $\Gamma$, since this is the case in every $G_U$. Moreover, as $K/\tU$ is elliptic for each subgroup $K\in \K$ of $G$, $K$ is elliptic in $\Pi_1(\G^\prime,\Gamma)$.  This yields the claim.
\end{proof}

In the proof  of Theorem \ref{relative virtfreeprod} we needed  that $\{H_i^h\mid h\in H, i=1,\ldots, k\}$ is $G$-invariant. If  $H=F\amalg H_1\amalg\cdots\amalg H_s$ is the Grushko decomposition relative to $\K$ (see the last paragraph of Section 7.2), then  this is automatic.  Thus we can state the following 

\begin{theorem}
\label{thm:virtfreeprod}
Let $G$ be a finitely generated pro-$p$ group containing an open normal subgroup
$H$ and a $G$-invariant collection $\K$ of subgroups of $H$ such that $H$ has a non-trivial Grushko decomposition as a free pro-$p$ product $H=F\amalg H_1\amalg\cdots\amalg H_s,$ relative to $\K$ (i.e. $F$ is a free pro-$p$ subgroup of rank $r$, and the $H_i$ are
$\amalg$-indecomposable relative to $\K$   and all subgroups of $\K$ are conjugate into $H_i$ for some $i$).
Then $G$ splits as the pro-$p$ fundamental group of
a finite graph $(\G, \Gamma)$ of pro-$p$ groups relative to $\K\cup \{ H_i^h\mid h\in H, i=1, \ldots,k\}$ with finite edge stabilizers whose order does not exceed $[G:H]$.
\end{theorem}

\begin{proof}  Since
$H_i$ is indecomposable relative to $\K$, we deduce from Proposition \ref{relative Grushko}  
that the indecomposable non-free subgroup
$H_i^g$ of $H$ equals $H_j^h$ for some $j\in\{1,\ldots,s\}$. Thus
$\{H_i^g\mid g\in G,\ 1\leq i\leq s\,\}=\{H_i^h\mid h\in H,\ 1\leq
i\leq s\,\}$, i.e. $\{H_i^h\mid h\in H,\ 1\leq
i\leq s\,\}$ is $G$-invariant and the result follows from Theore \ref{relative virtfreeprod}.

\end{proof}

Suppose $\K$ is closed for open normal subgroups. Then we can drop the assumption of normality for $H$.

\begin{corollary} Let $G$ be a finitely generated pro-$p$ group containing an open  subgroup
$H$ and a $G$-invariant collection $\K$ of subgroups of $G$ closed for taking open normal subgroups,  such that $H$ has a non-trivial Grushko decomposition as a free pro-$p$ product $H=F\amalg H_1\amalg\cdots\amalg H_s,$ relative to $\K$.
Then $G$ splits as the pro-$p$ fundamental group of
a finite graph $(\G, \Gamma)$ of pro-$p$ groups relative to $\K\cup \{ H_i^h\mid h\in H, i=1, \ldots,k\}$ with finite edge groups of order not exceeding $[G:H]$.

\end{corollary}

\begin{proof} By replacing $H$ by the core of $H$ in $G$ and applying the
Kurosh subgroup theorem for open subgroups (cf. \cite[Thm.~9.1.9]{RZ-10}),
we may assume that $H$ is normal in $G$. Refining the free decomposition if necessary
and collecting free factors isomorphic to $\Z_p$ we obtain a Grushko free decomposition relative to $\K$
\begin{equation}
\label{eq:Hdeco}
H=F\amalg H_1\amalg\cdots\amalg H_s,
\end{equation}
i.e. a decomposition where $F$ is a free subgroup of rank $r$, and the $H_i$ are
$\amalg$-indecomposable relatively to $\K$ finitely generated subgroups which are not
isomorphic to $\Z_p$. Then the result follows from Theorem \ref{thm:virtfreeprod}.

\end{proof}

\subsection{Virtually free pro-$p$ groups}

We start with  free pro-$p$ products with cyclic amalgamation.

\begin{theorem}\label{over cyclic and free product} Let $G=G_1\amalg_C G_2$ be a (non-fictitious) free product with infinite cyclic amalgamation. Suppose $G$ is finitely generated virtually free pro-$p$. Then either $G_1$ or $G_2$ splits as a free pro-$p$ product relative to $C$.  Moreover, $C$ is a free factor of $G_1$ or $G_2$.

\end{theorem}  

\begin{proof} Let $F$ be a maximal open normal free pro-$p$ subgroup of $G$. Then by Corollary \ref{open subgroup} $F$ is the fundamental pro-$p$ group of a finite graph of pro-$p$ groups $(\curlyF, \Delta)$  whose vertex and edge stabilizers are intersections of conjugates of $G(v)$ and  $\G(e)$ with $F$. Moreover, by Remark \ref{reduction} we may assume that $(\curlyF, \Delta)$ is reduced. In particular $F$ splits as a free pro-$p$ product with amalgamation $F_1\amalg_{C^g\cap F} F_2$ or pro-$p$ HNN-extension $HNN(F_1, C^g\cap F,t)$ over an edge group of $(\curlyF, \Delta)$, where $F_1$  is the fundamental pro-$p$ group of a subgraph $(\curlyF, \Delta_1)$ of pro-$p$ groups restricted to a subgraph $\Delta_1$ of $\Delta$ (and a similar situation happens with $F_2$).  Assume w.l.o.g that this edge group is $F\cap C$ over the edge $e$ that connects $G_1\cap F$ with $G_2\cap F$. Then by Corollary \ref{free}  $C\cap F$ is a free factor of $F_1$ or $F_2$ (in the case of amalgamation) or by Corollary \ref{HNNfree} either $C\cap F$ or $C^t\cap F$ is a free factor of $F_1$ (in the case of HNN-extension). Suppose w.l.o.g. $C\cap F$ is a free factor of $F_1$ and note that $G_1\cap F\leq F_1$. Then $C\cap F\not\leq \Phi(F\cap G_1)$ and so $G_1\cap F$ can not be cyclic as our edge $e$ is not fictitious. Note also that the edge groups of  $G_1\cap F$ are $G_1$-invariant, so by  Theorem \ref{relative virtfreeprod}   $G_1$ is the fundamental group of a finite graph $(\G_1,\Gamma_1)$ of pro-$p$ groups relative to $C$ with finite edge groups.

Let $w$ be a vertex of $\Gamma_1$ such that  $C\leq G_1(w)$. If $G_1(w)\neq C$ we can apply the same argument to $G_1(w)\amalg_C  G_2$. Then repeating the process, i.e. refining the graph of groups $(\G_1,\Gamma_1)$ this way,  we arrive eventually at situation when $C$ is equal to a vertex group of $(\G_1,\Gamma_1)$ and so is a free factor of its fundamental group (note that the process stops after finitely many steps because $F\cap G_1$ is finitely generated and $F\cap G_1(w)$ is a free factor of it). Thus $C$ is a free factor of $G_1$ or $G_2$. 
\end{proof}

Now we deal with HNN-extension.

\begin{theorem}\label{HNN over cyclic and free product} Let $G=HNN(G_1,C, t)$ be an HNN-extension with infinite cyclic associated subgroup.  Suppose $G$ is finitely generated having open  free pro-$p$ group $H$. Then   $C$ or $C^t$ is a free factor of $G_1$.  
\end{theorem}  

\begin{proof} Note that $\langle G_1, G_1^{t^{-1}}\rangle=G_1\amalg_C G_1^{t^{-1}}$. This subgroup satisfies the premises of Theorem \ref{over cyclic and free product} and so $C$ is a free factor of $G_1$ or $G_1^{t^{-1}}$. In the latter case $C^t$ is a free factor of $G_1$.
\end{proof}

\subsection{General case}

The next lemma  uses \cite{GJ} to show that being a free factor is preserved by the inverse limits.

\begin{lemma}\label{free factor} Let $G$ be a finitely generated pro-$p$ group and $H$ a finitely generated subgroup of $G$. Let $G=\varprojlim_i G_i$ be a decomposition of $G$ a surjective inverse limit and $H_i$ is the image of $H$ in $G_i$. If $H_i$ is a free factor of $G_i$ for all $i$ then so is $H$ of $G$.
\end{lemma}

\begin{proof} We use \cite[Proposition 3.1]{GJ}.  Let $K$ be a finite  $p$-group. Since $G$ is finitely generated, $Hom(G,K)$ is finite and therefore  so is the number $h(G,H,\gamma, K)$ of extensions of  a homomorphism $\gamma:H\longrightarrow K$ to a homomorphism $G\longrightarrow K$. Let  $H_i$ be the image of $H$ in $H_i$. As any homomorphism $G\longrightarrow K$ factors through some $G_i$ (and consequently $\gamma$ factors through some $\gamma_i:H_i\rightarrow K$) we achieve the equalities $$|Hom(G,K)|=|Hom(G_i,K)|\  {\rm and}\   h(G,H,\gamma, K)=h(G_i,H_i, \gamma_i,K)$$  for some $i$. As $H_i$ is a free factor of $G_i$, by \cite[Proposition 3.1]{GJ}    $$h(G_i,H_i, \gamma_i,K)=h(G_i,H_i, \kappa_i,K)$$ for every $\kappa_i\in Hom(H_i, K)$. Hence $h(G,H, \gamma,K)=h(G,H, \kappa,K)$ for every $\kappa\in Hom(H,K)$.   Since $K$ was chosen arbitrary, by \cite[Proposition 3.1]{GJ} $H$ is a free factor of $G$.
\end{proof}

\begin{theorem}\label{one edge} Let $G=G_1\amalg_{C} G_2$ be a finitely generated  free pro-$p$ product with cyclic amalgamation. Then $G$  splits as a free pro-$p$ product if and only if $C$  belongs to a free factor of $G_1$  or $G_2$. 
\end{theorem}

\begin{proof} If  $G$ is free pro-$p$, the result is the subject  of Proposition \ref{pro-p}. Otherwise, consider the Grushko decomposition    $G=\coprod_{i=1}^n H_i\amalg F$ as a free pro-$p$ product  of non-cyclic $\amalg$-indecomposable pro-$p$ groups $H_i$ and a free pro-$p$ group $F$.  

Case 1. At least one of the conjugates of  $H_i$, say $H_1$ intersects a conjugate of $C$.  Then w.l.o.g. we may assume that $H_1\cap C\neq 1$.  Since the centralizer of any element of $H_1$ is in $H_1$ (cf. \cite[Theorem 9.1.12]{RZ-00}),  $C\leq H_1$. Note that $G_i\cap H_1$ is a free factor of $G_i$ by the pro-$p$ version of the Kurosh subgroup theorem (see Theorem \ref{KST}). Write $G_i=(G_i\cap H_1)\amalg L_i$. If  $L_1=L_2=1$, then $G\leq H_1$, a contradiction. Therefore  $L_1$ or $L_2$ is non-trivial and therefore $C$ is in a non-trivial free factor of $G_1$ or $G_2$.

\medskip
Case 2.	
	Suppose now that all $H_i$s intersect conjugates of $C$ trivially. Then $C$ is infinite cyclic and acts freely on the standard pro-$p$ tree $S=S(G)$ corresponding to these free pro-$p$ product.   Since $H_i$ acts on $S$ with trivial edge stabilizers, by \cite[Theorem 9.6.1]{R} $H_i$ spits as a free pro-$p$ product of its vertex stabilizers and a free pro-$p$ group, so from the assumption on $H_i$ one deduces that all  $H_i$ are conjugate into $G_1$ or $G_2$. Let  $U$ be an open normal subgroup of $G$ and $\widetilde U$ the normal subgroup generated by the intersections $U\cap H_i^g, i=1,\ldots, n$. Then  $G_U:=G/\widetilde U=\coprod_{i=1}^n H_i\widetilde U/\widetilde U$ and for some $U$ the image of $C$ in $G_U$ is infinite cyclic and acts freely on $S/\widetilde U$. As every $H_i$ is conjugate into $G_1$ or $G_2$, so are the intersections $U\cap H_i^g, i=1,\ldots, n$, therefore denoting the image of $C$ in $G_U$ by the same letter we can write $G_U=G_{1U}\amalg_C G_{2U}$, where $G_{iU}=G_i\widetilde U/\widetilde U$.    Since $G_U$ is virtually free pro-$p$ ($U/\widetilde U$ is free pro-$p$), by Theorem \ref{over cyclic and free product} $C$ is a free factor of $G_{iU}$   for $i=1$ or 2, say $C$ is a free factor of $G_{1U}$.  
	Since $G=\varprojlim_U G_U$,  by Lemma \ref{free factor} $C$ is a free factor of $G_1$. 

\medskip	
	Conversely,	if $C$ is in the free factor $L_i$ of $G_i$, i.e. $G_i=K_i\amalg L_i$ then   $G=K_i\amalg  (L_i\amalg_C G_2)$ is a free pro-$p$ product.
\end{proof}

 \begin{corollary} Let $G=G_1\amalg_{C} G_2$ be a finitely generated  free pro-$p$ product  with  cyclic amalgamation $C$. If $C\leq \Phi(G_1)\cap \Phi(G_2)$  then $G$ does not split as a free pro-$p$ product.\end{corollary} 

\begin{corollary} Let $G_1\amalg_{C} G_2$  be a finitely generated   free pro-$p$ product  with infinite virtually cyclic amalgamation $C$. If $G$ splits as an amalgamated free pro-$p$ product or pro-$p$ HNN-extension over a finite $p$-group $K$, then    $G_1$ or $G_2$ splits as an amalgamated free pro-$p$ product or pro-$p$ HNN-extension over a finite $p$-group relative to $C$. 
	
\end{corollary}

 \begin{proof}  Let $H$ be a maximal open normal subgroup of $G$ intersecting $K$ trivially such that $C\cap H$  is infinite cyclic. Then by  Corollary \ref{open subgroup} $H =H_1\amalg H_2$ is a non-trivial free pro-$p$ product. As $H_1\cap H_2^g=1$ for any $g\in G$ (see \cite[Theorem 9.1.12]{RZ-10}) we deduce that either $G_1\cap H$ or $G_2\cap H$, say $G_1\cap H$ is not contained in the conjugate of $H_i$ for $i=1,2$ and so by the Kurosh subgroup theorem (Theorem \ref{KST}) $G_1\cap H$ splits as a non-trivial free pro-$p$ product. Now considering $(G_1\cap H)\amalg_{(C\cap H)} (G_2\cap H)$ we can apply Theorem \ref{one edge} to deduce that $C\cap H$ is a free factor of $G_1\cap H$ and clearly $\{(H\cap C)^h\mid h\in H\cap G_1\}$ is $G_1$-invariant (since $G_1$ leaves invariant the set of conjugates of $C$).  Then by Theorem \ref{relative virtfreeprod} $G_1$ splits as a finite graph of pro-$p$ groups with finite edge groups relative to $C$.    
  
\end{proof}

We finish with considering an HNN-extension.

\begin{theorem}\label{one loop} Let $G=HNN(G_1,C, t)$  be a finitely generated  HNN-extension with cyclic associated subgroup $C$.
If $G$  splits as a free pro-$p$ product $G=H_1\amalg H_2$ then $C$ or $C^t$ belongs to a free factor of $G_1$. 
\end{theorem}

\begin{proof} Note that $\langle G_1, G_1^{t^{-1}}\rangle=G_1\amalg_C G_1^{t^{-1}}$. The normal closure of $G_1$ in $G$ can not belong to a free factor $H_i$ of  $G$ since a free factor is sef-normalized (see \cite[Theorem 9.1.12]{RZ-10}). Hence there exists $n\in \Z$ such that $\langle G_1^{t^n},G_1^{t^{n+1}}\rangle$  not in a free factor $H_i$ of $G$ (indeed $\langle G_1^n\mid n\in\Z\rangle$ is dense in the normal closure $<<G_1>>$ of $G_1$ in $G$). Hence $\langle G_1, G_1^{t^{-1}}\rangle=G_1\amalg_C G_1^{t^{-1}}$ does not belong to a free factor $H_i$ of $G$. 

Thus by the pro-$p$ version of the Kurosh subgroup theorem (see Theorem \ref{KST}) this subgroup satisfies the premises of Theorem \ref{over cyclic and free product} and so  $C$ is a free factor of $G_1$ or $G_1^{t^{-1}}$. In the latter case $C^t$ is a free factor of $G_1$.
\end{proof} 

\begin{corollary} Let $G=HNN(G_1,C,t)$ be a finitely generated  HNN-extension with cyclic associated subgroup $C$. If $C, C^t\leq \Phi(G_1)$  then $G$ does not split as a free pro-$p$ product.\end{corollary}

	

  


\bigskip
 \href{https://www.mat.unb.br/pessoa/152/Pavel-Zalesski}
{Pavel A. Zalesski}, 

Universidade de Brasília, 

Departamento de Matemática, 

70910-900 Brasília DF, Brazil. 

{\it Email address}: \href{mailto:pz@mat.unb.br}{pz@mat.unb.br}.

\end{document}